\documentclass[12pt]{amsart}
\usepackage{amssymb}

\usepackage[
width=31pc,
height=48pc,
margin=1.in,
footskip=30pt,
]{geometry}
\usepackage{layout}

\usepackage{graphicx}
\usepackage{epsfig}
\usepackage{subcaption}
\usepackage{amsmath,amsfonts,epstopdf}

\theoremstyle{plain}
\newtheorem{theorem}{Theorem}[section]
\newtheorem{lemma}[theorem]{Lemma}

\theoremstyle{definition}

\newtheorem{thma}{Theorem}

\newtheorem{thmb}{Theorem}

\newtheorem{thmc}{Theorem}

\newtheorem{thmd}{Theorem}

\numberwithin{equation}{section}

\newcommand{\R}{\mathbb{R}}

\newcommand{\C}{\mathbb{C}}
\newcommand{\N}{\mathbb{N}}

\renewcommand{\Im}{\mathrm{Im}}
\renewcommand{\Re}{\mathrm{Re}}

\title[On the dynamics of Translated Cone Exchange Transformations]{On the dynamics of Translated Cone Exchange Transformations}
\author{Pedro Peres and Ana Rodrigues}
\address{Department of Mathematics\\University of Exeter\\Exeter EX4 4QF, UK}
\begin{document}

\begin{abstract}
 In this paper we investigate {\it translated cone exchange transformations}, a new family of piecewise isometries and renormalize its first return map to a subset of its partition. As a consequence we show that the existence of an embedding of an interval exchange transformation into a map of this family implies the existence of infinitely many bounded invariant sets. We also prove the existence of infinitely many periodic islands, accumulating on the real line, as well as non-ergodicity of our family of maps close to the origin.
\end{abstract}
\maketitle
 
 \section{Introduction}\label{Introduction}

One of the central problems in dynamical systems is to investigate renormalization of certain classes of maps. We say a map $f:X \rightarrow X$  is renormalizable if there is a subset $Y \subseteq X$ such that the first return map  $f_Y: Y \rightarrow Y$ is conjugated to a map in the same family. 

Although renormalization of Interval Exchange Transformations (IETs) has been well studied over the past years, renormalization of Piecewise Isometries (PWIs) is still far from understood. 
 
In \cite{AKT} Adler, Kitchens and Tresser find renormalization operators for three rational rotation parameters for a non ergodic piecewise affine map of the Torus.  Lowenstein and Vivaldi  \cite{LV14} gave a computer assisted proof of the renormalization of a family of piecewise isometries of a rhombus with one translation parameter and a fixed rational rotation parameter. These results however rely on fixing the rotation component and heavily restricting other parameters in order to perform computer assisted calculations on cyclotomic fields. Recently, Hooper  \cite{Ho13} investigated a two dimensional parameter space of polygon exchange maps, a family of PWIs with no rotation, invariant under a renormalization operation, related to corner percolation and Truchet tillings, where each map admits a return map affinely conjugate to a map in the same family. In \cite{AKY} the authors show how to construct minimal rectangle exchange maps, associated to Pisot numbers, using a cut-and-project method and prove that these maps are renormalizable. The maps described in these papers are PWIs with no rotational component, exhibiting  very particular behaviour among typical PWIs, making it difficult to generalize their techniques.




In this paper we renormalize a particular family of PWIs. Before introducing our family of maps let us  define an \emph{interval exchange transformation} as in \cite{AF} (see also \cite{CFS} and \cite{Ke}).
Let $d \geq 2$ be a natural number and let $\pi$ be an irreducible permutation of $\{1,...,d\}$, that is, such that $\pi(\{1,...,k\})\neq \{1,...,k\}$ for $1\leq k < d$. Let $a \in \mathbb{R}_+^d$.
Consider the points
\begin{equation*}\label{eq0}
x_0=0, \quad x_j=\sum_{k=1}^{j}a_k, \quad 1\leq j \leq d,
\end{equation*}
and the interval $I=\left[x_0, x_d \right)$ partitioned  into subintervals
$I_j=[x_{j-1},x_j),$ for $1\leq j \leq d.$

The interval exchange transformation $f_{a,\pi} : I \rightarrow I$ rearranges $I_j$ according to $\pi$, that is
$f_{a,\pi}(x)=x+ w_j(a,\pi),$ for $x \in I_j,$
where $(w_j(a,\pi))_{j=1,...,d}$ is the \textit{translation vector} associated to $f_{a,\pi}$ and is given by
\begin{equation*}\label{eq4}
w_j(a,\pi)=\sum_{\pi(k)<\pi(j)}a_k - \sum_{k<j}a_k.
\end{equation*}
We also call $f_{a,\pi}$ a  $d$-IET as it is an interval exchange transformation of $d$ subintervals.



We now introduce our family of translated cone exchange transformations (TCEs). Consider a  partition of the upper half plane $\mathbb{H}$ into $d+2$ cones $\mathcal{P}= \{P_0,P_1,\dots,P_d, P_{d+1}\}$, where $P_j=\{z \in \mathbb{C}: \ \arg(z) \in W_j \},$ and $W_j$ for $j=0,\ldots,d+1$ are defined as
\begin{equation*}
  W_j=\left\{
  \begin{array}{ll}
   [0, \beta), & {\rm for} \ j=0,\\
   \left[\beta, \beta + \alpha_1 \right] , & {\rm for} \ j=1,\\
   (\beta+\sum_{k=1}^{j-1}\alpha_k, \beta+\sum_{k=1}^{j} \alpha_k], &  {\rm for} \ j \in \{2,...,d\},\\
   (\pi-\beta, \pi],  & {\rm for} \ j=d+1.
\end{array}\right.
\end{equation*}
We denote by $\partial \mathcal{P}$ the union of the boundaries of the elements of the partition $\mathcal{P}$ and by $L'_1$ and $L'_d$, respectively, the lines $\overline{P_0}\cap \overline{P_1}$ and $\overline{P_d}\cap \overline{P_{d+1}}$.

Set $\alpha=(\alpha_{1},...,\alpha_{d}) \in \mathbb{A}$, where
\begin{equation*}\label{regalpha}
\mathbb{A}=\left\{ \alpha\in \R_+^d: 0< \sum_{j=1}^{d}\alpha_j <\pi \right \}.
\end{equation*} 
 
 Note that we have
 \begin{equation}\label{beta}
 \beta=\dfrac{\pi}{2}-\dfrac{|\alpha|}{2},
 \end{equation} 
where $|\alpha|$ is the $\ell_1$ norm of $\alpha$.

Let $G:\mathbb{H} \rightarrow \mathbb{H}$ be the following family of translation maps
 \begin{equation*}\label{ren01}
  G(z)=\left\{\begin{array}{ll}
   z-1, & z \in P_0, \\
   z -\eta, & z \in P_j , \ j \in \{1,...,d\},\\
   z+ \lambda, & z \in P_{d+1},
  \end{array}\right.
 \end{equation*}
 depending on the parameters $\alpha, \beta, \lambda$ and $\eta$ with
  $\beta>0$, $\lambda \in  \mathbb{R}^+\backslash \mathbb{Q}$ and $0<\eta<\lambda$.

 
 Consider a permutation $\tau \in S(d)$ and let $\theta_j(\alpha,\tau)$ be the angle associated to the permutation $\tau$ for the cone $P_j$ for $j=1,\ldots,d$. We have
 \begin{equation}\label{ren01a}
  \theta_{j}(\alpha,\tau)=\sum_{\tau(k)<\tau(j)}\alpha_k - \sum_{k<j}\alpha_k.
 \end{equation}
 
 Let $E:\mathbb{H} \rightarrow \mathbb{H}$ be the following family of exchange maps
  \begin{equation*}\label{ren02}
  E(z)=\left\{ \begin{array}{ll}
   z,& z \in P_0 \cup P_{d+1}, \\
   ze^{i \theta_j(\alpha,\tau)},& z \in P_j , \ j \in \{1,...,d\},\\
  \end{array}\right.
 \end{equation*}
  depending on $\theta_{j}(\alpha,\tau)$. This map also depends on $\alpha$ and $\beta$ as the partition elements $P_j$ depend on these parameters. Note that we have
  $$\beta +\arg\left( E(z)/|z|  \right) = f_{\alpha,\tau}(\arg(z)-\beta),$$
  for $z \in P_j$, $j = 1,...,d$, where $\arg : \C \rightarrow [0,2\pi)$ is the argument function. Hence $E$ exchanges these cones according to the permutation $\tau$.

 From the translation and exchange families of maps we get our family of TCEs,
  $F:\mathbb{H} \rightarrow \mathbb{H}$,  given by
 \begin{equation*}\label{ren03}
  F(z)=G \circ E (z).
 \end{equation*}
  The dynamics of $F$ restricted to $P_0$ is a translation to the left by $1$ while the dynamics of $F$ restricted to $P_{d+1}$ is a translation to the right by $\lambda$, via the action of the translation map $G$. The rest of the cones are all permuted, according to a permutation $\tau$, by the exchange map $E$ and horizontally translated by $\eta$ by the translation map $G$.
 
 
Note that TCEs are cone isometry transformations for which the map induced by projection onto the circle at infinity $\hat{F}$ (see \cite{AG10}) is invertible. $F$ is defined on $\mathbb{H}\subseteq \mathbb{C}$, partitioned into $d+2$ cones by $\mathcal{P}$, hence it is a cone exchange transformation. $\hat{F}$ is an interval exchange transformation with interval partition given by $\{W_0,...,W_{d+1}\}$ and combinatorial data given by the permutation $\hat{\tau}$, where $\hat{\tau}(0)=0$, $\hat{\tau}(d+1)=d+1$, and $\hat{\tau}(j)=\tau(j)$, for $j=1,...,d$.

Let us introduce some notation. We define the {\it middle cone} $P_c$ of $F$ as
\begin{equation*}\label{ren04}
P_c=P_1\cup...\cup P_d,
\end{equation*}
the {\it first hitting time of $z \in \mathbb{H}$ to $P_c$}, as the map $k:\mathbb{H}\rightarrow \N$ given by
\begin{equation*}\label{eqk}
k(z)=\inf\{n \geq 1: F^n(z)\in P_c\},
\end{equation*}
and the {\it first return map of $z \in P_c$ to $P_c$}, as the map $R:P_c\rightarrow P_c$ such that
\begin{equation*}\label{ren05}
R(z)=F^{k(z)}(z).
\end{equation*}


The typical notion of renormalization may not capture all possible self similar behaviour in PWIs. TCEs apparently exhibit invariant regions on which the dynamics is self similar after rescaling. Thus, we say a TCE is renormalizable if $R$, the first return map to $P_c$ described above, is conjugated to itself by a scaling map. In Theorem \ref{renormtheorem} we renormalize, in this sense, TCEs for all rotation parameters and for infinitely many translational parameters. We show that for a set of parameters, the first return map under a TCE to $P_c$, is self-similar by a scaling factor $\Phi^2$ where $$\Phi=(\sqrt{5}-1)/2.$$

\begin{thma}\label{renormtheorem}
 For all $\alpha \in \mathbb{A}$, $\lambda= 1/(k+\Phi)$ and $\eta=1-k\lambda$ with $k \in \N$, there is an open set $U$ containing the origin such that $F$ is renormalizable for all $z \in U$, that is
 \begin{equation}\label{ren50}
 R(\Phi^2 z)=\Phi^2 R(z).
 \end{equation}
\end{thma}


As a consequence of this we show that for these parameters $R$ is a PWI with respect to a partition $\mathcal{P}_R$ of countably many atoms.




We say, as in \cite{AGPR}, that $h:I\rightarrow X$ is a {\em continuous embedding} of an IET $f:I\rightarrow I$ into a PWI $T:X \rightarrow X$ if $h$ is a homeomorphism onto its image and
\begin{equation}\label{eq0s2}
h\circ f(x) = T\circ h(x)~~\mbox{ for all }x\in I.
\end{equation}
It was proved in \cite{AGPR} that \textit{non trivial} continuous embeddings of minimal IETs into PWIs, this is, continuous embeddings which are not unions of circles or lines, cannot exist for 2-PWIs and can be at most 3 for any given 3-PWI. In the same paper it is provided numerical evidence for the existence of non trivial embeddings into a 4-PWI belonging to the family of TCEs.

 We say a collection of atoms $\mathcal{B}\subseteq  \mathcal{P}  $ is a \emph{barrier} for a PWI $(T,\mathcal{P})$ if $X \backslash \mathcal{B}$ is the union of two disjoint connected components $B_1$, $B_2$ such that
$$B_1 \cap T(B_2)= T(B_1) \cap B_2 = \emptyset,$$
and for any $P \in \mathcal{P}$ such that $P \subseteq B_j$ and $T(P)\cap \overline{\mathcal{B}}\cap \overline{B_j}= \emptyset$ then $T(P)\cap \mathcal{B} = \emptyset$, for $j=1,2$.

For $\alpha \in \mathbb{A}$, $\lambda= 1/(k+\Phi)$ and $\eta=1-k\lambda$, $k \in \N$ let $U$ be the open set in Theorem \ref{renormtheorem}. We denote by $\mathfrak{A}(\lambda,\eta)$ the subset of $\mathbb{A}$ such that for all $\alpha \in \mathfrak{A}(\lambda,\eta)$ there exist $d'\geq2$, $a \in \R_+^{d'}$, $\pi \in S(d')$ and  a continuous embedding $h$ of $f_{a,\pi}:I\rightarrow I$ into $R:P_c\rightarrow P_c$ such that $h(I)\subset \Phi^2 U$, $h(0)\in L_d'$, $h(|a|)\in L_1'$ and the collection
\begin{equation*}
\mathcal{B}=\{ P\in \mathcal{P}_R :P \cap h(I) \neq \emptyset   \},
\end{equation*}
is a barrier for $R$.

In the next theorem we show, as a consequence of renormalization of TCEs, that the existence of one continuous embedding of an IET into a first return map $R$  of a TCE, satisfying the property that the image of the embedding is contained in a barrier, implies the existence of infinitely many embeddings of the same IET into $R$,  as well as infinitely many bounded and forward invariant regions. This shows in particular that if one non trivial embedding exists then the  results from \cite{AGPR} for 2,3-PWIs do not generalize for PWIs with partitions with a higher number of atoms. We prove that for $\alpha \in \mathfrak{A}(\lambda,\eta)$ there are infinitely many sets, bounded away from $0$ and infinity, which are forward invariant by $R$ and that there exist infinitely many continuous embeddings of IETs into $R$.

\begin{thmb}\label{infembed}
	Let $\lambda= 1/(k+\Phi)$ and $\eta=1-k\lambda$ with $k \in \N$. For all $\alpha \in \mathfrak{A}(\lambda,\eta)$,
	
	i) There exist sets $V_1,V_2,...$, which are forward invariant for $R$ and $y^*>0$ such that for all $z \in P_c$, satisfying  $0<\Im(z)<y^*$, there is an $n \in \N$ for which $z \in V_n$.
	
	ii) For all $n \in \N$ there exist constants $0<b_n<B_n$ such that for all $z \in V_n$ and $k \in \N$,
	\begin{equation}\label{bbounds}
	b_n< |F^k(z)|<B_n.
	\end{equation}

	iii) There exist infinitely many continuous embeddings of IETs into $R$.
	
\end{thmb}

\begin{figure}[t]
	\begin{subfigure}{.49\textwidth}
		\centering
		\includegraphics[width=0.95\linewidth]{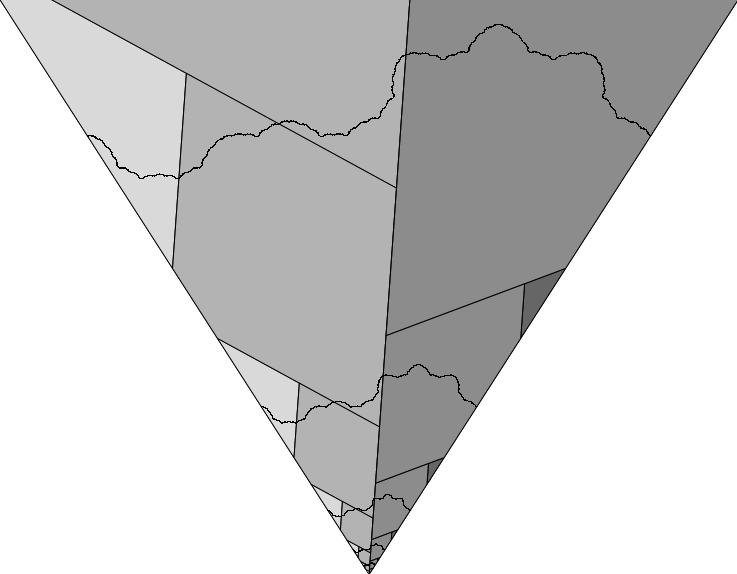}
		\caption{}
		\label{fig:R10}
	\end{subfigure}
	\begin{subfigure}{.49\textwidth}
		\centering
		\includegraphics[width=0.95\linewidth]{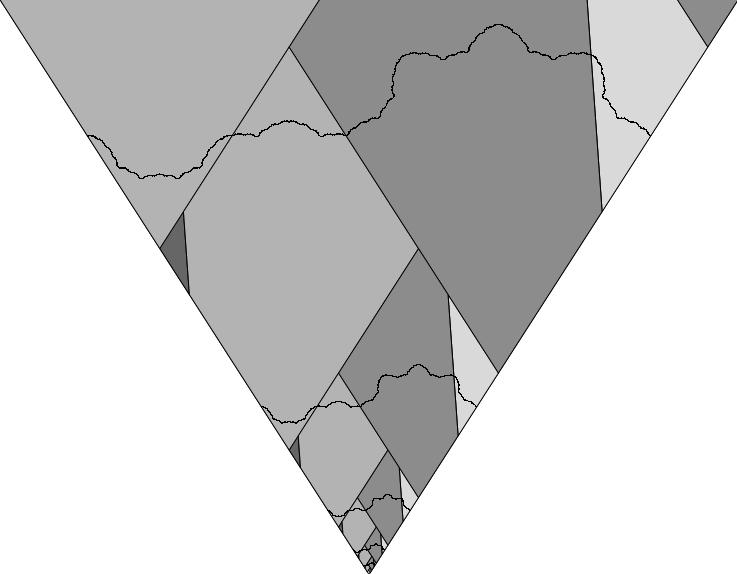}
		\caption{}
		\label{fig:R11}
	\end{subfigure}
	\centering
	\captionsetup{width=1\textwidth}
	\caption{A schematic representation of the action of $R$ on the cone $P_c$ close to the origin, for parameters $\tau=(12)$, $\alpha=(0.5,\pi-2.5)$,  $\lambda=\Phi$ and $\eta=1-\Phi$. $R$ is a PWI with respect to a partition of infinitely many atoms, which correspond to the polygons depicted in the figure (A). In (B) the image of this partition by $R$ can be seen. Each curve in both figures corresponds to the orbit of a given point. By Theorem \ref{renormtheorem} if the orbit remains close to the origin then there are infinitely many replica of this orbit accumulating on the origin. It is still an open question whether the closure of such an orbit can be the image of an embedding of an IET into $R$.}
	\label{fig:R}
\end{figure}

An \emph{horizontal periodic orbit} is a periodic orbit $\mathcal{O}$, such that there is an $h \in \R$ for every $z_k \in \mathcal{O}$ such that $\Im(z_k)=h$ for all $k \in \mathbb{N}$. We say $h$ is the \emph{height} of the orbit. 
An \emph{horizontal periodic island} is a periodic island that contains an horizontal periodic orbit.

Let $\mathcal{R}(\tau)$ denote the set of all $\alpha \in \mathbb{A}$ such that for some $j \in \{1,...,d\}$ we have
\begin{equation}\label{reflective}
\left|\sum_{\tau(k)>\tau(j)}\alpha_k -  \sum_{k<j}\alpha_k \right| < \alpha_{j}.
\end{equation}

Given $\tau \in S(d)$, let $J_{\mathcal{R}}(\tau)$ be the set of all $j\in \{1,...,d\}$ such that \eqref{reflective} holds for some $\alpha \in \mathbb{A}$.

Define the sets $\zeta_{-}(d)$ (resp. $\zeta_{+}(d)$) of all $\tau \in S(d)$ such that there is a $j'\in J_{\mathcal{R}}(\tau)$ and a $j'' \in \{1,...,d\}$ such that $j'<j''$ and $\tau(j'')<\tau(j')$ (resp. $j'>j''$ and $\tau(j'')>\tau(j')$). Denote by $\zeta(d)$ their union $\zeta_{-}(d) \cup \zeta_{+}(d)$.

In our next theorem we prove that there is a non-empty open set of rotation parameters for which TCEs have infinitely many horizontal periodic islands accumulating on the real line.

\begin{thmc}\label{infperisl}
	Let $\tau \in \zeta(d)$, $\lambda= 1/(k+\Phi)$ and $\eta=1-k\lambda$, for some $k \in \N$.
	There is a non-empty open set $\mathcal{A} \subseteq \mathbb{A}\cap \mathcal{R}(\tau)$ such that for all $\alpha \in \mathcal{A}$, $F$ has infinitely many horizontal periodic islands accumulating on the real line.
\end{thmc}

As a result  we get that for the same parameter set, TCEs are not ergodic in a neighbourhood of the origin.

\begin{thmd}\label{notergodic}
	Let $\tau \in \zeta(d)$, $\alpha \in \mathbb{A}\cap \mathcal{R}(\tau)$, $\lambda= 1/(k+\Phi)$ and $\eta=1-k\lambda$, for some $k \in \N$. If $U$ is an invariant set for $R$  that contains a neighbourhood of the origin then the restriction of $R$  to $U$ does not have a dense orbit. In particular $F$ is not ergodic.
\end{thmd}

This paper is organized as follows. In Section 2 we investigate a family of maps related to IETs. In Section 3 we study the sequence of bifurcation points and the bifurcation sequence for the family of maps introduced in the previous section making use of the theory of continued fractions. In Section 4 we introduce two sequences that we designate by dynamical sequences that will be an important tool to prove our main theorems. We derive inductive formulas to compute these sequences. In Section 5 we study the dynamics of the first return map to the middle cone $P_c$. Finally, in Sections 6 and 7 we prove our main results, theorems A,B, C and D.

\section{Bifurcation Points}\label{Bifurcation Points}
In this section we study a specific family of maps, closely related to IETs, 
on the interval $I=[0,1+\lambda]$
with $\lambda \in \mathbb{R}^+\backslash\mathbb{Q}$. We will introduce the left and right bifurcation points and bifurcation sets for this family. 

Consider the interval $I=[0,1+\lambda]$ and the following family of maps 
\begin{equation}\label{eq:3}
g_{\ell}(x)=\left\{\begin{array}{ll}
x + \lambda, & x \in I_1(\ell)\\
x, & x \in I_c(\ell)\\
x-1, & x \in I_2(\ell).
\end{array}\right.
\end{equation}
with $I_1(\ell)=[0,1]$, $I_c(\ell)=(1,1+\ell)$ and $I_2(\ell)=[1+\ell,1+\lambda]$. To simplify notation we will only include the argument when it is necessary, otherwise we just refer to these intervals as $I_j,$ for $j=1,2,c$.

Given $\beta \in (0, \pi/2)$, consider the region
$$\mathcal{R}_{\lambda,\beta} = \{ z \in \mathbb{H} \backslash P_c: \Re(z) + \Im(z)\cot(\beta) \in [-1,\lambda] \ \textrm{and} \ 2 \Im(z)\cot(\beta)\leq \lambda   \}.$$
The next lemma relates iterates of our family of maps $F$ with iterates of $g_{\ell}$ for some values of $z$.

\begin{lemma}\label{cor2.3}
For any $\lambda>0$, $\beta \in (0, \pi/2)$ and $z \in \mathcal{R}_{\lambda,\beta}$ we have
	\begin{equation}\label{cor2.3eq0}
	F^n(z)=s^{-1}\circ g_{2\Im(z)\cot(\beta)}^n\circ s (\Re(z))+ i \Im(z),
	\end{equation}
	for all $n\leq k(z)$, where $s(x)=x+1+\ell/2$.
\end{lemma}

\begin{proof}
As $z \in \mathcal{R}_{\lambda,\beta}$ we have $z \in P_{0}\cap \mathcal{R}_{\lambda,\beta}$ or $z \in P_{d+1}\cap \mathcal{R}_{\lambda,\beta}$. By the definitions of $P_{0}$ and $P_{d+1}$, in both cases we have $\Re(F(z))= s^{-1}\circ g_{2\Im(z)\cot(\beta)}\circ s (\Re(z))$.
It is direct to see that for $n \leq k(z)$ we have $F^n(z) \in \mathcal{R}_{\lambda,\beta}$ and thus repeating the previous argument $n$ times we get $\eqref{cor2.3eq0}$.
\end{proof}

We define  the {\it first hitting time} of $x$ to $\overline{I_c(\ell)}$ as the map $n_{\ell}: I \rightarrow \mathbb{N}$ given by
\begin{equation*}\label{eqN2}
n_{\ell}(x) = \inf \{n \geq 1: g_{\ell}^n(x) \in \overline{I_c(\ell)}\},
\end{equation*}
and the {\it first hitting map} of $x$ to $\overline{I_c(\ell)}$, as the map
\begin{equation*}\label{1hitmap}
r_{\ell}(x) = g_{\ell}^{n_{\ell}(x)}(x).
\end{equation*}

For our next lemma we need also to consider the map
\begin{equation*}\label{eqr1}
r'_{\ell}(x)=\left\{\begin{array}{ll}
r_{\ell}(x),& x \notin \overline{I_c(\ell)},\\
x,& x \in \overline{I_c(\ell)}.
\end{array}\right.
\end{equation*}

\begin{lemma}\label{prop2.4}
 Let $\lambda \in \R^+ \backslash \mathbb{Q}$ and $0<\beta<\frac{\pi}{2}$. If $z \in P_c$ with
 $2\Im(F(z))\cot(\beta)\leq \lambda,$ then
 \begin{equation}\label{eqr2}
 R(z)=s^{-1}\circ r'_{2\Im(F(z))\cot(\beta)}\circ s (\Re(F(z)))+ i \Im(F(z)).
 \end{equation}
\end{lemma}

\begin{proof}
	It is clear that if $F(z) \in P_c$, then we have \eqref{eqr2}, so we may assume  $F(z) \in \mathbb{H}\backslash P_c$.
%
Since $2 \Im(F(z)) \cot(\beta) \leq \lambda$, by definition of $F$ we get
	\begin{equation*}\label{ineqs4}
	-1-\Im(z)\cot(\beta)\leq \Re(F(z)) \leq \lambda-\Im(z)\cot(\beta),
	\end{equation*}
	and thus $F(z) \in \mathcal{R}_{\lambda,\beta}$. From Lemma \ref{cor2.3} it follows that \eqref{cor2.3eq0} holds for all $n \leq k(F(z))$.
	
	It is simple to see that	
	\begin{equation*}\label{ineqs5}
	k(F(z))=n_{2\Im(F(z))\cot(\beta)}(\Re(F(z))),
	\end{equation*}
	and thus by definition of $r_{\ell}'$ we get \eqref{eqr2} as intended.
\end{proof}

Let $\lambda \in \mathbb{R}^+\backslash\mathbb{Q}$ and $I=[0,1+\lambda]$. Consider the map
\begin{equation*}\label{eqg}
	g(x)=\left\{\begin{array}{ll}
	\vspace{0.2cm}
		x + \lambda, & x \in [0,1]\\
		x-1, & x \in (1,1+\lambda].
	\end{array}\right.
\end{equation*}

Let $N \in \N$. Define
\begin{equation*}\label{dminus}
d^-(N)=\left\{\begin{array}{ll}
\vspace{0.2cm}
1, & \textrm{if} \ g^n(1)>1 \ \textrm{for all} \ 1\leq n \leq N,\\
1-\max_{1\leq n \leq N}\left\{g^n(1)\leq 1\right\}, & \textrm{otherwise},
\end{array}\right.
\end{equation*}
and
\begin{equation*}\label{dplus}
d^+(N)=\left\{\begin{array}{ll}
\vspace{0.2cm}
\lambda, & \textrm{if} \ g^n(1)<1 \ \textrm{for all} \ 1\leq n \leq N,\\
\min_{1\leq n \leq N}\left\{g^n(1)\geq 1\right\}-1, & \textrm{otherwise}.
\end{array}\right.
\end{equation*}

We want now to investigate orbits by $g$ of points which are in a small neighbourhood of $1$. We prove the next lemmas.
\begin{lemma}\label{dlemma} Assume that $\lambda \in \R^+\backslash \mathbb{Q}$.

	i) If $N\geq 0$ and $0 \leq \ell < d^+(N)$, then for all $0 \leq n \leq N$ we have
	\begin{equation}\label{g1ml}
	g^n(1-\ell)=g^n(1)-\ell.
	\end{equation}

	ii) If $N\geq 2$  and $0 \leq \ell \leq d^-(N)$, then for all $2 \leq n \leq N$ we have
	\begin{equation}\label{g1pl}
	g^n(1+\ell)=g^n(1)+\ell.
	\end{equation}
\end{lemma}

\begin{proof}
	To simplify notation we denote $d^+=d^+(N)$ and  $d^-=d^-(N)$.
	Let us prove i) by induction on $n$. It is clear that \eqref{g1ml} holds for $n \in \{0,1\}$. We now assume \eqref{g1ml} holds for $1 \leq n <N$ and we show it holds for $n+1$.
	
It follows from  the definitions of $d^-$ and $d^+$ that 
	$g^n(1)\notin (1-d^-,1+d^+),$
	for  $1 \leq n \leq N$, thus $g^n(1)\leq 1-d^-$ or $g^n(1)\geq 1+d^+$.
	
	If $g^n(1)\leq 1-d^-$, then as $\ell \geq 0$ and since we are assuming \eqref{g1ml} holds for $n$ we have $g^n(1-\ell)\leq 1-d^-$. Therefore $g^n(1-\ell) \in [0,1]$ and since $g^n(1)\in [0,1]$ we get
%
	$$g^{n+1}(1-\ell)=g^{n+1}(1)-\ell.$$

	If  $g^n(1)\geq 1+d^+$, then as $\ell <d^+$ and since we are assuming \eqref{g1ml} holds for $n$ we have $g^n(1-\ell)>1$. Therefore $g^n(1) \in (1,1+\lambda]$ and thus
	\begin{equation*}\label{dlemma2}
	g^{n+1}(1-\ell)=g^n(1)-\ell-1.
	\end{equation*}
	Since $g^n(1)\in (1,1+\lambda]$, we get that \eqref{g1ml}, holds for $n+1$ and we finish the proof of i).
	
	The proof of ii) is similar to the proof of i) so we omit it.
\end{proof}

Given $\ell>0$ and $x \in I\backslash [1,1+\ell]$, we define
\begin{equation}\label{ndm}
d^{-}(x,n_{\ell}(x))=1-\max_{0\leq n\leq n_{\ell}(x)}\left\{g^n(x)\leq 1  \right\}.
\end{equation}

\begin{lemma}\label{l2}	
	Assume  $0<\ell'<\ell$, $x \in I\backslash [1,1+\ell]$ and $x' \in (x-(\ell-\ell'),x+d^{-}(x,n_{\ell}(x))).$ Then for all $n \leq n_{\ell}(x)$ we have
	\begin{equation}\label{dag1}
	g_{\ell}^n(x)-g_{\ell'}^n(x')=x-x'.
	\end{equation}
	Moreover, $n_{\ell'}(x')\geq n_{\ell}(x)$.
\end{lemma}

\begin{proof}
		To simplify notation we denote $d^-=d^{-}(x,n_{\ell}(x))$. We proceed by induction on $n$. It is clear that \eqref{dag1} holds for $n=0$. Now assume \eqref{dag1} holds for $n<n_{\ell}(x)$ and we prove it for $n+1$ instead.
		
		As $n<n_{\ell}(x)$ we have $g_{\ell}^n(x) \notin [1,1+\ell]$. Since we are assuming  \eqref{dag1} holds for $n$, we get 
		$$g_{\ell'}^n(x') \in (g_{\ell}^n(x)-(\ell-\ell'),g_{\ell}^n(x)+d^{-}).$$
		
		If $g_{\ell}^n(x)<1$, then  $g_{\ell}^n(x)\leq 1- d^{-}$ and thus $g_{\ell'}^n(x') \in (1-d^{-}-(\ell-\ell'),1).$
		
		Otherwise, if $g_{\ell}^n(x)>1+\ell$ then $g_{\ell'}^n(x') \in (1+\ell',1+\ell+d^{-}),$
		thus we have  $g_{\ell'}^n(x') \in I_j$ if and only if $g_{\ell'}^n(x) \in I_j$, for $j=0,1$ and $g_{\ell'}^n(x') \notin [1,1+\ell']$. Therefore by \eqref{eq:3} we get $g_{\ell}^{n+1}(x)-g_{\ell'}^{n+1}(x')=x-x'$.
		This proves \eqref{dag1}  for $n \leq n_{\ell}(x)$. 
		
		Since $g_{\ell'}^n(x') \notin [1,1+\ell']$ for $ n<n_{\ell}(x)$ we have $n_{\ell'}(x')\geq n_{\ell}(x)$ and we finish our proof. 
\end{proof}

In the beginning of this section we defined the  first hitting map of $x$ to $\overline{I_c(\ell)}$, as 
$r_{\ell}(x) = g_{\ell}^{n_{\ell}(x)}(x)$, where $n_{\ell}(x)$ is the first hitting time of $x$ to $\overline{I_c(\ell)}$. We want now to investigate when is $1+ \ell$ mapped to $1$ under $r_{\ell}(x)$ and when is $1$ mapped to $1+ \ell$. Note that these are the endpoints of the middle interval $\overline{I_c(\ell)}$. We define the following points and sets.

We say $\ell$ is a {\it right bifurcation point} if $r_{\ell}(1+\ell)=1$, $\ell$ is a {\it left bifurcation point} if $r_{\ell}(1)=1+\ell$ and $\ell$ is a {\it bifurcation point} if it is either a left or right bifurcation point. 

The \emph{left/right bifurcation sets} are defined respectively as
\begin{equation*}\label{BifSetL}
\Lambda_L=\{0< \ell \leq \lambda: \textrm{for all  } l<\ell,\   n_\ell(1)<n_l(1)\},
\end{equation*}
and
\begin{equation*}\label{BifSetR}
\Lambda_R=\{0< \ell \leq \lambda: \textrm{for all  } l<\ell,\   n_\ell(1+\ell)<n_l(1+l)\}.
\end{equation*}

The main result of this section is the next theorem, relating bifurcation points with the bifurcation sets.

\begin{theorem}\label{prop:1}
 $\ell$ is a left (resp. right) bifurcation point if and only if $\ell \in \Lambda_L$ (resp. $\ell \in \Lambda_R$).
 Furthermore,  $\ell \rightarrow n_{\ell}(1)$ and $\ell \rightarrow n_{\ell}(1+\ell)$ are decreasing functions of $\ell$ and the sets $\Lambda_R$, $\Lambda_L$ are discrete with $0$ as the only possible point of accumulation.
\end{theorem}

\begin{proof}

	
	
	First assume that $r_{\ell}(1)=1+\ell$ and $l <\ell$. By the definitions of $n_{\ell}$ and $r_{\ell}$ we have, for $1 \leq n < n_{\ell}(1)$, that either $g_{\ell}^n(1) < 1$ or $g_{\ell}^n(1) > 1+\ell$. As $l <\ell$, by \eqref{eq:3} we have for $1 \leq n < n_{l}(1)$, $g_{l}^n(1) < 1$ or $g_{l}^n(1) > 1+l$. Thus $n_{\ell}(1)\leq n_{l}(1)$ and $g_{l}^{n_{l}(1)}(1)=g_{\ell}^{n_{\ell}(1)}(1)$.
	Since $g_{\ell}^{n_{\ell}(1)}(1) =1+\ell$ and $1+\ell>1+l$ this shows $g_{l}^{n_{l}(1)}(1)>1+l$ and thus $n_{\ell}(1)< n_{l}(1)$. This proves that if $\ell$ is a left bifurcation point, then  $\ell \in \Lambda_L$.

	Now assume that $r_{\ell}(1) \neq 1+ \ell$. As $\lambda$ is irrational we must have $r_\ell(1) \in (1,1+\ell)$, therefore there is  $0<\ell'<\ell$ such that $g_{\ell}^{n_{\ell}(1)}(1)=1+\ell'$.
	
	We show, by induction on $n$, that for all $l \in [\ell',\ell]$ and $0\leq n \leq n_{\ell}(1)$
	\begin{equation}\label{prop:1e1}
	g_{l}^n(1)=g_{\ell}^n(1).
	\end{equation}
	It is clear that \eqref{prop:1e1} holds for $n=0$. We assume it holds for $n<n_{\ell}(1)$ and we prove it for $n+1$. As $n<n_{\ell}(1)$ we have $g_{\ell}^n(1) \notin (1,1+\ell)$, and since $g_{l}^n(1)=g_{\ell}^n(1)$, this implies that $g_{l}^n(1)\notin (1,1+l)$, thus by \eqref{eq:3} we have that \eqref{prop:1e1} must hold for $n+1$.
	
	Since \eqref{prop:1e1} holds for $n=n_{\ell}(1)$ we have $g_{l}^{n_{\ell}(1)}(1)=1+\ell'$ and thus $n_{\ell}(1)=n_{l}(1)$ for all $l \in [\ell',\ell]$.
	
	Thus, there is  $l<\ell$ such that $n_{\ell}(1)\geq n_{l}(1)$. 	
	This proves that $\ell$ is a left bifurcation point if and only if $\ell \in \Lambda_L$. Note that it also shows that $\ell \rightarrow n_{\ell}(1)$ is a decreasing function of $\ell$.

	
	
	By Lemma \ref{l2}, for all $l <\ell$ and $0 \leq n \leq n_{\ell}(1+\ell)$ we have
	\begin{equation}\label{prop:1e2}
	g_{l}^n(1+l)=g_{\ell}^n(1+\ell)-(\ell-l).
	\end{equation}
	From which follows that $g_{l}^{n_{\ell}(1+\ell)}(1+l)=g_{\ell}^{n_{\ell}(1+\ell)}(1+\ell)-(\ell-l).$
	If $r_{\ell}(1+\ell)=1$, as $r_{\ell}(1+\ell)=g_{\ell}^{n_{\ell}(1+\ell)}(1+\ell)$, this implies
	\begin{equation*}\label{prop:1e4}
	g_{l}^{n_{\ell}(1+\ell)}(1+l)=1-(\ell-l) \notin \overline{I_c(l)},
	\end{equation*}
	thus, $n_{\ell}(1+\ell)<n_{l}(1+l)$.
	Then for all $l<\ell$, we have  $n_{\ell}(1+\ell)<n_{l}(1+l)$. This proves  that if $\ell$ is a right bifurcation point then $\ell \in \Lambda_R$,

	If $r_{\ell}(1+\ell)\neq 1$, as $\lambda$ is irrational we must have $r_\ell(1+\ell) \in (1,1+\ell)$, therefore there is an $0<\ell'<\ell$ such that $g_{\ell}^{n_{\ell}(1+\ell)}(1+\ell')=1$.
	
	Now take $l \in [\ell',\ell)$. By \eqref{prop:1e2} we get
	$$g_{l}^{n_{\ell}(1+\ell)}(1+l)=1+l-\ell' \in [1,1+l),$$
	hence $g_{l}^{n_{\ell}(1+\ell)}(1+l) \in \overline{I_c(l)}$ and we have $n_{\ell}(1+\ell)=n_{l}(1+l)$.
	Thus, there is a $l<\ell$ such that $n_{\ell}(1+\ell)\geq n_{l}(1+l)$. This proves that if $\ell \in \Lambda_R$ then $\ell$ is a right bifurcation point.	
	This proves that $\ell$ is a right bifurcation point if and only if $\ell \in \Lambda_R$. Note that it also shows that $\ell \rightarrow n_{\ell}(1+\ell)$ is a decreasing function of $\ell$.
	
	Since $\ell \rightarrow n_{\ell}(1)$ and $\ell \rightarrow n_{\ell}(1+\ell)$ are decreasing functions of $\ell$ and are also integer valued functions this implies that the sets $\Lambda_L$ and $\Lambda_R$ are discrete and each has at most one point of accumulation, which has to be $0$.
\end{proof}

\section{Bifurcation sequence}\label{Bifurcation sequence}

In this section we study the sequence of bifurcation points for the family $g_l$ (in \eqref{eq:3}). We will first recall some elements of the theory of continued fractions, and compute the sequence of errors of the semiconvergents of $\lambda=1/(k+\Phi)$, where $\Phi=(\sqrt{5}-1) /2$ and $k \in \N$. 
We will then relate the bifurcation sequence with the theory of continued fractions by showing that this sequence is equal to the sequence of errors of the semiconvergents of $\lambda$.

Throughout this section we assume that $\lambda \in (0,1)$ is an irrational real number with continued fraction expansion $\lambda=[0,\lambda_1,\lambda_2,...]$. Consider the sequence of its \emph{convergents} given by
\begin{equation*}\label{conv}
\left\{\frac{p_n}{q_n}\right\}_{n \geq 0}, \ \textrm{where} \ \ \frac{p_0}{q_0}=\frac{0}{1} \ \ \textrm{and} \ \ \frac{p_n}{q_n}=[0,\lambda_1,...,\lambda_n].
\end{equation*}


For all $n\geq 0$ it is well known that
\begin{equation}\label{CF}
\begin{array}{l}
p_{n+2}=p_{n}+ \lambda_{n+2} p_{n+1},\\
q_{n+2}=q_{n}+ \lambda_{n+2} q_{n+1}.
\end{array}
\end{equation}

Define the sequence of \emph{upper semiconvergents} of $\lambda$ as
$$\left \{\frac{p_n'}{q_n'} \right \}_n = \{[0,1], ...,[0,\lambda_1],[0,\lambda_1,\lambda_2,1],...,[0,\lambda_1,\lambda_2,\lambda_3],...\}.$$
which is the  sequence of best rational approximations of $\lambda$ by above , this is, any other fraction $\frac{a}{b}\neq \frac{p_n'}{q_n'}$, with $1 \leq b \leq q_n'$, satisfies $a - b \lambda >  p_n'-q_n' \lambda$ (see for instance \cite{Khi}).

The sequence of errors of approximation of the upper semiconvergents smaller than $\lambda$ is given by
$$\Gamma'_n=\left \{p_{n+\lambda_1-1}'-q_{n+\lambda_1-1}' \lambda \right \}_n.$$

Analogously, we define the sequence of \emph{lower semiconvergents} of $\lambda$ as
$$\left \{\frac{p_n''}{q_n''} \right \}_n = \{0,[0,\lambda_1,1],...,[0,\lambda_1,\lambda_2],[0,\lambda_1,\lambda_2,\lambda_3,1],...,[0,\lambda_1,\lambda_2,\lambda_3,\lambda_4],...\}.$$
which is the sequence of best rational approximations of $\lambda$ by below, this is, any other fraction $\frac{a}{b}\neq \frac{p_n''}{q_n''}$, with $1 \leq b \leq q_n''$, satisfies $b \lambda -a > q_n'' \lambda -p_n''$.

The sequence of errors of approximation of the lower semiconvergents is given by
$$\Gamma_n''=\left \{q_n'' \lambda - p_n'' \right \}_n.$$
Note that  $\Gamma'$ and $\Gamma''$ are monotonic sequences of positive real numbers that converge to $0$.
Finally, we define recursively the intercalation of $\Gamma_n'$ and $\Gamma_n''$ as $\Gamma_n$ given by
\begin{equation*}\label{Gamma}
\Gamma_0=\max(\Gamma'_0,\Gamma''_0), \quad \Gamma_n=\max\left((\Gamma'\cup\Gamma'')\backslash \bigcup_{k=0}^{n-1}\Gamma_k\right), \quad n \geq 1.
\end{equation*}

In the next lemma,  we compute explicitly the sequences $\Gamma_n$, $\Gamma'_n$ and $\Gamma''_n$.
\begin{lemma}\label{tnewlambda}
 Let $\Phi=(\sqrt{5}-1)/2$, $k \in \N$ and $\lambda=1/(k+\Phi)$. For all $n \geq 0$ we have
 \begin{equation}\label{tnewlambdaeq1}
 \Gamma'_n=\lambda \Phi^{2n+1},
 \end{equation}
 \begin{equation}\label{tnewlambdaeq1a}
 \Gamma''_n=\lambda \Phi^{2n},
 \end{equation}
 and
 \begin{equation}\label{tnewlambdaeq2}
 \Gamma_n=\left\{
 \begin{array}{ll}
 \Gamma'_{(n-1)/2}, & {\textrm if \ } n { \ \textrm is \ odd},\\
 \Gamma''_{n/2}, & {\textrm if \ } n { \ \textrm is \ even}.
 \end{array}\right.
 \end{equation}
\end{lemma}

\begin{proof}
			Let $\{F_n\}_{n \geq 0}$, be the Fibonacci sequence, given by $F_0=0$, $F_1=1$ and
			\begin{equation*}\label{fib}
			F_n=F_{n-1}+F_{n-2},
			\end{equation*}
			for $n \geq 2$.

			We begin by proving, by induction on $n$, that for all $n \geq 0$, 
			\begin{equation}\label{tnewlambdaeq3}
			\lambda \Phi^{2n+1}=F_{2n+1}-(F_{2n+1} k + F_{2n})\lambda,
			\end{equation}
			and that for all $n\geq 1$,
			\begin{equation}\label{tnewlambdaeq4}
			\lambda \Phi^{2n}=(F_{2n} k + F_{2n-1})\lambda-F_{2n}.
			\end{equation}
			Clearly, \eqref{tnewlambdaeq3} holds for $n=0$ and  \eqref{tnewlambdaeq4} holds for $n=1$.			
			Assuming that \eqref{tnewlambdaeq3} holds for $n$, \eqref{tnewlambdaeq4} holds for $n+1$ and using $1-\Phi=\Phi^2$, we get
			$F_{2n+3}-(F_{2n+3}k+F_{2n+2})\lambda=\lambda \Phi^{2n+3}.$ The proof of \eqref{tnewlambdaeq4} is similar to the proof of \eqref{tnewlambdaeq3} and so we omit it.

			Using the fact that $\{F_n\}_{n \geq 0}$ is the Fibonacci sequence and some elementary properties of continued fractions it can be easily proved by induction on $n$ that
			\begin{equation}\label{tnewlambdaeq7}
			p_n=F_{n}, \quad q_n= F_{n}k+F_{n-1}.
			\end{equation}

			Finally we show that \eqref{tnewlambdaeq1} holds for all $n\geq 0$. It is clear  that $\Gamma'_0=1-k \lambda= \lambda \Phi$ and $\Gamma''_n=\lambda \Phi^{2n}$ for $n=0,1$. Hence \eqref{tnewlambdaeq1} holds for $n=0$, and \eqref{tnewlambdaeq1a} holds for $n=0,1$.
			
			Now assume \eqref{tnewlambdaeq1} holds for all $0\leq n\leq N$ and \eqref{tnewlambdaeq1a}  for all $0\leq n\leq N+1$. We now prove that \eqref{tnewlambdaeq1} holds  for $0\leq n\leq N+1$ and \eqref{tnewlambdaeq1a}  for all $0\leq n\leq N+2$. We have
			$\Gamma_0''>\Gamma_0'>\Gamma_1''>...>\Gamma_{N}''>\Gamma_{N}'>\Gamma_{N+1}''.$

			Thus, we have \eqref{tnewlambdaeq2} for $n\leq 2(N+1)$, also
			\begin{equation}\label{tnewlambdaeq12}
			p''_{n}=p_{2n}, \quad  q''_{n}=q_{2n},
			\end{equation}
			for $1 \leq n\leq N+1$, and
			\begin{equation}\label{tnewlambdaeq13}
			p'_{n+\lambda_1-1}=p_{2n+1}, \quad  q'_{n+\lambda_1-1}=q_{2n+1},
			\end{equation}
			for $0 \leq n\leq N$.

			By \eqref{tnewlambdaeq3} and \eqref{tnewlambdaeq7}, we get 
			$p_{2N+3}-\lambda q_{2N+3}=\lambda \Phi^{2N+3}>0$. Thus, 
			 from \eqref{tnewlambdaeq12} and \eqref{tnewlambdaeq13} we have
			$p'_{N+\lambda_1}=p_{2N+3}, q'_{N+\lambda_1}=q_{2N+3}$ and we get 
 $\Gamma'_{N+1}=\lambda \Phi^{2N+3}$. This proves \eqref{tnewlambdaeq1} for  $0 \leq n\leq N+1$.

			Now, by \eqref{tnewlambdaeq4} and \eqref{tnewlambdaeq7}, 
			$\lambda q_{2N+4}-p_{2N+4}=\lambda \Phi^{2N+4}.$ Thus, from 
			 \eqref{tnewlambdaeq12} and \eqref{tnewlambdaeq13} we have
			$p''_{N+2}=p_{2N+4}, \quad  q''_{N+2}=q_{2N+4}$ and we get
		 $\Gamma''_{N+2}=\lambda \Phi^{2N+4}$. This proves now \eqref{tnewlambdaeq1a} for  $0 \leq n\leq N+2$.
			
			This completes our proof.
\end{proof}

Let $k'_0=\lambda_1+1,  k'_n=\min\{k\geq 1 : g^{k'_{n-1}}(1)<g^{k}(1)<1\}$ for all $n\geq 1$, $s'_n=1-g^{k'_n}(1),$ and consider the the sequence $S'$ given by
\begin{equation*}\label{S'}
S'=\{s'_n\}_{n\geq 0}.
\end{equation*}
We have 
$k'_n=\min\{k\geq 1 : 1-s'_{n-1}<g^{k}(1-s'_{n-1})<1\}+ k'_{n-1}.$
Also let $k''_1=\lambda_1+2,  k''_n=\min\{k\geq 1 : 1<g^{k}(1)<g^{k''_{n-1}}(1)\}$ for all $n\geq 2$, $s''_0=\lambda$ and $s''_n=g^{k''_n}(1)-1$, for $n \geq 1$. We define another sequence $S''$ as
\begin{equation*}\label{S''}
S''=\{s''_n\}_{n\geq 0},
\end{equation*}
Note that 
$k''_n=\min\{k\geq 1 : 1<g^{k}(s''_{n-1}+1)<s''_{n-1}+1\}+ k''_{n-1}.$
We are interested in studying the bifurcation sets $\Lambda_L$ and $\Lambda_R$ of $g$. By Theorem \ref{prop:1} they are discrete with $0$ as the only possible point of accumulation, hence they can be regarded as decreasing sequences, which we now define. Let the  \emph{right bifurcation sequence} $\Lambda'=\{\Lambda'_n\}_n$ be given by
\begin{equation*}\label{rbs}
\Lambda'_0=\max(\Lambda_R), \quad \Lambda'_n=\max\left(\Lambda_R\backslash \bigcup_{k=0}^{n-1}\Lambda'_k\right), \quad n \geq 1,
\end{equation*}
the \emph{left bifurcation sequence} $\Lambda''=\{\Lambda''_n\}_n$ by
\begin{equation*}\label{lbs}
\Lambda''_0=\max(\Lambda_L), \quad \Lambda''_n=\max\left(\Lambda_L\backslash \bigcup_{k=0}^{n-1}\Lambda''_k\right), \quad n \geq 1,
\end{equation*}
and finally we define recursively the sequence of all bifurcation points of $g$, $\Lambda_n$ (it follows from Theorem \ref{prop:1} that it is equal to the intercalation of $\Lambda'$ and $\Lambda''$)
\begin{equation*}\label{bs}
\Lambda_0=\max(\Lambda'_0,\Lambda''_0), \quad \Lambda_n=\max\left((\Lambda_R\cup\Lambda_L)\backslash \bigcup_{k=0}^{n-1}\Lambda_k\right), \quad n \geq 1.
\end{equation*}

In the next lemma we relate the sequences $s'_n$ and  $s''_n$ with $\Lambda'_n$ and $\Lambda''_n$ for all $n \geq 0$.
\begin{lemma}\label{n1}
The sequences $S', S''$ are equal to $\Lambda', \Lambda''$, respectively. 
\end{lemma}
\begin{proof}
We first prove by induction on $n$, that $s''_n=\Lambda''_n$, for all $n \in \N$. It is clear that $s''_0=\lambda=\Lambda''_0$. Assuming  $s''_n=\Lambda''_n$, we have $k''_n=\min\{k\geq 1 : 1<g^{k}(1)<1+\Lambda''_n\},$
and $n_{\ell}=k''_{n+1}$, for all $\ell \in [s''_{n+1},\Lambda''_n)$. This shows that $s''_{n+1}\geq \Lambda''_{n+1}$. As $g^{k''_{n+1}}(1)=1+s''_{n+1}$, we get that $s''_{n+1}= \Lambda''_{n+1}$. This proves that $S''$ is equal to $\Lambda''$.

We now prove, by induction on $n$, that $s'_n=\Lambda'_n$, for all $n \in \N$. It is clear that $s'_0=1-\lambda_1 \lambda=\Lambda'_0$, where $\lambda_1$ is the first coefficient in the continued fraction expansion of $\lambda$. Assume $s'_n=\Lambda'_n$. Let $\ell$ be a constant such that $s'_{n+1}\leq \ell <s'_n$. Since $s'_n=d^-(k'_{n+1}-1)$, $\ell < d^-(k'_{n+1}-1)$, where $k'_{n+1}-1>k'_{0}-1\geq2$ and with $k=k'_{n+1}$, by Lemma \ref{dlemma}, we have $g^{k'_{n+1}}(1+\ell)=g(g^{k'_{n+1}-1}(1)+\ell)$ for $s'_{n+1}\leq \ell < s'_n$ and since 
 $g^{k'_{n+1}}(1)=1-s'_{n+1}$ and $\lambda<1$, we have  $g^{k'_{n+1}-1}(1)=1-\lambda-s'_{n+1}$. Combining these, we get
\begin{equation}\label{n1eq5}
g^{k'_{n+1}}(1+\ell)=1-s'_{n+1}+\ell.
\end{equation}

By Lemma \ref{dlemma} we have
$g^{k}(1+\ell)=g^{k}(1) + \ell,$
for all $1<k<k_{n+1}'$. Note that  $g^k(1)\notin (g^{k'_n}(1),1)$. Combining these we get
$g^k(1+\ell)\notin (1-s'_n+\ell,1+\ell),$
and since $\lambda$ is irrational and $\ell < s'_n$, this gives
\begin{equation*}\label{n1eq2}
g^k(1+\ell)\notin [1,1+\ell],
\end{equation*}
for all $k<k_{n+1}'$. Thus from \eqref{n1eq5}, we get that $s'_{n+1}$ is the largest value $\ell$ can take such that $g^{k'_{n+1}}(1+\ell)=1$. Since we have $s'_n=\Lambda'_n$, this proves that $s'_{n+1}=\Lambda'_{n+1}$, and so,  that $S'$ is equal to $\Lambda'$.

\end{proof}

In the next theorem we relate  the sequences of errors of upper/lower semiconvergents of $\lambda$ with the right/left bifurcation sequences of $\lambda$.

\begin{theorem}\label{n7}
Assume $\lambda\in (0,1)$ is an irrational number. The sequences $\Lambda '$ and $\Lambda ''$ are equal to $\Gamma '$ and $\Gamma ''$, respectively. Moreover, the associated bifurcation sequence $\Lambda $ is equal to the sequence $\Gamma $ of errors of the semiconvergents of $\lambda$.
\end{theorem}

\begin{proof}
	Let $v:\N \times \N \rightarrow \N$ be given by
	\begin{equation*}\label{eqv'}
	v(m,n)=\left\{\begin{array}{ll}
	\vspace{0.2cm}
	n, & m\leq 1,\\
	\vspace{0.2cm}
	\lambda_2+\lambda_4+...+\lambda_{m}+n, & m > 0 \ \textrm{and} \ m \ \textrm{is\ even},\\
	\lambda_3+\lambda_5+...+\lambda_{m}+n, & m > 1 \ \textrm{and} \ m \ \textrm{is\ odd}.
	\end{array}\right.
	\end{equation*}
	for $m, n \in \N$.
	
	Since $\lambda$ is irrational, this shows that we have
	$$\bigcup_{m \geq 0 \ \textrm{even}} \left\{ v(m,n) \right \}_{0\leq n\leq \lambda_{m+2}}=  \bigcup_{m \geq 1 \ \textrm{odd}} \left\{ v(m,n) \right \}_{0\leq n\leq \lambda_{m+2}} = \N.$$
From this we get that $\Gamma''_k=\Lambda''_k$ and $\Gamma'_k=\Lambda'_k$ for all $k \in \N$ if and only if for all even $m \geq 0$
	\begin{equation*}\label{np2}
	\Gamma''_{v(m,n)}=\Lambda''_{v(m,n)},
	\end{equation*}
	for $0\leq n\leq \lambda_{m+2}$, and
	\begin{equation*}\label{np3}
	\Gamma'_{v(m+1,n)}=\Lambda'_{v(m+1,n)},
	\end{equation*}
	for $0\leq n\leq \lambda_{m+3}$.

	It is well known  (see for instance \cite{Khi}) that $[0,\lambda_1,...,\lambda_m]= (n p_{m}+p_{m-1})/(n q_{m}+q_{m-1}),$
for all $m,n \in \N$. From this it follows that for all even $m\geq0$ we have
	\begin{equation*}\label{np4}
	\Gamma''_{v(m,n)}=(q_{m}\lambda-p_{m})-n(p_{m+1}-q_{m+1}\lambda),
	\end{equation*}
	for $0\leq n\leq \lambda_{m+2}$, and
	\begin{equation*}\label{np5}
	\Gamma'_{v(m+1,n)}=(p_{m+1}-q_{m+1}\lambda)-n(q_{m+2}\lambda-p_{m+2}),
	\end{equation*}
	for $0\leq n\leq \lambda_{m+3}$.
	Combining the four expressions above it follows that $\Gamma''_k=\Lambda''_k$ and $\Gamma'_k=\Lambda'_k$ for all $k \in \N$ if and only if for all even $m\geq 0$ we have
	\begin{equation}\label{np6}
	\Lambda''_{v(m,n)}=(q_{m}\lambda-p_{m})-n(p_{m+1}-q_{m+1}\lambda),
	\end{equation}
	for $0\leq n\leq \lambda_{m+2}$, and
	\begin{equation}\label{np7}
	\Lambda'_{v(m+1,n)}=(p_{m+1}-q_{m+1}\lambda)-n(q_{m+2}\lambda-p_{m+2}),
	\end{equation}
	for $0\leq n\leq \lambda_{m+3}$.\\
	
	
	We now prove, by induction on $m$, that \eqref{np6} and \eqref{np7}  hold for all even $m\geq0$.
	Before, we prove by induction on $n$, that 
	\begin{equation}\label{np8}
	\Lambda''_{v(0,n)}=(q_{0}\lambda-p_{0})-n(p_{1}-q_{1}\lambda),
	\end{equation}
	for $0\leq n\leq \lambda_{2}$.
	
	We have $v(0,0)=0$, thus $\Lambda''_{v(0,0)}=\Lambda''_0$. Since $s''_0=\lambda$ and $(p_0,q_0)=(0,1)$, by Lemma \ref{n1} we have $\Lambda''_0=q_{0}\lambda-p_{0}$. Thus, 
	\eqref{np8} holds for $n=0$. 
	Now fix $n< \lambda_{m+2}$. We assume \eqref{np8} holds for $n$ and prove it for $n+1$ instead.
	
	We have that \eqref{np8} is equivalent to
$g^{1+n(1+\lambda_1)}(1)-1=\lambda-n(1-\lambda_1 \lambda),$
	therefore since we are assuming \eqref{np8} holds for $n$  we get
	\begin{equation}\label{np9}
	g^{1+(n+1)(1+\lambda_1)}(1)-1=g(g^{\lambda_1}(1+\Lambda''_{v(0,n)})).
	\end{equation}

	Recall the definition of $d^-$. We have
$d^-(N)=1-\max_{1\leq k \leq N}\left\{g^k(1)\leq 1\right\},$
	for any $N \geq 2$. Note that 
	$d^-(\lambda_1)=1-(\lambda_1-1)\lambda,$
	therefore $\Lambda_{v(0,n)}''\leq d^-(\lambda_1)$.

	Assume now that $\lambda_1\geq 2$. Applying Lemma \ref{dlemma} with $\ell=\Lambda''_{v(0,n)}$ and $N=\lambda_1$ we get
	$g^{\lambda_1}(1+\Lambda''_{v(0,n)})=g^{\lambda_1}(1)+\Lambda''_{v(0,n)}.$
	Since $1-g^{\lambda_1+1}(1)=s_0'$, $s_0'=1-\lambda_1 \lambda$ and $\lambda<1$ we have
	$g^{\lambda_1}(1)=(\lambda_1-1)\lambda.$
	Combining this we get
	\begin{equation*}\label{np12}
	g^{\lambda_1}(1+\Lambda''_{v(0,n)})=1-\lambda + (\lambda_1-(n+1)(1-\lambda_1 \lambda)),
	\end{equation*}
	which combined with \eqref{np9} gives
	\begin{equation}\label{np13}
	g^{1+(n+1)(1+\lambda_1)}(1)=1+(\lambda_1-(n+1)(1-\lambda_1 \lambda)),
	\end{equation}
	which is smaller than $1+\Lambda''_{v(0,n)}$.

	If $\lambda_1=1$ it is clear from \eqref{np9} that we get \eqref{np13} as well.
	Since $\Gamma''$ is the sequence of best rational approximations of $\lambda$ by below and we have $\Lambda_{v(0,n)}''=\Gamma_{v(0,n)}''$ and $g^{1+(n+1)(1+\lambda_1)}(1)-1=\Gamma_{v(0,n)+1}''$, we must have
	$$ 1+(n+1)(1+\lambda_1) = \min\{k\geq 1 : 1<g^{k}(1)<g^{1+n(1+\lambda_1)}(1)\}, $$
	and thus by Lemma \ref{n1} and \eqref{np13} we have	
	$\Lambda_{v(0,n+1)}''=(q_{0}\lambda-p_{0})-(n+1)(p_{1}-q_{1}\lambda).$
	This completes the proof that \eqref{np8} holds for $0 \leq n \leq \lambda_{m+2}$.\\
	
	
	In a similar way, it can be proved that
	\begin{equation*}\label{np14}
	\Lambda'_{v(1,n)}=(p_{1}-q_{1}\lambda)-n(q_{2}\lambda-p_{2}),
	\end{equation*}
	for all $0\leq n\leq \lambda_{3}$, so we omit the proof.
	

	We now assume that for $0\leq n \leq \lambda_{m+2}$, we have
	\begin{equation}\label{Lr'}
	\Lambda_{v(m,n)}''=(q_m \lambda - p_m)-n(p_{m+1}-q_{m+1}\lambda),
	\end{equation}
	and that for $0\leq n \leq \lambda_{m+3}$, we have
	\begin{equation}\label{Lr''}
	\Lambda_{v(m+1,n)}'=(p_{m+1}-q_{m+1} \lambda)-n(q_{m+2}\lambda-p_{m+2}).
	\end{equation}
	and prove that we have 
	\begin{equation}\label{Lr}
	\Lambda_{v(m+2,n)}''=(q_{m+2} \lambda-p_{m+2})-n(p_{m+3}-q_{m+3}\lambda),
	\end{equation}
	for all $0\leq n\leq \lambda_{m+4}$, and 
	\begin{equation}\label{Lr''2}
	\Lambda_{v(m+3,n)}'=(p_{m+3}-q_{m+3} \lambda)-n(q_{m+4}\lambda-p_{m+4}).
	\end{equation}
	for all $0\leq n\leq \lambda_{m+5}$.

First we prove \eqref{Lr}, by induction on $n$, for all $n\leq \lambda_{m+4}$ .

Since $v(m+2,0)=v(m,\lambda_{m+2})$,  by \eqref{Lr'} and \eqref{CF}, we get
$\Lambda_{v(m+2,0)}''=q_{m+2} \lambda - p_{m+2}.$ Thus, \eqref{Lr} holds for $n=0$.

Fix $n< \lambda_{m+4}$. We assume that \eqref{Lr} holds for $n$ and prove it for $n+1$ instead.
Recall the definition of $\Lambda''$. With
$K''(n)=p_{m+2}+q_{m+2}+n(p_{m+3}+q_{m+3}),$
we have that \eqref{Lr} is equivalent to
$g^{K''(n)}(1)-1=q_{m+2}\lambda-p_{m+2}-n(p_{m+3}-q_{m+3}\lambda),$
and combining these  we get
\begin{equation}\label{new2b}
g^{K''(n+1)}(1)=g(g^{(p_{m+3}+q_{m+3}-1)}(1+\Lambda''_{v(m+2,n)})).
\end{equation}

From \eqref{Lr''} we get
\begin{equation*}\label{new3b}
\Lambda_{v(m+1,n)}'=\Gamma_{v(m+1,n)}',
\end{equation*}
for $0\leq n\leq \lambda_{m+3}$. It follows from this identity and fact that the upper semi-convergents of $\lambda$ are its best rational approximations by above, that $v(m+1,\lambda_{m+3}-1)$ is the largest integer such that $k'_{v(m+1,\lambda_{m+3}-1)}<q_{m+3}+p_{m+3}$. 
	Recall the definition of $d^-$. We have
$d^-(N)=1-\max_{1\leq k \leq N}\left\{g^k(1)\leq 1\right\},$
	for any $N \geq 2$, thus
$d^-(q_{m+3}+p_{m+3}-1)=s'_{\lambda_{m+3}-1},$
and by Lemma \ref{n1}, \eqref{Lr''} and \eqref{CF} we get
\begin{equation*}\label{new5b}
d^-(q_{m+3}+p_{m+3}-1)=p_{m+3}-q_{m+3}\lambda + q_{m+2}\lambda-p_{m+2}.
\end{equation*}

Therefore $\Lambda_{v(m+2,n)}''< d^-(q_{m+3}+p_{m+3}-1)$. Applying Lemma \ref{dlemma} with $\ell=\Lambda''_{v(m+2,n)}$ and $N=p_{m+3}+q_{m+3}-1$ yields
\begin{equation}\label{new6b}
g^{p_{m+3}+q_{m+3}-1}(1+\Lambda_{v(m+2,n)}'')=g^{p_{m+3}+q_{m+3}-1}(1)+\Lambda_{v(m+2,n)}''.
\end{equation}

By Lemma \ref{n1}, \eqref{Lr''} and \eqref{CF} we have
$1-g^{p_{m+3}+q_{m+3}}(1)=p_{m+3}-q_{m+3}\lambda,$
and since  $\lambda<1$,  we get
\begin{equation*}\label{new8b}
g^{p_{m+3}+q_{m+3}}(1)=1-\lambda-(p_{m+3}-q_{m+3}\lambda).
\end{equation*}
Combining this identity with \eqref{new2b} and \eqref{new6b} we have
\begin{equation*}\label{new9b}
g^{K''(n+1)}(1)=g(1-\lambda-(p_{m+3}-q_{m+3}\lambda)+\Lambda''_{v(m+2,n)}),
\end{equation*}
and since \eqref{Lr} holds for $n$ we get
\begin{equation}\label{new10b}
g^{K''(n+1)}(1)=1+(q_{m+2}\lambda-p_{m+2})-(n+1)(p_{m+3}-q_{m+3}\lambda),
\end{equation}
which is smaller than $1+\Lambda''_{v(m+2,n)}$.

Since $\Gamma''$ is the sequence of best rational approximations of $\lambda$ by below and we have $\Lambda_{v(m+2,n)}''=\Gamma_{v(m+2,n)}''$ and $g^{K(n+1)}(1)-1=\Gamma_{v(m+2,n)+1}''$, we must have
$$ K''(n+1) = \min\{k\geq 1 : 1<g^{k}(1)<g^{K''(n)}(1)\}, $$
and thus by Lemma \ref{n1} and \eqref{new10b} we have
\begin{equation*}
\Lambda_{v(m+2,n+1)}''=(q_{m+2}\lambda-p_{m+2})-(n+1)(p_{m+3}-q_{m+3}\lambda).
\end{equation*}
This completes the proof that \eqref{Lr} holds for $0 \leq n \leq \lambda_{m+4}$.\\


We now prove \eqref{Lr''2} by induction on $n$, for all $n\leq \lambda_{m+5}$ .

Since $v(m+3,0)=v(m+1,\lambda_{m+3})$,  by \eqref{CF} and \eqref{Lr''}, we get
$\Lambda_{v(m+3,0)}' =p_{m+3}-q_{m+3} \lambda$ and so \eqref{Lr''2} holds for $n=0$. 

Fix $n< \lambda_{m+5}$. We assume that \eqref{Lr''2} holds for $n$ and prove it for $n+1$ instead.

Recall the definition of $\Lambda'$. With
$K'(n)=p_{m+3}+q_{m+3}+n(p_{m+4}+q_{m+4}),$
we have that \eqref{Lr''2} is equivalent to
$1-g^{K'(n)}(1)=p_{m+3}-q_{m+3}\lambda-n(q_{m+4}\lambda-p_{m+4}),$
and  we get
\begin{equation}\label{new2}
g^{K'(n+1)}(1)=g(g^{(p_{m+4}+q_{m+4}-1)}(1-\Lambda'_{v(m+3,n)})).
\end{equation}
From \eqref{Lr} we get
$
\Lambda_{v(m+2,n)}''=\Gamma_{v(m+2,n)}'',
$ 
for $0\leq n\leq \lambda_{m+4}$. 
It follows from this identity and from the fact that the lower semi-convergents of $\lambda$ are its best rational approximations by below, that $v(m+2,\lambda_{m+4}-1)$ is the largest integer such that $k''_{v(m+2,\lambda_{m+4}-1)}<q_{m+4}+p_{m+4}$.

Recall the definition of $d^+$. We have
$d^+(N)=\min_{1\leq k \leq N}\left\{g^k(1)\geq 1\right\}-1,$
for any $N \geq \lambda_1+1$. Thus, $d^+(q_{m+4}+p_{m+4}-1)=s''_{\lambda_{m+4}-1}.$
By Lemma \ref{n1}, \eqref{Lr} and \eqref{CF} we get
$d^+(q_{m+4}+p_{m+4}-1)=q_{m+4}\lambda-p_{m+4} + p_{m+3}-q_{m+3}\lambda.$
Therefore $\Lambda_{v(m+3,n)}'< d^+(q_{m+4}+p_{m+4}-1)$. Applying Lemma \ref{dlemma} with $\ell=1-\Lambda'_{v(m+3,n)}$ and $N=p_{m+4}+q_{m+4}-1$ we get
\begin{equation*}\label{new6}
g^{p_{m+4}+q_{m+4}-1}(1-\Lambda_{v(m+3,n)}')=g^{p_{m+4}+q_{m+4}-1}(1)-\Lambda_{v(m+3,n)}'.
\end{equation*}

By Lemma \ref{n1}, \eqref{Lr} and \eqref{CF} we have
$g^{p_{m+4}+q_{m+4}}(1)-1=q_{m+4}\lambda-p_{m+4},$
and since $\lambda<1$ we get
\begin{equation*}\label{new8}
g^{p_{m+4}+q_{m+4}}(1)=1-\lambda+(q_{m+4}\lambda-p_{m+4}).
\end{equation*}

Combining the two above identities with \eqref{new2} and since \eqref{Lr''2} holds for $n$ we get
\begin{equation}\label{new10}
g^{K'(n+1)}(1)=1-\left[(p_{m+3}-q_{m+3 \lambda})-(n+1)(q_{m+4}\lambda-p_{m+4})\right],
\end{equation}
which is larger than $1-\Lambda'_{v(m+3,n)}$.

Since $\Gamma'$ is the sequence of best rational approximations of $\lambda$ by above and we have $\Lambda_{v(m+3,n)}'=\Gamma_{v(m+3,n)}'$ and $1-g^{K'(n+1)}(1)=\Gamma_{v(m+3,n)+1}'$, we must have
$$ K'(n+1) = \min\{k\geq 1 : g^{K'(n)}(1)<g^{k}(1)<1\}, $$
and thus by Lemma \ref{n1} and \eqref{new10} we have
\begin{equation*}
\Lambda_{v(m+3,n+1)}'=(p_{m+3}-q_{m+3 \lambda})-(n+1)(q_{m+4}\lambda-p_{m+4}).
\end{equation*}
This completes the proof that \eqref{Lr''2} holds for $0 \leq n \leq \lambda_{m+5}$.


This proves \eqref{np6} and \eqref{np7} and therefore $\Lambda '$ is equal to the sequence $\Gamma '$ and $\Lambda ''$ is equal to the sequence $\Gamma ''$. By definition of $\Lambda$ and $\Gamma $, this implies that $\Gamma =\Lambda$ as well. This finishes our proof.

\end{proof}

Recall our definitions of first hitting time $n_{\ell}(x)$ of $x$ to $\overline{I_c(\ell)}$ and the first hitting map $r_{\ell}$. We want now to relate these to $\Gamma'$ and $\Gamma''$. This is done in the next theorem.

\begin{theorem}\label{n9}
 Let $0<\ell\leq \lambda$ and let $n_1, n_2 \in \N$ be such that $\Gamma'_{n_1+1}\leq \ell < \Gamma'_{n_1} $  and  $\Gamma''_{n_2+1}\leq \ell < \Gamma''_{n_2} $. Then
 $r_{\ell}(1+\ell)=1+\ell-\Gamma'_{n_1+1}$ and $r_{\ell}(1)=1+\Gamma''_{n_2+1}.$ Furthermore $n_{\ell}(1+\ell)=k_{n_1+1}'$ and $n_{\ell}(1)=k_{n_2+1}''$.
\end{theorem}

\begin{proof}
	We prove only that $r_{\ell}(1+\ell)=1+\ell-\Gamma'_{n_1+1}$. As the proof for the other case is similar, we omit it.
	
	By Theorem \ref{n7}, we have $\Gamma'=\Lambda'$, therefore $\Gamma'_{n_1+1}\leq \ell < \Gamma'_{n_1} $ implies that $\Lambda'_{n_1+1}\leq \ell < \Lambda'_{n_1} $. Also, combining Theorem \ref{n7} with Lemma \ref{n1}, we have  $S'=\Gamma'$ and we get
	\begin{equation}\label{n9e1}
	g^{k_{n_1+1}'}(1)=1-\Gamma'_{n_1+1}.
	\end{equation}
	As $k_{n_1+1}> 2$ and $\Gamma'_{n_1+1}\leq d^{-}(k_{n_1+1}')$,  applying Lemma \ref{dlemma} we get
	\begin{equation*}\label{n9e2}
	g^{k_{n_1+1}'}(1+\Gamma'_{n_1+1})=g^{k_{n_1+1}'}(1)+\Gamma'_{n_1+1}.
	\end{equation*}
	From these two identities we get $g^{k_{n_1+1}'}(1+\Gamma_{n_1+1}')=1$, thus $n_{\Gamma_{n_1+1}'}(1+\Gamma_{n_1+1}')=k_{n_1+1}'.$ Therefore $r_{\ell}(1+\ell)=g^{k_{n_1+1}'}(1+\ell).$
	
	As $\ell < d^-(k_{n_1+1}-1)$ by Lemma \ref{dlemma},
	$g^{k_{n_1+1}'-1}(1+\ell)=g^{k_{n_1+1}'-1}(1)+\ell.$
	Applying $g$ on both sides  and combining with \eqref{n9e1}, we get $g^{k_{n_1+1}'}(1+\ell)=1+\ell-\Gamma_{n_1+1}'.$ This finishes our proof.
\end{proof}

\section{Dynamical sequences}\label{Dynamical sequences}

In this section, we introduce the {\it dynamical sequences} $\{y_n \}_n$ and $\{p_n \}_n$. These will be an important tool in order to prove our main theorems and will be later related to the dynamics of our family of maps. We show inductive formulas to compute these sequences and prove that for some choice of parameters  $\{p_n \}_n$ is periodic with period at most $2$.
 
Let
\begin{equation}\label{ren121}
\ell(y)=\dfrac{2y}{\nu},
\end{equation}
and denote
\begin{equation}\label{CandD}
C= C(\mu,\nu) =\dfrac{2 \mu}{\mu+\nu}, \quad D= D(\mu,\nu)=\dfrac{2 \mu}{\mu - \nu}.
\end{equation}
We now inductively define our sequences $\{y_{n}\}_{n \in \N}$, $\{p_{n}\}_{n \in \N}$ and $\{\kappa_{n}\}_{n \in \N}$  depending on the parameters $\nu>0$, $|\mu|>\nu$, $\lambda \in (0,1)\backslash\mathbb{Q}$ and $0<\eta<\lambda$. We will denote these sequences  by  $\{y_{n}(\mu)\}_{n \in \N}$, $\{p_{n}(\mu)\}_{n \in \N}$ and $\{\kappa_{n}(\mu)\}_{n \in \N}$ when it is important to stress the dependence on the parameter $\mu$.

Set
\begin{equation}\label{ren66}
y_0=\eta\frac{\mu \nu}{\mu+\nu}.
\end{equation}
Note that as $\eta>0$ we have $y_0>0$.  Since $\lambda$ is irrational, $\{ \Gamma_n''  \}$ (see Section \ref{Bifurcation sequence}) is an infinite sequence and converges to $0$. Furthermore by the definitions of $\ell$ and $y_0$   we have $\ell(y_0)>0$. Thus there exists a smallest natural number $\kappa_0$ such that $\Gamma''_{\kappa_0}<\ell(y_0)$.
Set $\Upsilon_0=\Gamma''_{\kappa_0}$ and 
\begin{equation*}\label{ren3}
	p_0=\dfrac{\Upsilon_0}{\ell(y_0)}.
\end{equation*}

For $n \geq 0$ assume we defined $y_n$, $p_n$, $\kappa_n$ and $\Upsilon_n$ and that at least one of the conditions $y_n>0$ or $p_n =1/C$ holds. Set
 \begin{equation}\label{fdren5}
y_{n+1}= \left\{\begin{array}{ll}
\vspace{0.2cm}
(1-C p_{n})y_{n}, & \textrm{if} \ p_{n}<1/C, \\
\vspace{0.2cm}
(1-D(1-p_{n}))y_{n}, &  \textrm{if} \ p_{n}>1/C, \\
0, & \textrm{if} \ p_{n}=1/C .
\end{array}\right. 
\end{equation}

Since $\lambda$ is irrational, $\{ \Gamma_n'  \}$ and $\{ \Gamma_n''  \}$ (see Section \ref{Bifurcation sequence}) are infinite sequences and converge to $0$, furthermore if $p_{n}\neq 1/C$, by \eqref{CandD} and \eqref{fdren5}  we have $\ell(y_{n+1})>0$, thus there are integers $k'$ and $k''$ such that  $\Gamma_k'<\ell(y_{n+1})$ and $\Gamma_k''<\ell(y_{n+1})$ respectively. We set
\begin{equation*}\label{fdren5a}
\kappa_{n+1}= \left\{\begin{array}{ll}
\vspace{0.2cm}
\min\{ k \in \N: \Gamma_k''<\ell(y_{n+1})  \}, & \textrm{if} \ p_{n}<1/C, \\
\vspace{0.2cm}
\min\{ k \in \N: \Gamma_k'<\ell(y_{n+1})  \}, &  \textrm{if} \ p_{n}>1/C, \\
\kappa_{n}, & \textrm{if} \ p_{n}=1/C .
\end{array}\right. 
\end{equation*}
If $p_n \leq 1/C$ set $\Upsilon_{n+1}=\Gamma''_{\kappa_{n+1}}$, else if $p_n > 1/C$ set
\begin{equation}\label{ren1}
\Upsilon_{n+1}= \left\{\begin{array}{ll}
\vspace{0.2cm}
1, & \textrm{if} \ \ell(y_{n+1})>1, \\
\vspace{0.2cm}
1-\left( 1+ \left[ \frac{1-\ell(y_{n+1})}{\lambda}  \right]     \right)\lambda, &  \textrm{if} \ \Gamma_0'< \ell(y_{n+1}) \leq 1, \\
\Gamma_{\kappa_{n+1}}', & \textrm{if} \ \ell(y_{n+1})< \Gamma_0',
\end{array}\right. 
\end{equation}
where $\left[ \cdot  \right]$, denotes the integer part of a real number. Finally set
\begin{equation*}\label{fd4}
p_{n+1}= \left\{\begin{array}{ll}
\vspace{0.2cm}
\dfrac{	\Upsilon_{n+1}}{\ell(y_{n+1})}, & \textrm{if} \ p_{n}<1/C, \\
\vspace{0.2cm}
1-\dfrac{	\Upsilon_{n+1}}{\ell(y_{n+1})}, & \textrm{if} \ p_{n}>1/C, \\
0, & \textrm{if} \ p_{n}=1/C.
\end{array}\right. 
\end{equation*}

The following lemma characterizes the sequence $\{y_n\}_{n \in \mathbf{N}}$.

\begin{lemma}\label{lemmayn}
Given $\nu>0$, $|\mu|>\nu$, $\lambda \in (0,1)\backslash\mathbb{Q}$ and $0<\eta<\lambda$, the sequence $\{y_n\}_{n \in \mathbf{N}}$ with $\mathbf{N}=\{n \in \N:  y_n > 0  \}$ is strictly decreasing and it is either finite or converges to $0$. 
\end{lemma}

\begin{proof}
If $p_n <1/C$ or $p_n >1/C$ then $(1-C p_n) \in (0,1)$ or $(1-D ( 1-p_n)) \in (0,1)$, respectively, and by \eqref{fdren5}, $y_{n+1}<y_n$. If $p_n =1/C$ then, by definition, $y_{n+1}=0$ and thus $n+1 \notin \mathbf{N}$. This shows that $\{y_n\}_{n \in \mathbf{N}}$ is strictly decreasing and either $\mathbf{N}$ is finite or for all $n \in \N$ we have $y_{n+1}<y_n$.
	
	We now show that if $y_{n}>0$ we have $y_n \rightarrow 0$. Assume by contradiction that $\{y_n\}$ does not converge to $0$. Since it is strictly decreasing there must exist $y'>0$ such that $y_n \rightarrow y'$.
	Since for all $n \in \N$, $p_n \neq 1/C$, we must have that either $p_n<1/C$ or $p_n>1/C$ for infinitely many values of $n \in \N$.
	
	Assume the first case holds. Then there is a subsequence $\{ p_{n(l)}  \}_{l \in \N}$ such that $p_{n(l)}<1/C$ for all $l \in \N$.
	Since $y_{n}\rightarrow y'$ we have in particular that
	$$\lim_{l\rightarrow +\infty} y_{n(l)+1}=\lim_{l\rightarrow +\infty} y_{n(l)}=y'$$
	and by \eqref{fdren5}
	$y_{n(l)+1}=(1-C p_{n(l)}) y_{n(l)},$
	thus, we must have $1-C p_{n(l)} \rightarrow 1$ and therefore $p_{n(l)} \rightarrow 0$. Hence by the definition of $p_n$ and since $\ell(y_{n(l)})\rightarrow \ell(y')$, we have
	$
	\Gamma''_{\kappa_{n(l)}+1}\rightarrow 0.
	$
	
	Since $\lambda$ is irrational, $\{ \Gamma_n''  \}$ is an infinite sequence and converges to $0$. Thus there exists  an unique natural number $k'$ such that
	$ \Gamma''_{k'+1}<\ell(y')\leq \Gamma''_{k'},  $
	and by the definition of $\kappa_{n}$ we must have $\kappa_{n(l)}\rightarrow k'$. Thus we get $\Gamma''_{k'+1}=0$, which implies that $\lambda$ is rational which is a contradiction.
	
	The proof is analogous if $p_n>1/C$ for infinitely many values of $n \in \N$, hence we omit it.
\end{proof}

Given $n \in \N$ and $x \in \R$ we introduce the following maps:
\begin{equation*}\label{ren9}
\chi_n(x)=\left\{\begin{array}{ll}
\vspace{0.2cm}
x & , \ \textrm{if} \ p_{n}<1/C, \\
\vspace{0.2cm}
1-x & , \ \textrm{if} \ p_{n}>1/C, \\
1 & , \ \textrm{if} \ p_{n}=1/C,
\end{array}\right.
\quad \textrm{and} \quad
\omega_n(x)=\left\{\begin{array}{ll}
\vspace{0.2cm}
C & , \ \textrm{if} \ p_{n}<1/C, \\
\vspace{0.2cm}
D & , \ \textrm{if} \ p_{n}>1/C, \\
1 & , \ \textrm{if} \ p_{n}=1/C.
\end{array}\right.
\end{equation*}

The following lemma (we omit the proof) gives recursive expressions for $y_n$ and $p_n$ which can be obtained directly from the definitions of $y_n$ and $p_n$.

\begin{lemma}\label{propyN}
Given $\nu>0$ and $\mu \in \R$ satisfying $|\mu|>\nu$, for all $n \in \mathbf{N}\backslash\{0\}$ we have
\begin{equation*}\label{ren11}
y_n(\mu)=(1-\omega_{n-1}\chi_{n-1}(p_{n-1}(\mu)))y_{n-1}(\mu).
\end{equation*}
Moreover if $p_{n-1}(\mu)\neq 1/C$, we have
\begin{equation*}\label{ren12}
\chi_{n-1}(p_{n}(\mu))=\dfrac{\Upsilon_n(\mu)}{\Upsilon_{n-1}(\mu)}\dfrac{\chi_{n-2}(p_{n-1}(\mu))}{1-\omega_{n-1}\chi_{n-1}(p_{n-1}(\mu))}.
\end{equation*}
\end{lemma}

Next theorem provides, under some conditions on $\lambda$ and $\eta$, a closed form expression for the sequence $\{p_n\}$ and shows that it is periodic with period at most 2. We will denote
\begin{equation*}\label{ren22}
\bar{\mu}=\frac{\nu}{\Phi^3}.
\end{equation*}

\begin{theorem}\label{rencor}
	Assume $\nu>0$, $\lambda= 1/(k+\Phi)$ and $\eta=1-k\lambda$ with $k \in \N$. Let $\mu \in {\mathbb R}$ be such that  $|\mu|\geq \nu$. 
	If $|\mu|> \bar{\mu}$, then $p_{n}(\mu)=p_{n+2}(\mu)$ for all $n\geq 0$, in particular
	\begin{equation}\label{pep3}
	p_n(\mu)=\dfrac{1}{C\Phi} \ \textrm{and} \ \ell(y_n(\mu))=C \lambda \Phi^{n+1}, \ \textrm{if} \ n \ \textrm{is even},
	\end{equation}
	and
	\begin{equation}\label{pep4}
	p_n(\mu)=1-\frac{1}{D \Phi} \ \textrm{and} \ \ell(y_{n}(\mu))= D \lambda \Phi^{n+1}, \ \textrm{if} \ n \ \textrm{is odd}.
	\end{equation}
	If $|\mu|\leq \bar{\mu}$, then $p_{n}(\mu)=p_{n+1}(\mu)$ for all $n\geq 1$. In particular, if $-\bar{\mu} < \mu < -\nu$, then for all $n \geq 1$, 
	\begin{equation}\label{pep15}
	p_n(\mu)=1-\dfrac{\Phi}{D} \ \textrm{and} \ \ell(y_n(\mu))=D \lambda \Phi^{2n}.
	\end{equation}
	If $\nu < \mu < \bar{\mu}$, then for all $n \geq 0$, 
	\begin{equation}\label{pep21}
	p_n(\mu)=\dfrac{\Phi}{C} \ \textrm{and} \ \ell(y_n(\mu))=C \lambda \Phi^{2n+1}.
	\end{equation}
\end{theorem}

\begin{proof}

Let us first investigate $C$ and $D$ as in \eqref{CandD}. It is clear that
\begin{equation*}
	\dfrac{1}{\Phi} < 2 <  D=\dfrac{2\mu}{\mu - \nu} <  \dfrac{2}{1-\Phi^3} = \dfrac{1}{\Phi^2},
\end{equation*}
where we used the fact that $ \partial  D(\mu,\nu) / \partial \mu =-\nu /(\mu-\nu)<0,$
as long as $\nu>0$ and also $\Phi^2=1-\Phi$. Now, if $\mu < -\bar{\mu} < -\nu<0$ we have $1<D< 2 < 1/\Phi^2$.
Since $\mu < -\bar{\mu}$ we get
\begin{equation*}
	\mu < -\bar{\mu}=-\frac{\nu}{\Phi^3} = -\frac{\nu}{2\Phi-1},
\end{equation*}
which is equivalent to $2\mu \Phi < \mu-\nu$, and since $\mu <0$ we get
	$D>\frac{2\mu}{(\mu-\nu)} > \frac{1}{\Phi}.$

Since $C$ and $D$ are H{\"o}lder conjugate we have
$\frac{1}{1-\Phi^2} < C < \frac{1}{1-\Phi}.$
Combining this  we get 
 that if $|\mu|>\bar{\mu}$, then
\begin{equation}\label{ren30}
\frac{1}{\Phi} < C < \frac{1}{\Phi^2}, \quad \frac{1}{\Phi} < D < \frac{1}{\Phi^2}.
\end{equation}


We now prove by induction on $n$ that if $|\mu|>\bar{\mu}$, we have \eqref{pep3} and \eqref{pep4} for all $n \geq 0$.

From \eqref{ren66} we have
$\ell(y_0(\mu))=2\mu(1-k\lambda)/(\mu+\nu).$
Since $\lambda \Phi=1-k\lambda$ we get $\ell(y_0(\mu))=C \lambda \Phi$.
For $|\mu|>\bar{\mu}$ we get from \eqref{ren30} that
\begin{equation*}\label{pep1}
\lambda<\ell(y_0(\mu))<\lambda/\Phi \leq 1.
\end{equation*}
Hence we have that $\kappa_0=0$, and by Lemma \ref{tnewlambda} and \eqref{ren1} we get $\Upsilon_0(\mu)=\Gamma_0$ and by definition of $p_0$ we have
\begin{equation*}
	p_0=\frac{\Gamma_0}{\ell(y_0(\mu))}=\frac{\Gamma_0 \nu}{2 y_0}=\frac{\Gamma_0}{2 \mu \eta}(\mu+\nu)=\frac{\lambda}{C \eta} = \frac{1}{C \Phi}.
\end{equation*}
Thus we get \eqref{pep3} for $n=0$.

Since $\Phi <1$ and \eqref{pep3} holds for $n=0$, we have $p_0(\mu)>1/C$, hence by Lemma \ref{propyN} we have 
\begin{equation}\label{pep7}
\ell(y_1(\mu))=(1-D(1-p_0(\mu)))\ell(y_0(\mu)).
\end{equation}
Simple computations show that
\begin{equation}\label{pep5}
1-D(1-p_0(\mu))=1-D(1-\frac{\lambda}{C \eta})=\frac{D}{C}\Phi,
\end{equation}
where we used H{\"o}lder conjugacy of $C$ and $D$ several times to simplify the expression. By \eqref{pep5} and \eqref{pep7} we have $\ell(y_1(\mu))=\dfrac{D}{C}\Phi \ell(y_0(\mu))$ and since \eqref{pep3} holds for $n=0$, we get
\begin{equation}\label{pep6}
\ell(y_1(\mu))=D \lambda \Phi^2.
\end{equation}

Now,  by Lemma \ref{propyN} we have
\begin{equation}\label{pep8}
p_1(\mu)=1-\frac{\Upsilon_1(\mu)}{\lambda}\frac{p_0(\mu)}{1-D(1-p_0(\mu))},
\end{equation}
and by \eqref{pep6} and \eqref{ren30}, since $|\mu|>\bar{\mu}$, we get
$\lambda \Phi < \ell(y_1(\mu)) < \lambda$. This together with Lemma \ref{tnewlambda} and \eqref{ren1} shows that $\Upsilon_1(\mu)=\lambda \Phi$, and from \eqref{pep5} and \eqref{pep8} we get
$p_1(\mu) =1-1/(D\Phi).$
Together with \eqref{pep6} this shows \eqref{pep4} holds for $n=1$.

Let $n\geq0$ be an even number. We now assume that \eqref{pep3} holds for $n$, \eqref{pep4} holds for $n+1$ and prove that \eqref{pep3} holds for $n+2$ and \eqref{pep4} holds for $n+3$.

Note that since we assume \eqref{pep3} holds for $n$ and \eqref{pep4}  for $n+1$, by \eqref{ren30} we have
\begin{equation}\label{pep11}
\begin{array}{ll}
\lambda \Phi^{n} < \ell(y_n(\mu)) < \lambda \Phi^{n-1}, & \lambda \Phi^{n+1} < \ell(y_{n+1}(\mu)) < \lambda \Phi^{n}.
\end{array}
\end{equation}

Since $1/\Phi>1$ we have
$p_1(\mu)=1-\frac{1}{d\Phi} < 1-\frac{1}{D}=\frac{1}{C},$
thus $p_{n+1}(\mu)<1/C$, since we also have $p_n(\mu)>1/C$ we get from Lemma \ref{propyN} that
\begin{equation}\label{pep9}
\ell(y_{n+2}(\mu))=(1-C p_1(\mu))(1-D(1-p_0(\mu)))\ell(y_n(\mu)).
\end{equation}
After a simple computation we have
\begin{equation}\label{ren35}
1-C p_1(\mu)=\frac{\Phi}{d-1}=\frac{C}{D}\Phi.
\end{equation}
Combining \eqref{pep5}, \eqref{pep9} and \eqref{ren35} we get
\begin{equation}\label{pep10}
\ell(y_{n+2}(\mu))=\Phi^2\ell(y_n(\mu)).
\end{equation}
Since we assume \eqref{pep3} for $n$, we get  from \eqref{pep10} that
\begin{equation}\label{ren39}
\lambda \Phi^{n+2} < \ell(y_{n+2}(\mu)) < \lambda \Phi^{n+1}.
\end{equation}

We now prove that \eqref{pep3} holds for $n+2$. 
Since \eqref{pep3} holds for $n$, from \eqref{pep10} we get $\ell(y_{n+2}(\mu))=C \lambda \Phi^{n+3}$.
To see that $p_{n+2}(\mu)=p_0(\mu)$ note that by Lemma \ref{tnewlambda}, by the definition of $\Upsilon_{n+1}$ and by \eqref{pep11} and \eqref{ren39}, we have $\Upsilon_{n+1}(\mu)=\Gamma_{n+1}$ and $\Upsilon_{n+2}(\mu)=\Gamma_{n+2}$. Together with $p_n(\mu)>1/C$ and $p_{n+1}(\mu)<1/C$, Lemma \ref{propyN} implies
\begin{equation*}
	p_{n+2}(\mu)=\frac{\Gamma_{n+2}}{\Gamma_{n+1}}\frac{1-p_{n+1}(\mu)}{1-C p_{n+1}(\mu)}=\Phi \frac{\frac{1}{D \Phi}}{1-C +\frac{C}{D \Phi}}=\frac{1}{C \Phi}.
\end{equation*}
Finally we prove that \eqref{pep4} holds for $n+3$.
Since \eqref{pep4} holds for $n+1$ and \eqref{pep3} holds for $n+2$, from \eqref{pep10} we get $\ell(y_{n+3}(\mu))=D \lambda \Phi^{n+4}$. By \eqref{ren30} this gives
\begin{equation*}\label{pep12}
\lambda \Phi^{n+3} < \ell(y_{n+3}(\mu)) < \lambda \Phi^{n+2}.
\end{equation*}
To see that $p_{n+3}(\mu)=p_1(\mu)$ note that the above inequalities, by Lemma \ref{tnewlambda} and by the definition of $\Upsilon_{n+1}$, we have $\Upsilon_{n+3}(\mu)=\Gamma_{n+3}$. Together $p_{n+1}(\mu)<1/C$ and $p_{n+2}(\mu)>1/C$, by Lemma \ref{propyN} this gives
\begin{equation*}
	p_{n+3}(\mu)=1-\frac{\Gamma_{n+3}}{\Gamma_{n+2}}\frac{p_{n+2}(\mu)}{1-D(1- p_{n+2}(\mu))}=1-\Phi \frac{\frac{1}{C \Phi}}{1 - D(1-\frac{1}{C \Phi})}=1-\frac{1}{D \Phi}.
\end{equation*}
Therefore if $|\mu|> \bar{\mu}$, \eqref{pep3} and \eqref{pep4} holds for all $n\geq0$ and thus $p_{n}(\mu)=p_{n+2}(\mu)$ for all $n\geq 0$.

We now prove that if $-\bar{\mu} < \mu < -\nu$, we have
\begin{equation}\label{pep13}
\Phi < D < \frac{1}{\Phi}.
\end{equation}
Since $\lambda<1$ and $D > 1$ the left inequality follows.
Since $\mu>-\bar{\mu}$ and $D(-\bar{\mu},\nu)=1/\Phi$ we get $D < 1/\Phi$.

We now prove by induction on $n$ that if $-\bar{\mu} < \mu < -\nu$, we have \eqref{pep15} for all $n \geq 1$.
From \eqref{ren66} and $\lambda \Phi=1-k\lambda$ we get $\ell(y_0(\mu))=C \lambda \Phi$. Since $\mu < -\nu$, we have $\Phi < \ell(y_0(\mu)) < + \infty$, hence $\kappa_0=0$ and by Lemma \ref{tnewlambda} we get $\Upsilon_0(\mu)=\Gamma_0$. Thus by Lemma \ref{propyN} we have
\begin{equation*}\label{pep16}
p_0(\mu)=\frac{1}{C \Phi},
\end{equation*}
from which we get $p_0(\mu)>1/C$, hence (\ref{pep7}-\ref{pep8}) hold.
By \eqref{pep6} and \eqref{pep13}, since $\mu>-\bar{\mu}$, we have
$\lambda \Phi^3 < \ell(y_1(\mu)) < \lambda \Phi$. This together with Lemma \ref{tnewlambda} and by the definition of $\Upsilon_{n+1}$ shows that $\Upsilon_1(\mu)=\Gamma_3$, and from \eqref{pep5} and \eqref{pep8} we get $p_1(\mu)=1-\Phi/D$. Hence \eqref{pep15} holds for $n=1$.
We now assume that \eqref{pep15} holds for $n$ and prove that \eqref{pep15} holds for $n+1$.
Since $\Phi < 1$ we have $1-\Phi/D> 1-1/D=1/C$, thus by Lemma \ref{propyN},
\begin{equation*}\label{pep17}
\ell(y_{n+1}(\mu))=(1-D(1-p_1(\mu)))\ell(y_n(\mu)),
\end{equation*}
and as \eqref{pep15} holds for $n$, combining this with \eqref{pep13} we get
\begin{equation*}\label{ren42}
	\lambda \Phi^{2n+1}< \ell(y_n(\mu))<\lambda \Phi^{2n-1} \ \textrm{and} \ \lambda \Phi^{2n+3}< \ell(y_{n+1}(\mu))<\lambda \Phi^{2n+1}.
\end{equation*}
Therefore by Lemma \ref{tnewlambda} and by the definition of $\Upsilon_{n+1}$,  $\Upsilon_n(\mu)=\Gamma_{2n+1}$ and $\Upsilon_{n+1}(\mu)=\Gamma_{2n+3}$.
By Lemma \ref{propyN} we get
\begin{equation*}
	1-p_{n+1}(\mu)=\frac{\lambda \Phi^{2n+1}}{\lambda \Phi^{2n-1}}\frac{1-p_n(\mu)}{1-D(1-p_n(\mu))}=\Phi^2\frac{1-p_n(\mu)}{\Phi^2}=1-p_n(\mu).
\end{equation*}
Therefore if $-\bar{\mu} < \mu < -\nu$, then \eqref{pep15} holds for $n \geq 1$ and thus $p_{n}(\mu)=p_{n+1}(\mu)$ for all $n\geq 1$.

We now prove by induction on $n$ that if $\nu < \mu < \bar{\mu}$, \eqref{pep21} holds for all $n \geq 0$.

We have
\begin{equation*}\label{pep20}
C(\nu,\nu)=1 \ \textrm{and} \ C(\bar{\mu},\nu)=\dfrac{2\bar{\mu}}{\bar{\mu}+\nu}=\dfrac{1}{\Phi},
\end{equation*}
Since $\mu \rightarrow C(\mu,\nu)$ is a continuous map for $\nu < \mu < \bar{\mu}$ this gives
\begin{equation}\label{pep19}
1<C<\frac{1}{\Phi}.
\end{equation}
From \eqref{ren66} and $\lambda \Phi=1-k\lambda$ we get $\ell(y_0(\mu))=C \lambda \Phi$. Since $\mu < \bar{\mu}$, from \eqref{pep19} we get $\lambda \Phi < \ell(y_0(\mu)) < \lambda$.

Hence we also have $\lambda \Phi^2 < \ell(y_0(\mu)) < \lambda$. Therefore by Lemma \ref{tnewlambda} $\kappa_0=0$ and we get $\Upsilon_0(\mu)=\Gamma_2$. Thus by Lemma \ref{propyN} we have
\begin{equation*}\label{ren46}
	p_0=\frac{\Gamma_0}{\ell(y_0(\mu))}=\frac{\lambda \Phi^2}{C\lambda \Phi}= \frac{\Phi}{C},
\end{equation*}
and thus \eqref{pep21} holds for $n=0$.
We now assume that \eqref{pep21} holds for $n$ and prove that \eqref{pep15} holds for $n+1$.
Since $p_n(\mu)=p_0(\mu)< 1/C$ we have,
\begin{equation*}\label{ren47}
\ell(y_{n+1}(\mu))=(1-C p_0(\mu))\ell(y_n(\mu)) = \Phi^2 \ell(y_n(\mu)).
\end{equation*}
Since we assume \eqref{pep21} holds for $n$, then we have
\begin{equation*}\label{pep22}
\ell(y_n(\mu))=C \lambda \Phi^{2n+1}.
\end{equation*}
From these two identities, combined with \eqref{pep19}, we get
\begin{equation*}
	\lambda \Phi^{2n+2} < \ell(y_{n}(\mu)) < \lambda \Phi^{2n} \  \textrm{and} \ \lambda \Phi^{2n+4} < \ell(y_{n+1}(\mu)) < \lambda \Phi^{2n+2}.
\end{equation*}
Therefore by Lemma \ref{tnewlambda} and by the definition of $\Upsilon_{n+1}$ we have $\Upsilon_n(\mu)=\Gamma_{2(n+1)}$ and $\Upsilon_{n+1}(\mu)=\Gamma_{2(n+2)}$.
By Lemma \ref{propyN} we get
\begin{equation*}\label{ren48}
	p_{n+1}(\mu)=\frac{\Gamma_{n+1}}{\Gamma_{n}}\frac{p_n(\mu)}{1-C p_n(\mu)}=\frac{\lambda \Phi^{2n+2}}{\lambda \Phi^{2n}}\frac{\frac{\Phi}{C}}{1-\Phi}=\frac{\Phi}{C}
\end{equation*}
Therefore if $\nu < \mu < \bar{\mu}$ then  \eqref{pep21} holds for $n \geq 1$ and thus if $|\mu|\leq \bar{\mu}$ then $p_{n}=p_{n+1}$ for all $n\geq 1$.
This completes the proof.
\end{proof}

\section{Dynamics of the first return map to $P_c$}\label{Dynamics of the first return map}

In this section we introduce a map, denoted by $\rho$, containing information related to the first return under our   transformation $F$ to the middle cone $P_c$  and we show how it can be computed using tools from sections \ref{Bifurcation sequence} and \ref{Dynamical sequences}. This gives a dynamical meaning to the  sequences introduced in Section \ref{Dynamical sequences}, $\{y_n\}$ is the sequence of imaginary parts of the discontinuities of the map $\rho$, while $\{p_n\}$ is the sequence of ratios of the horizontal jumps produced by discontinuities of $\rho$ relative to the cone width $\ell(y_n)$.

Recall \eqref{beta}. Throughout the rest of the paper we set $\nu = \tan(\beta)$. Note that $\nu$ depends on $|\alpha|$, and when necessary to stress this dependence we write $\nu=\nu(|\alpha|)$. Let $\mu' \in \mathbb{R}$, be such that $|\mu'|>\nu$. With this notation we can write  $$P_c=\{z \in \mathbb{H}:-\nu \Re(z) < \Im(z) \wedge \Im(z) > \nu \Re(z)\}.$$

Note that by \eqref{ren121}, $\ell(y)$ is the length of the line segment
$P_c \cap \{z \in \mathbb{H}:\Im(z)=y\}.$
Denote by $L'_1$ and $L'_d$, respectively, the lines $\overline{P_0}\cap \overline{P_1}$ and $\overline{P_d}\cap \overline{P_{d+1}}$ and consider a line $L'_S$ of slope $\mu'$ lying on $P_c$,

\begin{equation*}\label{ren104}
\begin{array}{l}
L'_1=\{z \in \mathbb{H}:\Im(z)=\nu \Re(z)\},\\
L'_d=\{z \in \mathbb{H}: \Im(z)=-\nu \Re(z)\},\\
L'_S=\{z \in \mathbb{H}: \Im(z)=\mu ' \Re(z)\}.
\end{array}
\end{equation*}

Let $L''_S$ be the image of $L'_S$ by $F$, denote its slope by $\mu$ and consider also the line $L_S$ of slope $\mu$ lying on $P_c$, this is:
\begin{equation*}
	\begin{array}{ll}
		L_S=\{z \in \mathbb{H}: \Im(z)=\mu \Re(z)\}, \\ 
		L''_S=\{z \in \mathbb{H}:\Im(z)=\mu \Re(z)+\big(1+\dfrac{\mu}{\nu}\big)y_0\},
	\end{array}
\end{equation*}
where $y_0$ is as in \eqref{ren66}.

Let $R_{\theta}$ be a rotation by an angle $\theta$ centred at the origin. If $\mu '$ is such that $L'_S=L'_S(\mu ')$ is contained in $P_j$, $j=1,...,d$ then
we have $L_S(\mu)=R_{\theta_j}(L_S'(\mu'))$, with $\theta_j=\theta_j(\alpha,\tau)$.
Thus $\mu$ and $\mu'$  are related by the expression
\begin{equation*}\label{ren23}
\mu ' = \frac{\mu-\tan(\theta_j)}{1+\mu\tan(\theta_j)}.
\end{equation*}
or, equivalently,
\begin{equation*}\label{ren23a}
\mu  = \frac{\mu'+\tan(\theta_j)}{1-\mu'\tan(\theta_j)}.
\end{equation*}
The image by $F$ of a point  $z' \in L'_S$ is  a point $z \in L''_S$ where
\begin{equation*}\label{ren60}
\Im(z)=\gamma(\mu,\mu') \Im(z') , \quad \gamma(\mu,\mu')=\sqrt{ \left(1+\dfrac{1}{\mu '^2} \right) / \left( 1+\dfrac{1}{\mu^2} \right)} \ .
\end{equation*}


\begin{figure}[t]
	\begin{subfigure}{0.49\textwidth}
		\centering
		\includegraphics[width=1\linewidth]{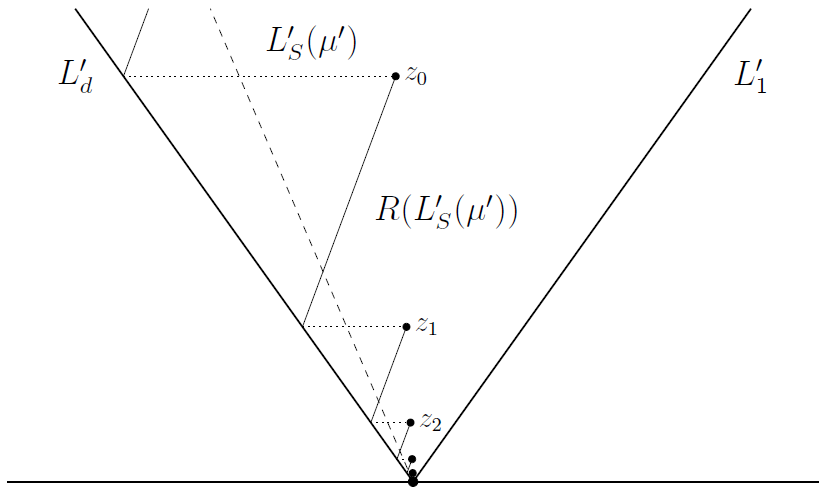}
		\caption{}
		\label{fig:FRLS1}
	\end{subfigure}
	\begin{subfigure}{0.49\textwidth}
		\centering
		\includegraphics[width=1\linewidth]{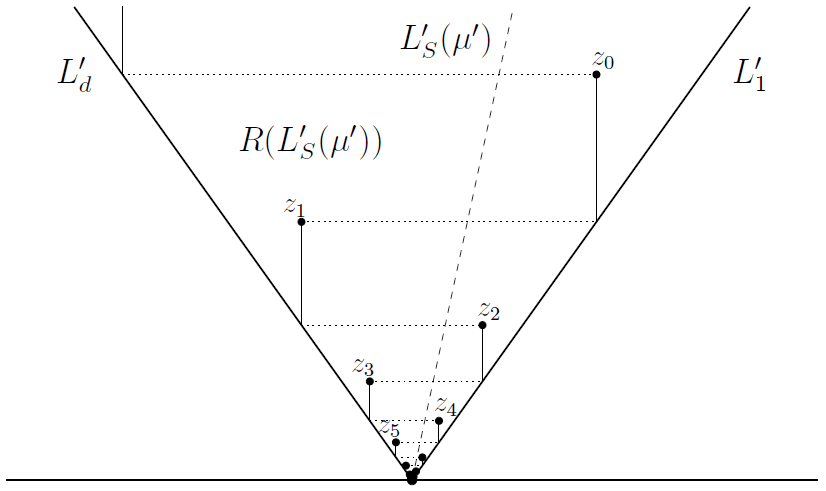}
		\caption{}
		\label{fig:FRLS2}
	\end{subfigure}\\
	\centering
	\captionsetup{width=1\textwidth}
	\caption{An illustration of the action of map $R$. The dashed line is $L'_S(\mu')$ and the union of disjoint line segments is  $R(L'_S(\mu'))$. Also marked are the points $z_n=(p_n-1/2)\ell(y_n)+iy_n$. (A) shows points $z_n$ such that the corresponding sequence $p_n(\mu)$ satisfies $p_{n}(\mu)=p_{n+1}(\mu)$ for all $n\geq 1$. (B) shows points $z_n$ such that the corresponding sequence $p_n(\mu)$ satisfies $p_{n}(\mu)=p_{n+2}(\mu)$ for all $n\geq 0$.}
	\label{fig:FRLS}
\end{figure}

Note that by the definition of $\mu$, since $|\mu'|>\nu$, we also have $|\mu|>\nu$. Now for $y>0$, let $\xi_S(y)$ denote the point $z \in L'_S(\mu')$, such that $\Im(F(z))=y$. This point is unique and is given by $z = (\mu'^{-1}+i)\gamma(\mu,\mu')^{-1}y.$

Define the first return of $\xi_S(y)$ to $P_c$ as the map $\rho:\R^+ \rightarrow P_c$ given by
\begin{equation*}\label{rhoandRa}
\rho(y)=R(\xi_S(y)).
\end{equation*}
By the definition of $R$ we have
\begin{equation*}\label{rhoandR}
R(z)= \rho(\Im(F(z))),
\end{equation*}
for $z \in P_c$. Thus, the study of the map $\rho$ and $R$ are very closely related.

Let
\begin{equation*}\label{fd9}
\mathcal{D}=\{y>0: \rho \ \textrm{is discontinuous at} \ y   \}.
\end{equation*}
Theorem \ref{t5.2} relates the sequence $\{y_n\}_{n \in \mathbf{N}}$ to the set $\mathcal{D}$, and characterizes the map $\rho$. Before stating and proving this theorem we need the following lemma.

\begin{lemma}\label{l1}
	i) Assume there is an $n_1 \in \N$ and constants $\delta\leq \Gamma_{n_1+1}'$, $\ell$ and $\ell'$ such that $\Gamma_{n_1+1}'\leq \ell < \ell' < \ell+ \delta \leq \Gamma_{n_1}'$, then $n_{\ell}(1+\ell')=k_{n_1+1}'$ and
	\begin{equation}\label{(ii)}
	r_{\ell}'(1+\ell')=1+\ell'-\Gamma_{n_1+1}'.
	\end{equation}

	ii) Assume there is an $n_d \in \N$ and constants $\ell$ and $\ell'$ such that $0< \ell'\leq \Gamma_{n_d+1}''\leq \ell < \Gamma_{n_d}''$, then $n_{\ell}(1-\ell')=k_{n_d+1}''$ and
	\begin{equation}\label{(iii)}
	r_{\ell}'(1-\ell')=1-\ell'+\Gamma_{n_d+1}''.
	\end{equation}
\end{lemma}

\begin{proof}
				We begin by proving i). First note that as $1+\ell' \notin (1,1+\ell)$ we have
				\begin{equation}\label{(IV)}
				r_{\ell}'(1+\ell')=r_{\ell}(1+\ell').
				\end{equation}
				Also it is clear that for $1 \leq n < n_{\ell'}(1+\ell')$ we have $g^n(1+\ell') \notin [1,1+\ell']$, and since $\ell\leq \ell'$ this shows that $g^n(1+\ell') \notin [1,1+\ell]$ as well. Thus $n_{\ell}(1+\ell')\geq n_{\ell'}(1+\ell')$.
				
				Since $\Gamma_{n_1+1}'\leq \ell'  < \Gamma_{n_1}'$, by Theorem \ref{n9} we have
				\begin{equation*}\label{(V)}
				g^{n_{\ell'}(1+\ell')}(1+\ell')=1+\ell'-\Gamma_{n_1+1}',
				\end{equation*}
				and as $\ell \leq \ell' \leq \ell+ \delta$ this implies that $g^{n_{\ell'}(1+\ell')}(1+\ell') \in [1,1+\ell]$, thus $n_{\ell}(1+\ell')= n_{\ell'}(1+\ell')$ and from \eqref{(IV)} we get \eqref{(ii)}. Since by Theorem \ref{n9} we have $n_{\ell'}(1+\ell')=k_{n_1+1}'$ this shows that $n_{\ell}(1+\ell')=k_{n_1+1}'$ as well.
				
				We now prove ii). Note that as $1-\ell'<1$ we have
				\begin{equation}\label{(VI)}
				r_{\ell}'(1-\ell')=r_{\ell}(1-\ell').
				\end{equation}
				By the definition of $d^+$, since $\ell'<\ell$ we have $\ell'< d^+(n_\ell(1-\ell'))$, hence, by Lemma \ref{dlemma} we get
				\begin{equation*}\label{(VII)}
				g^{n_\ell(1-\ell')}(1-\ell')=r_{\ell}(1)-\ell'.
				\end{equation*}
				As $\Gamma_{n_d+1}''\leq \ell < \Gamma_{n_d}''$ we can apply Theorem \ref{n9} from whence we obtain
				\begin{equation}\label{(VIII)}
				g^{n_\ell(1-\ell')}(1-\ell')=1-\ell'+\Gamma_{n_d+1}''.
				\end{equation}
				As $0 < \ell \leq \Gamma_{n_d+1}''$ we get $g^{n_\ell(1-\ell')}(1-\ell')\in [1,1+\ell']$, hence $n_{\ell'}(1-\ell')=n_\ell(1-\ell')$ which implies that $r_{\ell}(1-\ell')=g^{n_\ell(1-\ell')}(1-\ell')$. Thus, combining \eqref{(VI)} and \eqref{(VIII)} we get \eqref{(iii)}. Since by Theorem \ref{n9} we have $n_{\ell}(1)=k_{n_d+1}''$ and combined with \eqref{(VIII)} this proves  that $n_{\ell}(1-\ell')=k_{n_d+1}''$ as well.
\end{proof}

\begin{theorem}\label{t5.2}
	Assume $\lambda \in (0,1) \backslash \mathbb{Q}$ and $|\mu'|>\nu>0$. Then $\rho$ is a piecewise affine map of slope $\mu^{-1}$. The set $\mathcal{D}$ is equal to the union of all points in the sequence $\{y_n\}_{n \in \mathbf{N}}$. Furthermore, for all $n \in \mathbf{N}$;
	If $\rho(y_n) \in L'_1$, for $y_{n+1}\leq y < y_n$ we have
	\begin{equation}\label{a4}
	\rho(y)=F^{k(\xi_S(y_n))}(\xi_S(y))-\Upsilon_n.
	\end{equation}
	If $\rho(y_n) \in L'_d$, for $y_{n+1}\leq y < y_n$ we have
	\begin{equation}\label{a5}
	\rho(y)=F^{k(\xi_S(y_n))}(\xi_S(y))+\Upsilon_n.
	\end{equation}
	Also $\rho(y_n) \in L'_1$ (resp. $\rho(y_n) \in L'_d$) if and only if $p_{n-1}>1/C$ (resp. $p_{n-1}<1/C$). 
\end{theorem}

\begin{proof}
	We begin by proving, by induction on $n$, that for all $n \in \mathbf{N}$
	\begin{equation}\label{fd12}
	\textrm{card}\left\{\mathcal{D} \cap \{y \in \R^+:y>y_n\}  \right\}=n,
	\end{equation}
	$\rho(y_n) \in L_1'\cup L_d'$ and that for all $y < y_n$, we have
	\begin{equation}\label{b1}
	k(\xi_S(y))>k(\xi_S(y_n)).
	\end{equation} 
	For all $n \in \mathbf{N}$, we prove that the map $\rho_n: [0,y_n)\rightarrow \mathbb{H}$ such that 
	\begin{equation}\label{a1}
	\rho_n(y)=F^{k(\xi_S(y_n))}(\xi_S(y)),
	\end{equation} 
	is an affine map of slope $\mu^{-1}$. Furthermore, if $\rho(y_n) \in L_1'$ (resp. $\rho(y_n) \in L_d'$) then for all $y<y_n$ we have
	\begin{equation}\label{a2}
	F^{k(\xi_S(y'))}(\xi_S(y))=\rho_n(y)-\Upsilon_{n}
	\end{equation}
	\begin{equation}\label{a3}
	\left( \textrm{resp.} \ F^{k(\xi_S(y'))}(\xi_S(y))=\rho_n(y)+\Upsilon_{n}\right),
	\end{equation}
	where $y'=y_{n+1}$ if $n+1 \in \mathbf{N}$ and $y'=y_{n}/2$ otherwise. For $y_{n+1}\leq y < y_n$ we have \eqref{a4} (resp. \eqref{a5}).\\
	
	We first show that for $n=0$ we have \eqref{fd12}, \eqref{b1}, $\rho(y_0) \in L'_d$ and that $\rho_0$ is an affine map of slope $\mu^{-1}$.

	Note that for all $y \geq 0$ we have 
	\begin{equation}\label{b2}
	F(\xi_S(y))=(\mu^{-1}+i)y-\left(\frac{1}{\mu}+\frac{1}{\nu}\right)y_0,
	\end{equation}
	which is an affine map of slope $\mu^{-1}$. By \eqref{ren66} we have that $\rho(y_0) \in L'_d$ and thus $y_0 \in \mathcal{D}$.

	As $L_S''\cap \{z \in \mathbb{H}: \Im(z)>y_0\} \subseteq P_c$,  for $y>y_0$ we have \eqref{fd12}  and $$\rho(y)=F(\xi_S(y)).$$
	
	Thus by \eqref{b2} we have that $\rho_0$ is an affine map of slope $\mu^{-1}$.  Note that for $y<y_0$ we have $F(\xi_S(y))\in P_{d+1}$ and thus we have \eqref{b1} as well.\\

	It is clear that if $p_0>1/C$ (resp. $p_0<1/C$) then $\rho(y_1) \in L_1'$ (resp. $\rho(y_1) \in L_d'$).
	Now assume, for $n \in \mathbf{N}$ that $\rho(y_n) \in L_1'$ (resp. $\rho(y_n) \in L_d'$),  $p_{n-1}>1/C$ (resp. $p_{n-1}<1/C$), that \eqref{fd12} and \eqref{b1} are true and $\rho_n$ is an affine map of slope $\mu^{-1}$. We show  that \eqref{fd12} and \eqref{b1} hold for $n+1$.  If $\rho(y_n) \in L_1'$ (resp. $\rho(y_n) \in L_d'$) then for all $y<y_n$ we have \eqref{a2} (resp. \eqref{a3}) and for $y_{n+1}\leq y < y_n$ we have \eqref{a4} (resp. \eqref{a5}).
	In particular if $y_{n+1}>0$ then $\rho_{n+1}$ is an affine map of slope $\mu^{-1}$ and $\rho(y_{n+1}) \in  L'_1 \cup L_d'$.

	Assume that $\rho(y_n) \in L_1'$. We begin by proving that there is  $\tilde{y}<y_{n}$ such that for $\tilde{y}\leq y<y_{n}$ we have $\rho(y)= F^{k(\xi_S(\tilde{y}))}(\xi_S(y))$ and \eqref{a4}.
	
	Since $\rho_n$ is an affine map of slope $\mu^{-1}$ and $\rho(y_{n})\in L_1'$, for $y<y_{n}$ we have 
	\begin{equation}\label{a7}
	\rho_n(y)=y_{n}\left(\frac{1}{\nu}-\frac{1}{\mu}\right)+\frac{1}{\mu}y+iy.
	\end{equation}

	We now consider that $\ell(y_n)\leq \Gamma'_0$. As in the other case the proof is similar we will omit it for brevity. By the definition of $\kappa_n$ we have
	\begin{equation}\label{a6}
	\Gamma'_{\kappa_n}<\ell(y_{n})\leq \Gamma'_{\kappa_n-1}.
	\end{equation}
	
	As $\Gamma'_0\leq \lambda$ we get $\ell(y_n)\leq \lambda$, hence by the definition of $\ell$ we get \eqref{eqr2} for $z=\xi_S(y)$ and combining Lemma \ref{prop2.4} with $\rho(y)=R(\xi_S(y))$, by \eqref{b1} and \eqref{a1} we get
	\begin{equation}\label{a11}
	\Re(\rho(y))=s^{-1}\circ r_{\ell(y)}' \left( 1+ \frac{\ell(y)}{2} +  \Re(\rho_n(y))  \right).
	\end{equation}

	Recall the sequence $\{\Upsilon_n\}_{n \in \mathbf{N}}$ as in \eqref{ren1}. Take $0<\delta'<\Upsilon_n$ and
	\begin{equation*}\label{a8}
	\tilde{y}=\max\left(y_{n}-\left(\frac{1}{\nu}-\frac{1}{\mu}\right)^{-1}\delta',\frac{\nu \Upsilon_n}{2}\right).
	\end{equation*}
	Note that we have $\tilde{y}<y_{n}$, since by \eqref{ren121} and \eqref{a6}, we have $\frac{\nu \Gamma_{\kappa_n}'}{2}<y_n$ and as $|\mu|>\nu$ we also have $\left(1/\nu-1/\mu\right)^{-1}>0$.

	We now show that for $\tilde{y}\leq y<y_{n}$ we have
	\begin{equation}\label{a12}
	\Gamma_{\kappa_n}'\leq \ell(y)<\frac{\ell(y)}{2}+ \Re(\rho_n(y))< \ell(y)+\delta \leq \Gamma_{\kappa_n-1}',
	\end{equation}
	with
	\begin{equation}\label{a13}
	\delta=\max(\Gamma_{\kappa_n-1}'-\ell(y),\Gamma_{\kappa_n}').
	\end{equation}
	
	First note that as $y \geq \tilde{y}\geq \nu \Gamma_{\kappa_n}'$ we have $\Gamma_{\kappa_n}'\leq \ell(y)$.
	As $\rho_n(y) \in P_0$ we have $\Re(\rho_n(y))>\ell(y)/2$ and thus $\ell(y)<\ell(y)/2+ \Re(\rho_n(y))$.
	
	By \eqref{a7} and the definition of $\ell$ we have
	\begin{equation}\label{a16}
	\frac{\ell(y)}{2}+\Re(\rho_n(y))= \ell(y)+ \left(\frac{1}{\nu}-\frac{1}{\mu}\right)(y_{n}-y).
	\end{equation}	
	As $|\mu|>\nu$ we have $(1/\nu+1/\mu)>0$, thus, as $y<y_n$ we get that
	$\ell(y)/2+\Re(\rho_n(y))<2y_{n}/\nu$,	which combined with \eqref{a6} and \eqref{ren121} shows that
	\begin{equation*}\label{a15}
	\frac{\ell(y)}{2}+\Re(\rho_n(y))<\ell(y)+(\Gamma_{\kappa_n}'-\ell(y)).
	\end{equation*}
	Since $\delta'<\Gamma_{\kappa_n}'$ we have $y \geq \tilde{y}>y_{n}-(1/\nu-1/\mu)^{-1}\Gamma_{\kappa_n+1}'$ and from \eqref{a13} and \eqref{a16} we get
	$$\frac{\ell(y)}{2}+ \Re(\rho_n(y))< \ell(y)+\delta.$$
	Finally note that if $\ell(y)>\Gamma_{\kappa_n-1}'-\Gamma_{\kappa_n}'$ then $\ell(y)+\delta=\Gamma_{\kappa_n-1}'$ and if $\ell(y)\leq\Gamma_{\kappa_n-1}'-\Gamma_{\kappa_n}'$ then
	$$ \ell(y)+\delta=\ell(y)+\Gamma_{\kappa_n-1}'\leq  \Gamma_{\kappa_n-1}'.$$
	This shows that \eqref{a12} holds true.
	
	Therefore the conditions for applying Lemma \ref{l1} i) are satisfied. With $\ell=\ell(y)$ and $\ell'=\ell(y)/2 +  \Re(\rho_n(y))$  we get
	\begin{equation*}\label{a18}
	r_{\ell(y)}' \left( 1+ \frac{\ell(y)}{2} +  \Re(\rho_n(y))  \right)=1+ \frac{\ell(y)}{2}+\Re(\rho_n(y)) -\Gamma_{\kappa_n}',
	\end{equation*}
	and $n_{\ell(y)}(1+ \ell(y)/2 +  \Re(\rho_n(y)))=k_{\kappa_n}'$.
	
	Combining this with \eqref{a11} and noting that $\Im(\rho(y))=\Im(\rho_n(y))=y$ we get \eqref{a4} for $y \in [\tilde{y},y_{n})$.
	Since $k(\xi_S(y))=n_{\ell(y)}(1+ \ell(y)/2 +  \Re(\rho_n(y)))+1$ we get that for $y \in [\tilde{y},y_{n})$, $k(\xi_S(y))=k_{\kappa_n}'+1$, and thus $k(\xi_S(\tilde{y}))=k(\xi_S(y))$ and $\rho(y)= F^{k(\xi_S(\tilde{y}))}(\xi_S(y))$.\\

	Denote 
	$$d^{-}=d^-\left(1+\frac{\ell(\tilde{y})}{2}+\Re(F(\xi_S(\tilde{y}))), n_{\ell(\tilde{y})}\left(1+\frac{\ell(\tilde{y})}{2}+\Re(F(\xi_S(\tilde{y})))\right)\right),$$
	and let 
	\begin{equation*}\label{b3}
	\Delta(y,\tilde{y})=\frac{\ell(y)}{2}+\Re(F(\xi_S(y)))-\frac{\ell(\tilde{y})}{2}-\Re(F(\xi_S(\tilde{y}))).
	\end{equation*}
	we will show that
	\begin{equation}\label{a23}
	F^{k(\xi_S(\tilde{y}))}(\xi_S(y))=\rho_n(y)-\Upsilon_n,
	\end{equation}
	for all $y < y_n$.

	Let us first prove \eqref{a23} for all $y<y_{n}$. Since it holds for $y \in [\tilde{y},y_{n})$, we are left to prove it for $y<\tilde{y}$.
	
	Note first that by \eqref{b2}, we have
	\begin{equation*}\label{a20}
	\Delta(y,\tilde{y})=\left(\frac{1}{\nu}+\frac{1}{\mu}\right)^{-1}(y-\tilde{y})<0,
	\end{equation*}
	and since $d^-\geq 0$ we have $\Delta(y,\tilde{y})<d^-.$
	Combining this with \eqref{ren121}, we get for  $y<\tilde{y}$, 
	\begin{equation*}\label{a19}
	-(\ell(\tilde{y})-\ell(y))<\Delta(y,\tilde{y})<d^-.
	\end{equation*}	
	From these inequalities and Lemma \ref{l2} we get that for $n\leq n_{\ell(\tilde{y})}(1+ \ell(\tilde{y})/2 +  \Re(F(\xi_S(\tilde{y}))))$
	\begin{equation}\label{a24}
	g_{\ell(y)}^{n}\left(1+ \frac{\ell(y)}{2}+\Re(F(\xi_S(y)))    \right)=g_{\ell(\tilde{y})}^{n}\left(1+ \frac{\ell(\tilde{y})}{2}+\Re(F(\xi_S(\tilde{y})))    \right)+ \Delta(y,\tilde{y}).
	\end{equation}
	
	Recalling that $n_{\ell(\tilde{y})}(1+ \ell(\tilde{y})/2 +  \Re(F(\xi_S(\tilde{y}))))=k(\xi_S(\tilde{y}))-1$ by Lemma \ref{prop2.4}, $F(\xi_S(\tilde{y}))) \in \mathcal{R}_{\lambda,\beta}$ and we have
	\begin{equation*}\label{a25}
	s^{-1}\circ g_{\ell(\tilde{y})}^{k(\xi_S(\tilde{y}))-1}\left(1+ \frac{\ell(\tilde{y})}{2}+\Re(F(\xi_S(\tilde{y})))    \right)=\Re(\rho(\tilde{y})).
	\end{equation*}
	
	By Lemma \ref{cor2.3}, combining  the previous identity with \eqref{a24} gives
	\begin{equation*}\label{a26}
	F^{k(\xi_S(\tilde{y}))}(\xi_S(y)) = \Re(\rho(\tilde{y})) - \frac{1}{\mu}(\tilde{y}-y)+iy,  
	\end{equation*}
	and since \eqref{a4} holds true for $y=\tilde{y}$, by \eqref{a7} we also have
	\begin{equation*}\label{a27}
	\Re(\rho(\tilde{y}))=y_{n}\left( \frac{1}{\nu} - \frac{1}{\mu}   \right) + \frac{1}{\mu}\tilde{y}- \Gamma_{\kappa_n}'.
	\end{equation*}
	Combining the two expressions above and  \eqref{a7} we get
	$
		F^{k(\xi_S(\tilde{y}))}(\xi_S(y))=\rho_n(y)-\Gamma_{\kappa_n}',
	$
	which together with \eqref{ren1} gives \eqref{a23} as intended.

	We now prove that for all $y_{n+1} \leq y < y_{n}$,
	\begin{equation}\label{a29}
	k(\xi_S(y))= k(\xi_S(\tilde{y})).
	\end{equation}
	
	By Lemma \ref{l2}, 
	$  n_{\ell(\tilde{y})}(1+ \ell(\tilde{y})/2 +  \Re(F(\xi_S(\tilde{y})))) \leq n_{\ell(y)}(1+ \ell(y)/2 +  \Re(F(\xi_S(y)))),  $
	for $y \leq \tilde{y}$,  thus $k(\xi_S(y))\geq k(\xi_S(\tilde{y})).$ 
	
	For all $y \in [\tilde{y},y_{n})$, since  $k(\xi_S(y))= k(\xi_S(\tilde{y}))$, to prove  \eqref{a29} for  $y_{n+1} \leq y < y_{n}$ it is enough instead to show that  
	\begin{equation}\label{aextra}
	F^{k(\xi_S(\tilde{y}))}(\xi_S(y))\in P_c.
	\end{equation}

	Begin by noting that by \eqref{a23} we have
	\begin{equation}\label{a30}
	F^{k(\xi_S(\tilde{y}))}(\xi_S(y)) = y_{n} \left( \frac{1}{\nu} -\frac{1}{\mu}   \right) + \frac{1}{\mu} y -\Gamma_{\kappa_n}' +i y.
	\end{equation}

	Combining \eqref{CandD} with the definitions of $y_{n+1}$, $\kappa_{n+1}$, $\Upsilon_{n+1}$ and $p_{n+1}$, we get 
	\begin{equation}\label{a28}
	y_{n+1}=\left\{\begin{array}{ll}\vspace{0.2cm}
	0 \ , & y_{n}=\left( \frac{1}{\nu} -\frac{1}{\mu}   \right)^{-1}\Gamma_{\kappa_n}',\\ \vspace{0.2cm}
	y_{n}-\left( \frac{1}{\nu} -\frac{1}{\mu}   \right)^{-1}\Gamma_{\kappa_n}' \ , & y_{n}>\left( \frac{1}{\nu} -\frac{1}{\mu}   \right)^{-1}\Gamma_{\kappa_n}',\\ 
	\left( \frac{1}{\nu} +\frac{1}{\mu}   \right)^{-1}\Gamma_{\kappa_n}' -  \dfrac{\mu-\nu}{\mu+\nu} y_{n}\ , & y_{n}<\left( \frac{1}{\nu} -\frac{1}{\mu}   \right)^{-1}\Gamma_{\kappa_n}'.\\
	\end{array}\right.
	\end{equation}

	It is clear from \eqref{a28} and using $|\mu|>\nu$ that $y_{n+1}>0$ if $y_{n}\neq ( 1/\nu - 1/\mu   )^{-1}\Gamma_{\kappa_n}'$ and $y_{n+1}=0$ otherwise. 
	
	We consider the three separate cases in \eqref{a28}.
	
	If $y_{n} = ( 1/\nu - 1/\mu   )^{-1}\Gamma_{\kappa_n}'$, by \eqref{a30} we have
	$ F^{k(\xi_S(\tilde{y}))}(\xi_S(y)) = y/ \mu + i y ,$
	which, since $|\mu|>\nu$ proves \eqref{aextra}.

	If $y_{n}> ( 1/\nu - 1/\mu   )^{-1}\Gamma_{\kappa_n}'$ it follows from \eqref{a30} and  $|\mu|>\nu$ that $  -y / \nu< \Re(  F^{k(\xi_S(\tilde{y}))}(\xi_S(y))  ), $
	also it follows from \eqref{a30} that
	\begin{equation*}\label{a31}
	\Re(  F^{k(\xi_S(\tilde{y}))}(\xi_S(y))  ) = (y_{n}-y) \left( \frac{1}{\nu} -\frac{1}{\mu}   \right)-\Gamma_{\kappa_n}' + \frac{1}{\nu}y,
	\end{equation*}
	and since $y\geq y_{n+1}$, we get from \eqref{a28} that
	$ \Re(  F^{k(\xi_S(\tilde{y}))}(\xi_S(y))  ) \leq y / \nu,$
	proving \eqref{aextra} in this case.
	
	Finally, if  $y_{n}< ( 1/\nu - 1/\mu   )^{-1}\Gamma_{\kappa_n}'$, it follows from \eqref{a30} and  $|\mu|>\nu$ that\\
	$  \Re(  F^{k(\xi_S(\tilde{y}))}(\xi_S(y))  )  < y / \nu , $
	and from \eqref{a30} that
	\begin{equation*}\label{a32}
	\Re(  F^{k(\xi_S(\tilde{y}))}(\xi_S(y))  ) = y_{n} \left( \frac{1}{\nu} -\frac{1}{\mu}   \right)+  y \left( \frac{1}{\nu} +\frac{1}{\mu}   \right)  -\Gamma_{\kappa_n}' - \frac{1}{\nu}y.
	\end{equation*}
	Since $y\geq y_{n+1}$, we get from the above expression and \eqref{a28} that
	$$ \Re(  F^{k(\xi_S(\tilde{y}))}(\xi_S(y))  ) \geq -y / \nu,$$
	and thus \eqref{aextra}. 
	
	This shows that for all $y_{n+1} \leq y < y_{n}$ we have \eqref{a29}.\\

	From \eqref{a29} it follows that \eqref{fd12} holds for $n+1$. It also follows that $F^{k(\xi_S(\tilde{y}))}(\xi_S(y))=F^{k(\xi_S(y'))}(\xi_S(y))$, hence by \eqref{a23} we have that \eqref{a2} holds for all $y<y_{n}$. Also from \eqref{a29} it follows that for all  $y_{n+1} \leq y < y_{n} $, $\rho(y)=F^{k(\xi_S(y'))}(\xi_S(y))$ and thus from \eqref{a2} we get \eqref{a4} as well.

	Finally note that if $y_{n+1}>0$, then $y'=y_{n+1}$ and hence by \eqref{a4} $\rho_{n+1}$ is an affine map of slope $\mu^{-1}$. As $\rho(y_n) \in L_1'$ we have $p_{n-1}>1/C$, hence by \eqref{a28} and the definitions of $y_n$ and $p_n$ it is straightforward to check that $\rho(y_{n+1}) \in L_1'$ (resp. $\rho(y_{n+1}) \in L_d'$) if and only if $y_{n}>\left( \nu^{-1} -\mu^{-1}   \right)^{-1}\Gamma_{\kappa_n}'$ (resp. $y_{n}<\left( \nu^{-1} -\mu^{-1}   \right)^{-1}\Gamma_{\kappa_n}'$) if and only if $p_n>1/C$ (resp. $p_n<1/C$).

The proof for the case  $\rho(y_{n}) \in L_d'$ is similar to the previous one and so we omit it.

	By  \eqref{a4}, \eqref{a5} and \eqref{b2} we get that $\rho(y)$ is a an affine map of slope $\mu^{-1}$ for all $y_{n+1}\leq y <y_n$, $n \in \mathbf{N}$, hence by Lemma \ref{lemmayn} it is a picewise affine map in $[0,y_0]$. Also by Lemma \ref{lemmayn} and \eqref{fd12} it follows that the set of discontinuities $\mathcal{D}$ is equal to the union of all $\{y_n\}_{n \in \mathbf{N}}$.
\end{proof}


\section{Proof of Theorems \ref{renormtheorem} and \ref{infembed}}\label{Proofs of Theorems}
In this section we prove our main results, theorems \ref{renormtheorem} and \ref{infembed}.

Set $x_n(\mu)=\Re(\rho(y_n^-(\mu)))$. By Theorem \ref{t5.2} and by the definition of $\Upsilon_{n}$, for all $n \in \mathbf{N}$, we have
\begin{equation*}\label{ren1a}
\rho(y_n^-(\mu))=\left\{\begin{array}{ll}
\dfrac{y_n(\mu)}{\nu}-\Upsilon_{n}(\mu) + i y_n(\mu), & \rho(y_n(\mu)) \in L'_1, \\
\Upsilon_{n}(\mu) - \dfrac{y_n(\mu)}{\nu} + i y_n(\mu), & \rho(y_n(\mu)) \in L'_d,
\end{array}\right.
\end{equation*}
which by the definitions of $\ell$ and $p_n$ gives
\begin{equation}\label{ren120}
p_n(\mu)=\frac{x_n(\mu)}{\ell(y_n(\mu))}+\frac{1}{2}, \ \textrm{for \ all} \ n \in \mathbf{N}.
\end{equation}

\subsection{Proof of Theorem \ref{renormtheorem}}

 Let $\{y_n(\mu)\}$ be the sequence associated to $L''_S(\mu)$. Recall that by \eqref{beta} we have $\beta=(\pi-|\alpha|)/2$. 
 
 We begin by proving that there is a positive real number $\bar{y_1}$ such that, for all $\mu$ satisfying $|\mu|>\tan(\beta)=\nu$, we have $y_1(\mu)\geq \bar{y_1}$. 
 Let $\varphi$, $\varphi' \in [\beta,\pi-\beta]$ be such that
 \begin{equation}\label{ren85}
 \mu=\tan(\varphi) \quad \textrm{and} \quad \mu'=\tan(\varphi').
 \end{equation}
 Let $L'(\mu')\subseteq P_j$, we define
 \begin{equation}\label{ren82}
 \gamma_j(\varphi)=
 \left|\cos(\theta_j)-\sin(\theta_j)\cot(\varphi)\right|^{-1},
 \end{equation}
 and
 \begin{equation*}\label{ren83}
 \gamma_j'(\varphi')=
 \left|\cos(\theta_j)-\sin(\theta_j)\cot(\varphi')\right|,
 \end{equation*}
 where $\theta_j = \theta_j(\alpha,\tau)$. By the definition of $\mu'$ we can see that
 \begin{equation}\label{ren118}
 \gamma(\mu,\mu')=\gamma_j(\varphi)=\gamma_j'(\varphi').
 \end{equation}
 Recall from \eqref{ren66} that
 \begin{equation*}
 y_0(\mu)=\eta\frac{\mu \nu}{\mu+\nu}.
 \end{equation*}
 Hence using \eqref{ren82}, we have
 \begin{equation}\label{ren100}
 y_0(\tan(\varphi))\gamma_j(\varphi)^{-1} = \eta \nu\left|\cos(\theta_j)\right| \left| \frac{1-\tan(\theta_j)\cot(\varphi)}{1+\nu\cot(\varphi)}\right|.
 \end{equation}
 
 Let
 \begin{equation}\label{ren84}
 \bar{y_0}=\min_{j \in \{1,...,d\}}\left\{\inf_{\varphi \in W_j}\left\{y_0(\tan(\varphi))\gamma_j(\varphi)\right\}\right\}.
 \end{equation}
 Fix $j \in J=\{1\leq j \leq d: \theta_{j}=\pi/2    \}$. By \eqref{ren100}, if $\varphi \neq \pi/2$, we have
 \begin{equation*}\label{ren103}
 y_0(\tan(\varphi))\gamma_j(\varphi) =\eta \nu\left| \frac{\cot(\varphi)}{1+\nu\cot(\varphi)}\right|>0.
 \end{equation*}
 
 We now show that $\pi/2 \notin W_j$. Assume that $\varphi=\pi/2 \in W_j$.
 Note that from the definition of $L_S'$ and \eqref{ren85} we get
 $
 \varphi'=\varphi-\theta_j.
 $
 Therefore, since $\theta_j=\pi/2$, we have $\varphi'=0$, which is impossible since $\nu=\tan(\beta)>0$ and $\varphi'\in [\beta,\pi-\beta]$. 
 Thus, we get
  \begin{equation*}\label{ren101}
 \bar{y_0}=\min_{j \notin J}\left\{\inf_{\varphi \in W_j}\left\{y_0(\tan(\varphi))\gamma_j(\varphi)\right\}\right\}.
 \end{equation*}

 Now fix $j \in \{1,...,d\}$. Since $\varphi' \in [\beta, \pi-\beta]$ we have $\varphi'>\arctan(\nu)$, and thus, since $\varphi'=\varphi-\theta_j$, we have $\varphi-\theta_j>\arctan(\nu).$ Thus, $\varphi$ is bounded away from $\theta_j$ and this bound depends only on $\nu$. Therefore $\tan(\theta_j)\neq \tan(\varphi)$ and thus there is  $\tilde{c}(\nu,j)>0$ such that
 \begin{equation*}\label{ren106}
 |1-\tan(\theta_j)\cot(\varphi)|>\tilde{c}(\nu,j).
 \end{equation*}
 Since $\varphi \in [\arctan(\nu), \pi-\arctan(\nu)]$ we have $|\nu \cot(\varphi)\leq 1|$, thus we also have $|1+\nu\cot(\varphi)|\leq 2$. From this and the above inequality we get
\begin{equation*}\label{ren107}
\eta \nu \left|\cos(\theta_j)\right| \left| \frac{1-\tan(\theta_j)\cot(\varphi)}{1+\nu\cot(\varphi)}\right|\geq \frac{\eta \nu}{2}\tilde{c}(\nu,j)\left|\cos(\theta_j)\right|>0.
\end{equation*}
Combining this with \eqref{ren100} and \eqref{ren84} we get
\begin{equation*}
\bar{y_0}\geq \min_{j \in \{1,...,n\}}\left\{\frac{1}{2}\tilde{c}(\nu,j)\left|\cos(\theta_j)\right|\right\}>0.
\end{equation*}
Thus, for all $\nu>0$, we have $\bar{y_0}>0$.

Note that from \eqref{ren85} and the definitions of $C$ and $D$, we can write $\mathfrak{D}(\varphi)=D/C$ as a function of $\varphi$ as
\begin{equation*}\label{ren111}
\mathfrak{D}(\varphi)=\frac{1+\nu \cot(\varphi)}{1-\nu \cot(\varphi)}.
\end{equation*}
Define the interval $W^{\varphi}=[\arctan(\nu),\pi-\arctan(\bar{\mu})]$. 
Note that $\mathfrak{D}(\varphi)$ is  a positive, continuous and decreasing function of $\varphi \in W^{\varphi}$. Since $\varphi \leq \pi - \arctan(\bar{\mu})$, we have
\begin{equation*}\label{ren113}
\mathfrak{D}(\varphi)\geq \frac{1+\nu (-\Phi^3/\nu)}{1-\nu (-\Phi^3/\nu)} =\frac{1-\Phi^3}{1+\Phi^3} = \Phi,
\end{equation*}
since $\Phi^2=1-\Phi$. Thus we obtain
\begin{equation}\label{ren112}
\inf_{\varphi \in W^{\varphi}}\mathfrak{D}(\varphi)\geq \Phi.
\end{equation}

It follows from Theorem \ref{rencor} that  $y_1=y_0 D \Phi /C$ if $\mu \geq -\bar{\mu}$ and $y_1=y_0 \Phi^2$ if $\mu < -\bar{\mu}$. This implies that for all $\varphi \in W^{\varphi}$, we have that
\begin{equation*}\label{ren114}
y_1(\tan(\varphi))=
\left\{\begin{array}{ll}
\Phi^2 y_0(\tan(\varphi)), & \mu < \bar{\mu},\\
\mathfrak{D}(\varphi)y_0(\tan(\varphi)) & \mu \geq \bar{\mu}.
\end{array}
\right.
\end{equation*}
By \eqref{ren84} this gives
\begin{equation*}\label{ren116}
\min_{j \in \{1,...,d\}}\left\{\inf_{\varphi \in W_j}\left\{y_1(\tan(\varphi))\gamma_j(\varphi)\right\}\right\}\geq  \min\left(\Phi^2,\inf_{\varphi \in W^{\varphi}}\mathfrak{D}(\varphi)\right) \bar{y_0}  
\end{equation*}

Define $\bar{y_1}=\Phi^2 \bar{y_0}$. Note that since $\bar{y_0}>0$, we have $\bar{y_1}>0$ as well. 
From the above inequality and \eqref{ren112} we get 
\begin{equation}\label{ren117}
y_1(\mu)\geq \min_{j \in \{1,...,d\}}\left\{\inf_{\varphi \in W_j}\left\{y_1(\tan(\varphi))\gamma_j(\varphi)\right\}\right\}\geq \bar{y_1}.
\end{equation}


Define $U=\{z\in P_c: \Im(z)<\bar{y_1}  \}$. We now prove \eqref{ren50}  for $z \in U$.
Let $\mu'$ be such that $z \in L'_S(\mu')$, then $\Phi^2 z \in L'_S(\mu ')$, hence by the definition of $\gamma(\mu,\mu')$ and as $R(z)= \rho(\Im(F(z)))$ we have 
\begin{equation}\label{ren51}
 \frac{1}{\Phi^2}R(\Phi^2 z)=\frac{1}{\Phi^2} \rho(\gamma(\mu,\mu') y \Phi^2),
 \end{equation}
Set $y' = \gamma(\mu,\mu') y$. 
From \eqref{ren118} and \eqref{ren117} we have
\begin{equation}\label{ren110}
y_1(\mu)=\gamma(\mu,\mu')\gamma_j(\mu)^{-1}y_1(\mu)\geq \gamma(\mu,\mu')\bar{y_1},
\end{equation}
for $j$ such that $(x,y)\in P_j$.
Since $\Im(\rho(y'))=y'$, by \eqref{ren51} and \eqref{ren110}, to prove \eqref{ren50} it is enough to prove that
\begin{equation}\label{ren108}
\Re(\rho(y' \Phi^2))=\Phi^2 \Re(\rho(y')),
\end{equation}
for $y' < y_1(\mu)$.
We prove \eqref{ren108} for  $y' < y_1(\mu)$. Recall that  $y_{1}=y_1(\mu)$. By \eqref{ren110}, there must be an $n\geq 1$, such that
\begin{equation}\label{ren122}
y_{n+1}(\mu)\leq y' < y_{n}(\mu).
\end{equation}
 Recall from Theorem \ref{t5.2} that $\rho(y')$ is a piecewise affine map of constant slope $\mu^{-1}$ and it is continuous if $y'$ satisfies \eqref{ren122}. From this we have 
  \begin{equation*}\label{ren123}
  \rho(y')=\rho(y_{n+1})-\frac{y_{n+1}-y'}{\mu},
  \end{equation*}
  and combining this with \eqref{ren123} and by the definition of $\ell$, we have
   \begin{equation}\label{ren124}
  \Re(\rho(y'))=(2p_n(\mu)-1)\frac{y_{n+1}(\mu)}{\nu}-\frac{y_{n+1}(\mu)-y'}{\mu}.
  \end{equation}
 Now multiplying \eqref{ren122} by $\Phi^2$ we get
 \begin{equation*}\label{ren125}
 y_{n+1}(\mu)\Phi^2 \leq y' \Phi^2 < y_n(\mu) \Phi^2,
 \end{equation*}
 thus by Theorem \ref{rencor} we have
 \begin{equation*}\label{ren53}
 \left\{\begin{array}{ll}
 y_{n+2}(\mu) \leq y' \Phi^2 < y_{n+1}(\mu) & , \ \textrm{if} \ |\mu| < \bar{\mu} \\
 y_{n+3}(\mu) \leq y' \Phi^2 < y_{n+2}(\mu) & , \ \textrm{if} \ |\mu| \geq \bar{\mu}. \\
 \end{array}\right.
 \end{equation*}
 By a similar argument to the used to prove \eqref{ren123}, from the above inequalities we get
 \begin{equation*}\label{ren54}
 \Re(\rho(y' \Phi^2))=\left\{\begin{array}{ll}
 (2p_{n+1}(\mu)-1)\dfrac{y_{n+2}(\mu)}{\nu}-\dfrac{y_{n+2}(\mu)-y'\Phi^2}{\mu} & , \ \textrm{if} \ |\mu| < \bar{\mu} \\
 (2p_{n+2}(\mu)-1)\dfrac{y_{n+3}(\mu)}{\nu}-\dfrac{y_{n+3}(\mu)-y'\Phi^2}{\mu} & , \ \textrm{if} \ |\mu| \geq \bar{\mu},
 \end{array}\right.
 \end{equation*}
applying Theorem \ref{rencor} to this expression gives
 \begin{equation*}\label{ren55}
\Re(\rho(y' \Phi^2)) = (2p_{n}(\mu)-1)\frac{y_{n+1}(\mu)\Phi^2}{\nu}-\frac{y_{n+1}(\mu)\Phi^2-y'\Phi^2}{\mu}.
 \end{equation*}
 Comparing this identity with \eqref{ren124} we get \eqref{ren108}.
This completes our proof.

\hfill\ensuremath{\square}\\

Recall our definition of first return map $R$ of $z \in P_c$ to the middle cone $P_c$.
Before proving  Theorem \ref{infembed} we need the following result showing that in the conditions of Theorem \ref{renormtheorem}, $R$ is a PWI with respect to a partition of countably many atoms.
\begin{theorem}\label{RPWI}
 For all $\alpha \in \mathbb{A}$, $\lambda= 1/(k+\Phi)$ and $\eta=1-k\lambda$ with $k \in \N$, $R$ is a piecewise isometry with respect to a partition of countably many atoms.
\end{theorem}

\begin{proof}
We begin by noting that $R$ is a PWI since it is the first return map under $F$ to $P_c$ which is a union of elements of the partition of $F$.
We now prove that the partition of $R$ has countably many atoms. Assume by contradiction that there is  $N \in \N$, a partition $\{Q_j\}_{j\in \{0,...,N-1\}}$ of $P_c$, and $\theta_j(\alpha,\tau)$,  $\lambda_j$ for $j\in \{0,...,N-1\}$ such that
\begin{equation*}\label{conseq1}
R(z) = e^{i \theta_j(\alpha,\tau)} z + \lambda_j , \quad z \in Q_j.
\end{equation*}

By Theorem \ref{renormtheorem} there is an open set $U$ of $P_c$, containing the origin, where $R$ is renormalizable. 
Consider the set $U'=U \backslash \Phi^2 U$ and take $j' \in \{0,...,N-1\}$ such that $U'\cap P_{j'} \neq \emptyset$. Since $\lambda$ and $\eta$ are irrational numbers, we have that
$R(z) = e^{i \theta_{j'}(\alpha,\tau)} z + \lambda_{j'}$ for $z \in U' \cap P_{j'}$,

Define the sequence $\{\tilde{U}_k\}_{k \geq 0}$, where
\begin{equation*}\label{conseq3}
\tilde{U}_0=U'\cap P_{j'} \quad \textrm{and} \quad \tilde{U}_k=\Phi^{2(k-1)}\tilde{U}_0 \backslash \Phi^{2 k}\tilde{U}_0, \ \textrm{for} \ k\geq 1.
\end{equation*}
For every $k \geq 0$  and all $z \in \tilde{U}_k$ we have that  $\Phi^{-2k} z \in \tilde{U}_0$. Since $\tilde{U}_k \subseteq U$, we can renormalize $R$, $k$ times to get
\begin{equation*}\label{conseq4}
R(z)= \Phi^{2k} R(\Phi^{-2k}z)= e^{i\theta_{j'}(\alpha,\tau)}z+\Phi^{2k}\lambda_{j'}.
\end{equation*}

Since $\lambda_{j'} \neq 0$, $\Phi^{2k}\lambda_{j'}$ takes countably many different values, hence for each $k$ there must be a $j_k$ such that for $z \in \tilde{U}_k$ we have $z \in P_{j_k}$ and $j_k\neq j_{k'}$ for $k \neq k'$. But $j_k \in \{0,...,N-1\}$ hence there must exist $k' \neq k''$ such that $j_{k'} = j_{k''}$, which is a contradiction. This finishes our proof.
\end{proof}

\subsection{Proof of Theorem \ref{infembed}}

We begin by proving that $P_c$ can be separated into two connected regions $C_b$ and $C_u$, which are forward invariant for $R$, such that $C_b$ is bounded and $C_u$ is unbounded.

By the proof of Theorem \ref{renormtheorem} there exists a $\overline{y_1}>0$ and an open set
\begin{equation}\label{U}
U=\{ z \in P_c: \Im(z) <  \overline{y_1} \},
\end{equation}
such that we have \eqref{ren50} for all $z \in U$.

Since $\lambda= 1/(k+\Phi)$ and $\eta=1-k\lambda$ with $k \in \N$, by Theorem \ref{RPWI}, $R$ is a PWI with respect to a partition of countably many atoms which we denote $\mathcal{P}_R$.
Furthermore, since $\alpha \in \mathfrak{A}(\lambda,\eta)$, there exist $d'\geq2$, $a \in \R_+^{d'}$, $\pi \in S(d')$ and  a continuous embedding $h$, of $f_{a,\pi}:I\rightarrow I$ into $R:P_c\rightarrow P_c$, such that $h(I)\subset \Phi^2 U$, $h(0)\in L_d'$, $h(|a|)\in L_1'$ and 
\begin{equation*}
\mathcal{B}=\{ P\in \mathcal{P}_R :P \cap h(I) \neq \emptyset   \},
\end{equation*}
is a barrier for $R$. Let
$$\mathfrak{L}_1=\left \{ z \in L_1': \Im(z)\leq \Im(h(|a|))    \right \}, \quad \mathfrak{L}_d=\left \{ z \in L_d': \Im(z)\leq \Im(h(0))    \right \}.$$

Since  $h(|a|)\in L_1'$ and $h(0)\in L_d'$ we have that $h(|a|)\in \mathfrak{L}_1$ and $h(0)\in \mathfrak{L}_d$ respectively. As $h$ is a homeomorphism of $I$, $\mathfrak{L}_1\cap h(I)= h(|a|)$ and $\mathfrak{L}_d\cap h(I)= h(0)$,  we have that $J=\mathfrak{L}_1\cup \mathfrak{L}_d \cup h(I)$ is homeomorphic to a circle, hence by the Jordan curve Theorem $\C \backslash J$ consists of two connected components, a bounded $C_b'$ and an unbounded $C_u'$.\\

Take $C_b=\overline{C_b'} \cap P_c$ and $C_u=C_u' \cap P_c$. We now show that for any $P \in \mathcal{B}$ we have $R(P \cap C_u) \subseteq C_u$ and $R(P \cap C_b) \subseteq C_b$.

Let $P \in \mathcal{B}$. Note that the restriction $R|_{P}$ of $R$ to $P$ is an orientation preserving isometry. Furthermore since $h$ is a continuous embedding it is order preserving, hence $R|_{P\cap h(I)}$ is order preserving as well. Thus it is possible to construct an orientation preserving homeomorphism $\tilde{h}: \C \rightarrow \C$ such that $\tilde{h}|_{P}=R|_{P}$. $\tilde{h}$ must map $C_b$ into $C_b$ and $C_u$ into $C_u$. In particular if $z \in P \cap C_u$ (resp. $z \in P \cap C_b$) then $R(z)= \tilde{h}(z) \in C_u$ (resp. $C_b$).

We now show that $R(C_u)\subseteq C_u$. Note that since $\mathcal{B}$ is a barrier, $P_c \backslash \mathcal{B}$ is the union of two disjoint connected components $B_u$, $B_b$. Since $h(I) \subset \bigcup_{B \in \mathcal{B}}B$, these regions must be contained in $C_u$ or $C_b$. Without loss of generality assume $B_u \subseteq C_u$ and $B_b \subseteq C_b$.

Assume by contradiction that there is a $z \in C_u$ such that $R(z) \notin C_u$. Since for any $P \in \mathcal{B}$ we have $R(P \cap C_u)\subseteq C_u$, we must have $z \in B_u$. Since $\mathcal{B}$ is a barrier we have that $R(z) \notin B_b$, thus we must have $R(z) \in C_b\backslash B_b$. Let $P\subseteq B_u$ be the atom of the partition $\mathcal{P}_R$ such that $z \in P$. Since $R(z) \in C_b \backslash B_b$ we have $R(P)\cap\mathcal{B} \neq \emptyset$ and since $\mathcal{B}$ is a barrier this implies that $R(P)\cap (\overline{\mathcal{B}} \cap \overline{B_u}) \neq \emptyset$.

As $\overline{B_u}\subseteq \overline{C_u}$ we have that either $R(P)\cap h(I)\neq \emptyset$ or $R(P)\cap C_u\neq \emptyset$. In the later case, as $R(z)\in C_b$, $R(P)$ is connected and $C_u$ and $C_b$ are disjoint we have that $R(P) \cap \overline{C_b}\cap \overline{C_u} \neq \emptyset$ and hence $R(P)\cap h(I)\neq \emptyset$ as well. As $h$ is bijective this is only possible if $B \in \mathcal{B}$ which contradicts $P \subseteq B_u$.

Similarly we can see that $R(C_b)\subseteq C_b$. We will omit this part for brevity of the argument.\\

We now construct sets $V_1, V_2,...$, which are forward invariant by $R$. We first define a set $V_1\subseteq U$ and show that $R(V_1)\subseteq V_1$.

Let $h'=\Phi^{-2}h$, we show that $h':I\rightarrow \Phi^{-2}h(I)$ is a continuous embedding of $f_{a,\pi}$ into $R$. Since $h(I)\subseteq \Phi^{2}U$, by Theorem \ref{renormtheorem} we have \eqref{ren50} for all $z \in \Phi^{-2}h(I)$. Hence for all $x \in I$ we have
\begin{equation*}\label{d1}
R \circ h'(x) = \Phi^{-2} R \circ h(x).
\end{equation*} 
Combining this with \eqref{eq0s2}, which holds as $h$ is an embedding, we get
\begin{equation*}\label{d2}
R \circ h'(x) = h' \circ f(x),
\end{equation*}
for all $x \in I$.

As before $h'(I)$ separates $P_c$ into two disjoint connected components, one bounded $C_b''$ and other unbounded $C_u''$. Take $V_1=C_b''\cap C_u$. Since $h'(I)\subset U$ we have $C_b'' \subseteq U$ and thus $V_1 \subseteq U$. To see that $V_1$ is forward invariant by $R$, note that if $z \in C_b''$, then  $\Phi^{2}z \in C_b$ and hence $R(\Phi^{2}z) \in C_b$. Since $C_b'' \subseteq U$, by Theorem \ref{renormtheorem} we have $R(z)\in \Phi^{-2}C_b\subseteq C_b''$.
Thus $R(C_b'')\subseteq C_b''$ and as $R(C_u)\subseteq C_u$ we get that $R(V_1)\subseteq V_1$ as intended.

Take $V_n= \Phi^{2(n-1)}V_1$, for $n \geq 2$. To see that $V_n$ is forward invariant by $R$, take $z \in V_n$, then $\Phi^{2(n-1)}z\in V_1\subseteq U$. Hence by Theorem \ref{renormtheorem} we have $$R(z)=\Phi^{2(n-1)} R(\Phi^{-2(n-1)}z), $$ and thus $R(z)\in V_n$.\\

We now prove that
\begin{equation}\label{d7}
\bigcup_{n=1}^{+\infty}V_n = C_b'' \backslash \{0\}.
\end{equation}
First we show, by induction on $n$, that for all $n \geq 1$ we have
\begin{equation}\label{d3}
V_1 \cup ... \cup V_n = C_b'' \cap \Phi^{2(n-1)}C_u.
\end{equation}
It is simple to see that \eqref{d3} holds for $n=1$. We assume \eqref{d3} holds for $n$ and show it holds for $n+1$. By \eqref{d3} we get
\begin{equation*}\label{d4}
V_1 \cup ... \cup V_{n+1} = (C_b'' \cap \Phi^{2(n-1)}C_u) \cup (\Phi^{2n}C_b'' \cap \Phi^{2n}C_u).
\end{equation*}
As $\Phi^{2n}C_b''= \Phi^{2(n-1)}C_b$ we have that
\begin{equation*}\label{d5}
C_b''=\Phi^{2n}C_b'' \cup (C_b''\cap\Phi^{2(n-1)} C_u ),
\end{equation*}
and as $\Phi^{2(n-1)} C_u \subseteq \Phi^{2n} C_u$ we have
\begin{equation*}\label{d6}
\Phi^{2n} C_u=\Phi^{2n} C_u\cup (C_b''\cap \Phi^{2(n-1)}C_u).
\end{equation*}
Combining the three expressions above we get that \eqref{d3} is  true for $n+1$, as intended.\\

Since $h(I)\subseteq \Phi^{2}U$, we have that $P_c \backslash \Phi^{2}U \subseteq C_u$, hence, by \eqref{U}, if $\Im(z)> \overline{y_1}\Phi^{2}$ then $z\in C_u$. Similarly it can be seen that if $\Im(z)>\overline{y_1}\Phi^{2n}$, then $z \in C_u\Phi^{2(n-1)}$. Therefore, as $\Phi<1$, for all $z \in P_c\backslash\{0\}$, there is an $n \in \N$ such that $z \in \Phi^{2(n-1)}C_u$. Combining this with \eqref{d3} we get \eqref{d7}.\\

We now show that there exists an $m \in \N$ such that $\Phi^{2m}U\subseteq C_b''$. Let
\begin{equation*}\label{d8}
y'=\inf_{x \in I}\left\{ \Im(h(x))   \right\}.
\end{equation*}
Note that as $h'$ is an embedding we must have $y'>0$. Hence there must be an $m \in \N$ such that $y'> \overline{y_1}\Phi^{2m}$. Thus $h(I) \subset P_c\backslash \Phi^{2m}U$. As $P_c\backslash \Phi^{2m}U$ is unbounded we must have $C_u \subseteq P_c\backslash \Phi^{2m}U$ and hence $\Phi^{2m}U\subseteq C_b''$.\\

To conclude the proof of i), take $y^*=\overline{y_1}\Phi^{2m}$. For any $z \in P_c$, such that $0<\Im(z)<y^*$, by \eqref{U}, as $\Phi^{2m}U\subseteq C_b''$ we have $z \in C_b'\backslash \{0\}$. Hence by \eqref{d7} there must be a $n \in \N$ such that $z \in V_n$.\\

We now prove ii). We show that for all $n\geq 1$ we have
\begin{equation}\label{d9}
V_n \subseteq \Phi^{2(n-1)} U \backslash \Phi^{2(m+n)} U.
\end{equation}
Note that we have
\begin{equation*}
\Phi^{2m} U \subseteq C_b'' \subseteq U,
\end{equation*}
therefore as $C_b=\Phi^2 C_b''$ we get
\begin{equation*}
\Phi^{2(m+1)} U \subseteq C_b \subseteq \Phi^{2}U,
\end{equation*}
hence $C_u \subseteq P_c \backslash \Phi^{2(m+1)} U$ and thus
\begin{equation*}
C_b''\cap C_u \subseteq (P_c \backslash \Phi^{2(m+1)} U) \cap U.
\end{equation*}
Therefore $V_1 \subseteq U \backslash \Phi^{2(m+1)} U$. As $V_n=\Phi^{2n}C_b''\cap \Phi^{2n}C_u$ we get \eqref{d9} as intended.\\

We now show  that for any $n \in \N$ there exist constants $0<b_n<B_n$ such that for all $z \in V_n$ and $k \in \N$ we have \eqref{bbounds}. Let
\begin{equation}\label{bnu}
b_n= \overline{y_1}\Phi^{2(n+m)} \sin(\beta)   ,
\end{equation}
\begin{equation}\label{bno}
B_n=\left(  \left| 1+\overline{y_1}\Phi^{2(n-1)} \cot(\beta)\csc(\beta) \right|^2+ \overline{y_1}^2\Phi^{4(n-1)}\csc^2(\beta) \right)^{\frac{1}{2}}.
\end{equation}
As $\beta<\pi/2$ it is straightforward to check that $0<b_n<B_n$.

We first show that $|F^k(z)|\geq b_n$ for all $k\in \N$. Recall the definition of $\gamma(\mu,\mu')$. For $1\leq  k \leq k(z)  $ we have
\begin{equation}\label{d10}
\Im(F^k(z))=\gamma\Im(z). 
\end{equation}
Let $j \in \{1,...,d\}$ be such that $z \in P_j$, by \eqref{ren82} and \eqref{ren118} we have
$$ \gamma=\dfrac{\sin(\arg(z))}{\sin(\arg(z)-\theta_j)},$$
as $\{\arg(z),\arg(z)-\theta_j\} \subset [\beta, \pi-\beta]$, this shows
\begin{equation}\label{gammasin}
\sin(\beta)\leq\gamma\leq\csc(\beta).
\end{equation}
Combining \eqref{d10} and \eqref{gammasin} we get
$
\min_{k \leq k(z)}\Im(F^k(z)) \geq \sin(\beta)  \Im(z). 
$ 
As $z \in V_n$, by \eqref{U} and \eqref{d9} we have
\begin{equation}\label{d12}
\overline{y_1} \Phi^{2(n+m)}< \Im(z) < \overline{y_1} \Phi^{2(n-1)}.
\end{equation}
Combining the inequalities above we get
\begin{equation*}\label{d13}
|F^k(z)|\geq \min_{k \leq k(z)}\Im(F^k(z))\geq \overline{y_1}\Phi^{2(n+m)} \sin(\beta),
\end{equation*}
hence, by \eqref{bnu} we get that $|F^k(z)|\geq b_n$ for all $k\leq k(z)$. Since $F^k(z)=R(z)\in V_n$ this holds for all $k \in \N$.

We now prove that $|F^k(z)|\leq B_n$ for all $k\in \N$.
If $\Im(F(z))\leq \lambda/(2\cot(\beta))$, then $F(z) \in \mathcal{R}_{\lambda,\beta}$ and by Lemma \ref{cor2.3}, we get that for $k\leq k(z)$
\begin{equation}\label{d14}
|\Re(F^k(z))|\leq |1+ \Im(F^k(z))\cot(\beta)|.
\end{equation} 
If $\Im(F(z))> \lambda/(2\cot(\beta))$, we get
$$|\lambda-\Im(F(z))\cot(\beta)|<|1+\Im(F(z))\cot(\beta)|,$$
and combining this with the definition of $F$ we get that \eqref{d14} holds in this case as well.

By \eqref{gammasin}, \eqref{d10}, \eqref{d12} and noting that $\csc(\beta)>1$, for $0\leq k \leq k(z)$ we have
\begin{equation*}\label{d15}
|\Im(F^k(z))|\leq \csc(\beta) \overline{y_1}\Phi^{2(n-1)}.
\end{equation*}
Combining this with \eqref{d14} we get
\begin{equation*}\label{d16}
|\Re(F^k(z))|\leq |1+  \overline{y_1}\Phi^{2(n-1)}\cot(\beta) \csc(\beta)|.
\end{equation*}
From the two inequalities above we obtain
\begin{equation*}\label{d17}
|F^k(z)|\leq \left(  \left| 1+\overline{y_1}\Phi^{2(n-1)} \cot(\beta)\csc(\beta) \right|^2+ \overline{y_1}^2\Phi^{4(n-1)}\csc^2(\beta) \right)^{\frac{1}{2}}.
\end{equation*}
hence, by \eqref{bno} we get that $|F^k(z)|\leq B_n$ for all $k\leq k(z)$. Since $F^k(z)=R(z)\in V_n$ this holds for all $k \in \N$.\\

Finally we prove iii). Let $h_{n}(x)=  \Phi^{2n} h(x)$, for all $x \in I$. We show that for all $n \in \N$,  $h_{n}$ is an embedding of $f_{a,\pi}$ into $R$.

As $h$ is an embedding it is clear that  $h_{n}:I \rightarrow \Phi^{2n} h(I)$ is a homeomorphism. Since $h(I)\subset U$ we have that $h_n(I)\subset \Phi^{2n}U$, hence by Theorem \ref{renormtheorem} we get
\begin{equation*}\label{d18}
R \circ  h_n(x) = \Phi^{2n} R \circ h(x),
\end{equation*}
for all $x \in I$. Since $h_{n}=  \Phi^{2n} h$ by \eqref{eq0s2} we also have
\begin{equation*}\label{d19}
\Phi^{2n} R \circ h(x)= h_n \circ f_{a,\pi}(x),
\end{equation*}
for all $x \in I$. Combining the identities above we get
\begin{equation*}\label{d20}
R \circ  h_n(x)= h_n \circ f_{a,\pi}(x),
\end{equation*}
for all $x \in I$, and hence  $h_{n}$ is an embedding of $f_{a,\pi}$ into $R$.
\hfill\ensuremath{\square}

\section{Proof of Theorems C and D}

In this section we prove Theorems \ref{infperisl} and \ref{notergodic}. We begin by proving Theorem \ref{thm:i3}, which states that periodic points of a TCE are contained in periodically coded islands formed by unions of invariant circles.

We introduce \textit{reflective interval exchange transformations}, relate them to TCEs and prove Theorem \ref{ren56} which shows that for a family of TCEs for every $n \in \mathbf{N}$ such that $p_n$ belongs to a certain interval $I_{P(\mu_j)}$ there is a horizontal periodic orbit for the TCE. 
The final part of the section contains the proof of Theorems \ref{infperisl} and \ref{notergodic}.\\

%
%
%

We define the itinerary of a point $z \in \mathbb{H}$, under $F$, to be $i(z)=i_0 i_1...$, with
\begin{equation*}\label{conseq5}
i_k=\begin{cases}
0, \ \textrm{if} \ F^k(z) \in P_0,\\
j, \ \textrm{if} \ F^k(z) \in P_j, \ j=1,...,d,\\
d+1, \ \textrm{if} \ F^k(z) \in P_{d+1},\\
\end{cases}
\end{equation*}
for $k \in \mathbb{N}$. Given $\delta>0$, denote by $S_{\delta}(z)$, the circle of radius $\delta$ centered at $z$. Let $m_j'(k)$ be the number of $j$s in the $k$-th first symbols of the itinerary of $p$, for $j=1,...,d$.
In the next theorem we prove that for  $\lambda$ irrational, every periodic orbit that does not fall on the boundary of the partition must have a family of invariant manifolds. These are unions of circles centered on the periodic point parametrized by their radii.

\begin{theorem}\label{thm:i3}
 Let $p \in \mathbb{H} \backslash \bigcup_{j'=0}^{k}F^{-j'}(\partial\mathcal{P})$ be a periodic point of $F$ of period $k$. Assume $\lambda \in \mathbb{R^+} \backslash \mathbb{Q}$ . There exists  $\epsilon>0$ such that for all $0<\delta<\epsilon$ the union $\bigcup_{r=0}^{k-1}S_{\delta}(F^r(p))$ is an invariant set for $F$. The orbit of any $z \in \bigcup_{r=0}^{k-1}S_{\delta}(F^r(p))$ is dense on this set if and only if  $m_1'(k) \theta_1(\alpha,\tau)+...+m_d'(k) \theta_d(\alpha,\tau) \in \pi \cdot \mathbb{R} \backslash \mathbb{Q}$.
\end{theorem}
\begin{proof}
 We begin by showing that the itinerary of $p$ contains at least one symbol in $\{1,...,d\}$.
 Assume by contradiction that $i(p)$ is a periodic sequence of $0$s and $d+1$s. It is clear that
 $$F^k(p)=F^{m_0'(k)+m_{d+1}'(k)}(p)= z + m_{d+1}'(k) \lambda -m_0'(k).$$
 Since $p$ is a periodic point of $F$ of period $k$ we have  $z=F^k(z)=m_{d+1}'(k) \lambda -m_0'(k) +z$. Therefore we get  that $\lambda = m_0'(k) / m_{d+1}'(k) \in \mathbb{Q}$, contradicting the assumption that $\lambda$ is irrational.
 
 Hence we can assume $i_0(p) \in \{1,...,d\}$ without loss of generality, since we can choose to start the periodic orbit at the first iterate that falls in $P_j$ for some $j=1,...,d$. Since $p \in \mathbb{H} \backslash \bigcup_{j'=0}^{k}F^{-j'}(\partial\mathcal{P})$, then $p$ belongs to some open cell $U_k$ in the $k$-th refinement of the partition. Since all points in this cell will share the first $k$ addresses in the itinerary, we have $i_0(p)...i_k(p) = i_0(z)...i_k(z)$ for $z \in U_k$. Therefore $F^k:U_k \rightarrow \mathbb{C}$ is such that
 $$F^k(z) = e^{i\theta'(\alpha,\beta)} z + t'(\alpha,\beta,\lambda,\eta),$$
 for some functions $\theta' : [0,\pi)^2 \rightarrow [0,\pi)$ and $t':[0,\pi)^2 \times \mathbb{R}_+^2 \rightarrow \mathbb{R}$. Since $F^k(p)=p$ we have $$p=\dfrac{t'(\alpha,\beta,\lambda,\eta)}{1-e^{i\theta'(\alpha,\beta)}}.$$
 From this it is easy to check that we can rewrite
 $$F^k(z)=e^{i \theta'}(z-p)+pe^{i\theta'}+t=e^{i \theta'}(z-p)+p,$$
 and we get
 \begin{equation}\label{conseq6}
 |F^k(z)-p|=|e^{i \theta'}(z-p)+p-p|=|z-p|.
 \end{equation}
 This implies that $F^k$ is invariant in the largest circle with center $p$ contained in $U_k$. 
 
 Take $\epsilon >0$ such that $B_{\epsilon}(p) \subseteq U_k$. We now see that for $l=1,...,k-1$ we have $F^l(B_{\epsilon}(p))=B_{\epsilon}(F^l(p))$. 
 
 From \eqref{conseq6} we have $|F^k(z)-p|=|z-p| <\epsilon$ which implies that $F^k(z) \in B_{\epsilon}(p)$. Therefore $F^k(B_{\epsilon}(p)) \subseteq B_{\epsilon}(p)$. This implies that for all $r \in \mathbb{N}$, we have $ F^{r k}(z)\in B_{\epsilon}(p)$, hence we also have for $l=1,...,k-1$ that $i(F^l(z))=i(F^{r k+l}(z))$. Therefore every $z \in B_{\epsilon}(p)$ has the same itinerary of $p$. It follows that $B_{\epsilon}(F^l(p))$ is also an invariant set for $F^{l}$, since we can repeat the above argument for $l=1,...,k-1$ and conclude $F^l(B_{\epsilon}(p))=B_{\epsilon}(F^l(p))$.

 For any $0<\delta < \epsilon$ we know that $z \in S_{\delta}(p)$ if and only if $z =p+\delta e^{i \nu'}$ for some $\nu' \in [0,2\pi)$. Since $F^k(z)=\delta e^{i(\theta'+\nu')}+p$, we have $F^k(S_{\delta}(p))\subseteq S_{\delta}(p)$. Therefore $F^k(S_{\delta}(p))=S_{\delta}(p)$, since the reverse inclusion is clear.
 We can repeat this argument for $l=1,...,k-1$ and conclude that $F^l(S_{\delta}(p))=S_{\delta}(F^l(p))$ is an invariant set for $F^l$. Therefore $\bigcup_{r=0}^{k-1}S_{\delta}(F^r(p))$ is an invariant set for $F$.

 Finally we prove that the orbit of any $z \in \bigcup_{r=0}^{k-1}S_{\delta}(F^r(p))$ is dense on this set if and only if  $m_1'(k) \theta_1(\alpha,\tau)+...+m_d'(k) \theta_d(\alpha,\tau) \in \pi \cdot \mathbb{R} \backslash \mathbb{Q}$. Note that $$\theta'(\alpha,\beta)=m_1'(k) \theta_1(\alpha,\tau)+...+m_d'(k) \theta_d(\alpha,\tau).$$ We also have that $F^k$ acts as a rotation by an angle $\theta'$ in $S_{\delta}(p)$, so the orbit of $F^k$ is dense if and only if $m_1'(k) \theta_1(\alpha,\tau)+...+m_d'(k) \theta_d(\alpha,\tau) \in \pi \cdot \mathbb{R}\backslash\mathbb{Q}$. The statement for $F$ follows by $F^l(S_{\delta}(p))=S_{\delta}(F^l(p))$.
\end{proof}

Recall the definition of interval exchange transformation (IET) in the Introduction.
Given $\alpha \in \R_+^d$, $\tau \in S(d)$, we say an IET $f_{\alpha,\tau}$ is \textit{reflective} if there is a point $x\in I$ such that $f_{\alpha,\tau}(x)=|\alpha|-x$. Where $|\alpha|$ denotes the $\ell_1$ norm of $\alpha$.

Recall, from the Introduction, that $\mathcal{R}(\tau)$ denotes the parameter region of all $\alpha \in \R_+^d$ such that for some $j \in \{1,...,d\}$ we have \eqref{reflective}. The following lemma gives an alternative characterization of this set.

\begin{lemma}\label{reflemma}
 Let $\alpha \in \R_+^d$ and $\tau \in S(d)$. Then $f_{\alpha,\tau}$ is reflective if and only if $\alpha \in \mathcal{R}(\tau)$.
\end{lemma}
\begin{proof}
 Consider the map $\tilde{f}:I\rightarrow I$ such that $\tilde{f}(x)= |\alpha|-f_{\alpha,\tau}(x)$, for $x \in I$. By definition of this property, $f_{\alpha,\tau}$ is reflective if and only if $\tilde{f}$ has a fixed point.
 Note that for all $j \in \{1,...,d\}$ the restriction of $\tilde{f}$ to $I_j$ is an orientation reversing continuous bijection, hence $\tilde{f}$ has a fixed point if and only if there is a $j \in \{1,...,d\}$ such that $\tilde{f}(I_j)\cap I_j \neq \emptyset$. It is simple to see that this condition is satisfied if and only if \eqref{reflective} holds. Thus $f_{\alpha,\tau}$ is reflective if and only if $\alpha \in \mathcal{R}(\tau)$ as desired.
\end{proof}

Recall, from the Introduction, that given $\tau \in S(d)$, $J_{\mathcal{R}}(\tau)$ is the set of all $j\in \{1,...,d\}$ such that \eqref{reflective} holds, for some $\alpha \in \R_+^d$.

Given $\alpha \in \mathbb{A}\cap \mathcal{R}(\tau)$ and $j \in J_{\mathcal{R}}(\tau)$ set
\begin{equation}\label{muS}
\mu_{j}(\alpha,\tau)=\tan\left( \dfrac{\pi+\theta_{j}(\alpha,\tau)}{2}  \right).
\end{equation}
We omit, for simplicity, the arguments of $\mu_{j}(\alpha,\tau)$ when this does not cause ambiguity.
\begin{lemma}\label{r1}
Let $\tau \in S(d)$, $\alpha \in \mathbb{A}\cap \mathcal{R}(\tau)$, $j \in J_{\mathcal{R}}(\tau)$ and $\mu_{j}(\alpha,\tau)$ as in \eqref{muS}. We have $L_S'(-\mu_j)\subseteq P_{j}$ and for all $z \in L_S'(-\mu_j)$ we have $\Im(F(z))=\Im(z)$.
\end{lemma}
\begin{proof}
We begin by showing that there is a $j \in \{1,...,d\}$ and a $\varphi \in W_{j}$ such that
 \begin{equation}\label{conseq12}
 f_{\alpha,\tau}(\varphi-\beta)= \pi-\beta-\varphi,
 \end{equation}
with $\beta$ as in \eqref{beta}. Since $\alpha \in \mathcal{R}(\tau)$ we have that $f_{\alpha,\tau}$ is a reflective IET, hence there is a $j \in \{1,...,d\}$ and a $\varphi' \in I_j$ such that $f_{\alpha,\tau}(\varphi')=|\alpha|-\varphi'$. Since $|\alpha|=\pi-2\beta$, by taking $\varphi=\varphi'+\beta$ we get \eqref{conseq12}.
We show that for $z \in L_S'(\tan(\varphi))$ we have $\Im(F(z))=\Im(z)$.
By the definition of the map $E$ and by \eqref{ren01a}, for $z \in P_c$ we have
\begin{equation}\label{conseq13}
E(z)=|z|\exp \left[i \left( \beta +  f_{\alpha,\tau}(\arg(z)-\beta)  \right)\right].
\end{equation}
In particular for $z \in L_S'(\tan(\varphi))$, by the definition of $F$, \eqref{conseq12} and \eqref{conseq13} we have
\begin{equation*}\label{conseq14}
F(z)=|z|e^{i(\pi-\varphi)}-\eta.
\end{equation*}
From \eqref{conseq13} it follows that $\Im(z)=|z|\sin(\varphi)=\Im(F(z))$.
We now prove that $\tan(\varphi)=-\mu_j$. By comparing the two identities above we get
\begin{equation*}\label{conseq15}
\varphi=\dfrac{\pi-\theta_j(\alpha,\tau)}{2}.
\end{equation*}
Therefore, by \eqref{conseq13} the slope of $L_S''$ is equal to $\tan(\pi-\varphi)$ which coincides with $\mu_j$. Thus $\tan(\varphi)=-\mu_j$, which completes the proof.
\end{proof}

Given $\nu>0$ and $\mu$ such $|\mu|>\nu$, let
\begin{equation*}\label{PS}
P(\mu)=\{z \in P_c: -\dfrac{\Im(z)}{|\mu|}< \Re(z)  <\dfrac{\Im(z)}{|\mu|} \}.
\end{equation*}
Define the interval $I_{P(\mu)}$ as
\begin{equation*}
I_{P(\mu)}=\left\{\begin{array}{ll}
\vspace{0.2cm}
\left(1/D(\mu,\nu),1/C(\mu,\nu)\right), & \mu>\nu,\\
\left(1/C(\mu,\nu),1/D(\mu,\nu)\right), & \mu<-\nu.
\end{array}\right.
\end{equation*}

The following theorem shows that a simple condition for the existence of a horizontal periodic island, as defined in the Introduction, for a TCE. A visual depiction of this can be seen in Figure \ref{fig:islandsattractor}.

\begin{figure}[t]
		\centering
		\includegraphics[width=1\linewidth]{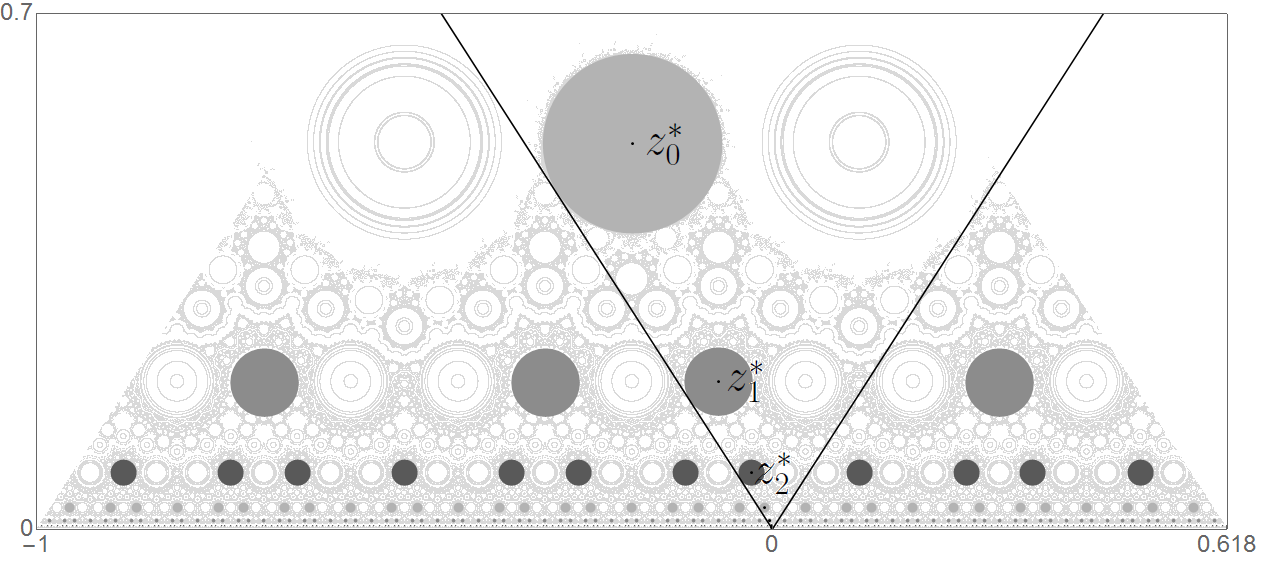}
		\captionsetup{width=1\textwidth}
		\caption{Periodic structures of the TCE with parameters $d=2$, $\alpha=(0.7,\pi-2.7)$, $\tau=(12)$, $\lambda=\Phi$ and $\eta=\Phi^2$. The lines represented are $L_1'$ and $L_2'$ and the differently coloured disks are periodic islands, formed by invariant circles, containing periodic points $z_0^*,z_1^*,...$ . In light grey the first $10^4$ iterates of the orbits of 320 points can be seen.}
		\label{fig:islandsattractor}
\end{figure}

\begin{theorem}\label{ren56}
 Let $\tau \in S(d)$, $\alpha \in \mathbb{A}\cap \mathcal{R}(\tau)$, $j \in J_{\mathcal{R}}(\tau)$ and $\mu_{j}(\alpha,\tau)$ as in \eqref{muS}. For every $n \in  \mathbf{N}$ such that $p_n(\mu_j) \in I_{P(\mu_j)}$, $F$ has a horizontal periodic orbit at height $\hat{y_n}$, for a certain $y_{n+1}(\mu_j)< \hat{y_n}<y_{n}(\mu_j)$. If $L_S'(\mu_j ')\cap \partial \mathcal{P}= \emptyset$, then $F$ has an horizontal periodic island.
\end{theorem}

\begin{proof}
 Since $\tau \in S(d)$ and $\alpha \in \mathbb{A}\cap \mathcal{R}(\tau)$, by Lemma \ref{r1} we have for all $z \in L_S'(-\mu_j)$ that $\Im(F(z))=\Im(z)$.
 Recall \eqref{ren120}. We begin by proving that if for some $n \in \mathbf{N}$ we have $p_n(\mu_j) \in I_{P(\mu_j)}$, then $x_n(\mu_j) + i y_n(\mu_j)\in P(\mu_j)$. By the definition of $\ell$ and from \eqref{ren120} we have
 \begin{equation*}\label{conseq8}
 p_n(\mu_j)=\frac{1}{\ell(y_n(\mu_j))}\left(\frac{y_n(\mu_j)}{\nu}+x_n(\mu_j)\right).
 \end{equation*} 
 From this, we have 
 $
 x_n(\mu_j)+i y_n(\mu_j)\in P(\mu_j),
 $
  if and only if we have
 \begin{equation*}\label{conseq9}
 \dfrac{1}{2}\left(1-\dfrac{\nu}{|\mu|}\right) < p_n(\mu_j) < \dfrac{1}{2}\left(1+\dfrac{\nu}{|\mu|}\right).
 \end{equation*}
 By \eqref{CandD}
 it is direct to see that these inequalities are satisfied if and only if $p_n(\mu_j) \in I_{P(\mu_j)}$.\\
 
 We now prove that if $p_n(\mu_j) \in I_{P(\mu_j)}$, there is an $\hat{y_n}$ satisfying
 \begin{equation}\label{conseq10}
 y_{n+1}(\mu_j)< \hat{y_n}<y_{n}(\mu_j),
 \end{equation}
 such that $\xi_S(\hat{y_n})$ is a horizontal periodic orbit of $F$ at height $\hat{y_n}$.
 
 We split the proof in two cases $\mu_j>\nu$ and $\mu_j<-\nu$, but omit the $\mu_j<-\nu$ case as it is analogous to the other case.
 
 Assume $\mu_j>\nu$. As for $y>0$, $\xi_S(y) \in  L_S'(-\mu_j)$ we have $\Re(\xi_S(y))=-y/\mu_j$, moreover as $p_n(\mu_j) \in I_{P(\mu_j)}$ we have $x_n(\mu_j)+i y_n(\mu_j)\in P(\mu_j)$ and hence $x_n>-y_n/\mu_j$. Since $x_n(\mu_j)=\Re(\rho(y_n^-))$ this shows that
 \begin{equation*}\label{hp1}
 \Re(\xi_S(y_n^-)) < \Re(\rho(y_n^-)).
 \end{equation*}
 As $\mu_j>\nu$ and $p_n(\mu_j) \in I_{P(\mu_j)}$ we have $p_n(\mu_j)<1/C$, hence by Theorem \ref{t5.2} we get that $\rho(y_{n+1})\in L_d'$. As $\xi_S(y_{n+1}) \in \textrm{int}(P_c)$ we get 
 \begin{equation*}\label{hp2}
  \Re(\rho(y_{n+1}))<\Re(\xi_S(y_{n+1})).
 \end{equation*}
 By Theorem \ref{t5.2}, $\rho(y)$ is an affine map for $y_{n+1}\leq y <y_n$ and the map $y \rightarrow \xi_S(y)$ is also affine, in particular both maps are continuous for $y_{n+1}\leq y <y_n$.
 Therefore by the two inequalities above, there must be a $\hat{y_n}$ satisfying \eqref{conseq10} such that
 \begin{equation*}\label{hp3}
 \Re(\rho(\hat{y_n}))=\Re(\xi_S(\hat{y_n})).
 \end{equation*}
 As $\xi_S(\hat{y_n})\in  L_S'(-\mu_j)$, by Lemma \ref{r1} we have that
  \begin{equation*}\label{hp4}
 \Im(\xi_S(\hat{y_n}))=\Im(F(\xi_S(\hat{y_n})))=\hat{y_n}.
 \end{equation*}
 By Theorem \ref{t5.2}, $\Im(\rho(\hat{y_n}))=\hat{y_n}$, hence by the two identities above we get that $\rho(\hat{y_n})=\xi_S(\hat{y_n})$. Thus by the definition of $\rho$, $\xi_S(\hat{y_n})$ is a periodic orbit for $F$.
 Moreover by Lemma \ref{cor2.3} we have that the imaginary part of $\xi_S(\hat{y_n})$ remains constant, and equal to $\hat{y_n}$, throughout its orbit, hence it is an horizontal periodic orbit for $F$.\\
 
 Finally we show that if $L_S'(-\mu_j)\cap \partial \mathcal{P}= \emptyset$, then $F$ has a periodic island that contains this periodic orbit.
 Since $\xi_S(\hat{y_n}) \in L_S'(-\mu_j)$ and $L_S'(-\mu_j)\cap \partial \mathcal{P}= \emptyset$ we can apply Theorem \ref{thm:i3} which shows that this orbit shadows a periodic island which is formed by the union of infinitely many invariant circles.
\end{proof}
 

We now prove Theorems \ref{infperisl} and \ref{notergodic}.

\subsection{Proof of Theorem \ref{infperisl}}
 We divide the proof in two cases $\tau \in \zeta_{-}(d)$ (resp. $\zeta_{+}(d)$) and prove that there is a non-empty open set $\mathcal{A}_{-} \subseteq \mathbb{A}\cap \mathcal{R}(\tau)$ (resp. $\mathcal{A}_{+}$) such that for all $\alpha \in \mathcal{A}_{-}$ (resp. $\mathcal{A}_{+}$), $F$ has infinitely many horizontal periodic islands accumulating on the origin. Having proved this, taking $\mathcal{A}= \mathcal{A}_{-} \cup \mathcal{A}_{+}$ gives the desired result.
 
 We begin by considering the case $\tau \in \zeta_{-}(d)$. Given $j \in J_{\mathcal{R}}(\tau)$, consider the set
 \begin{equation*}\label{infperisleq1}
 J_{\zeta_{-}}(j,\tau)=\left\{  j'' \in \{1,...,d\}: j < j'' \ \textrm{and} \  \tau(j'')<\tau(j')  \right\}.
 \end{equation*}
 Since $\tau \in \zeta_{-}(d)$, we can take $j'\in J_{\mathcal{R}}(\tau)$ such that $J_{\zeta_{-}}(j',\tau) \neq \emptyset$ and take $j'' \in J_{\zeta_{-}}(j',\tau)$.
 
 Let $\mu_{j'}(\alpha,\tau)$ be as in \eqref{muS}. Consider the set
 $\mathcal{V}_{-}$ of all $\alpha \in \mathbb{A}\cap \mathcal{R}(\tau)$, such that:
 \begin{equation}\label{infperisleq2}
 |\alpha| \notin\left\{\dfrac{2\pi}{n}\right\}_{n\geq 1}, \quad \dfrac{\mu_{j'}(\alpha,\tau)}{\nu(|\alpha|)}<-1 \quad \textrm{and} \quad \dfrac{\mu_{j'}(\alpha,\tau)+\nu(|\alpha|)}{\mu_{j'}(\alpha,\tau)-\nu(|\alpha|)}<\Phi.
 \end{equation}
 We now show that if $|\alpha| \notin\left\{2\pi/n\right\}_{n\geq 1}$, there is a $\delta>0$ such that for $\theta_{j'}(\alpha/|\alpha|,\tau) \in (1-\delta,1)$, we have \eqref{infperisleq2}.
 
 Since the map $r \rightarrow (r+1)/(r-1)$ is continuous for all $r \in \R \backslash\{-1\}$ and zero for $r=-1$, there is an $\epsilon>0$, such that for all $\alpha \in \mathcal{V}_{-}$ such that if:
 \begin{equation}\label{infperisleq3}
  \dfrac{\mu_{j'}(\alpha,\tau)}{\nu(|\alpha|)} \in (-1-\epsilon,-1),
 \end{equation}
 then we have \eqref{infperisleq2}. By \eqref{muS} we have
 \begin{equation*}\label{infperisleq4}
 \dfrac{\mu_{j'}(\alpha,\tau)}{\nu(|\alpha|)} = \tan\left( \dfrac{\pi+\theta_{j'}(\alpha,\tau)}{2}    \right)/\tan\left( \dfrac{\pi-|\alpha|}{2}    \right).
 \end{equation*}
 Using linearity of $\alpha \rightarrow \theta_{j'}(\alpha,\tau)$ and simple trigonometric identities, from the above identity, we get
 \begin{equation*}\label{infperisleq5}
 \dfrac{\mu_{j'}(\alpha,\tau)}{\nu(|\alpha|)} = -\cot\left(|\alpha|\frac{\theta_{j'}(\alpha/|\alpha|,\tau)}{2}\right)\tan\left(\frac{|\alpha|}{2}\right).
 \end{equation*}
 Since $\alpha \rightarrow \theta_{j'}(\alpha/|\alpha|,\tau)$ is independent of $|\alpha|$ and we have $|\alpha| \notin\left\{2\pi/n\right\}_{n\geq 1}$, the map $\theta \rightarrow -\cot(|\alpha|\theta/2)\tan(|\alpha|/2)$ is continuous and therefore there is a $\delta>0$ such that for $\theta_{j'}(\alpha/|\alpha|,\tau) \in (1-\delta,1)$, we have \eqref{infperisleq3} and thus \eqref{infperisleq2}.
 
 We now show that there is a nonempty open set $\mathcal{A}_{-}'\subseteq \mathcal{V}_{-}$. To do this we construct an open set $\mathcal{A}_{-}'$ such that for $\alpha \in \mathcal{A}_{-}'$ we have $\theta_{j'}(\alpha/|\alpha|,\tau) \in (1-\delta,1)$.
 By \eqref{ren01a} and \eqref{reflective}, it suffices to show there is an $\tilde{\alpha} \in \mathcal{V}_{-}$ such that we have
 \begin{equation}\label{infperisleq6}
 \sum_{\tau(k)<\tau(j')}\tilde{\alpha}_k -  \sum_{k<j'}\tilde{\alpha}_k  >|\tilde{\alpha}|(1-\delta),
 \end{equation}
 \begin{equation}\label{infperisleq7}
 \left|\sum_{\tau(k)>\tau(j')}\tilde{\alpha}_k -  \sum_{k<j'}\tilde{\alpha}_k \right| < \tilde{\alpha}_{j'}.
 \end{equation}
 Since the above inequalities are strict, we have that there is a neighbourhood $\mathcal{A}_{-}'\subseteq \mathcal{V}_{-}$ of $\tilde{\alpha}$, such that both inequalities are true for all $\alpha \in \mathcal{A}_{-}'$.
 
 We now prove there is $\tilde{\alpha} \in \mathcal{V}_{-}$ satisfying \eqref{infperisleq6} and \eqref{infperisleq7}.
 Assume first that $d=2$ and take $\tilde{\alpha}$ such that $\tilde{\alpha}_{j'}=|\tilde{\alpha}|\delta /2$ and $\tilde{\alpha}_{j''}=|\tilde{\alpha}|(1-\delta /2)$. Since $j'' \in J_{\zeta_{-}}(j',\tau)$, we have $j < j''$ and  $\tau(j'')<\tau(j')$, we have $j'=1$ and $j''=2$, hence
 \begin{equation*}\label{infperisleq8}
 \sum_{\tau(k)<\tau(j')}\tilde{\alpha}_k -  \sum_{k<j'}\tilde{\alpha}_k  = |\tilde{\alpha}|(1-\delta /2),
 \end{equation*}
 thus \eqref{infperisleq6} holds. We also have
 \begin{equation*}\label{infperisleq9}
 \sum_{\tau(k)>\tau(j')}\tilde{\alpha}_k -  \sum_{k<j'}\tilde{\alpha}_k  =0,
 \end{equation*}
 hence, since $\tilde{\alpha}_{j'}>0$, we get \eqref{infperisleq7} as well. 
 
 Now assume $d>2$ and set $\tilde{\alpha}=(\tilde{\alpha}_j)_{j=1,...,d}$, where
 \begin{equation}\label{infperisleq10}
 \tilde{\alpha}_j=\left\{
 \begin{array}{ll}
 |\tilde{\alpha}|\delta /6, & j=j',\vspace{0.1cm}\\
 |\tilde{\alpha}|(1-\delta /4), & j=j'',\vspace{0.1cm}\\
 \dfrac{|\tilde{\alpha}| \delta}{12(d-2)}, & j\neq j',j''.\\
 \end{array}\right.
 \end{equation}
 We show that \eqref{infperisleq6} is true for $\tilde{\alpha}$. Since $j''\in J_{\zeta_{-}}(j',\tau)$ we have
 \begin{equation*}\label{infperisleq12}
 \sum_{k<j'}\tilde{\alpha}_k + \sum_{\tau(k)\geq\tau(j')}\tilde{\alpha}_k  \leq 2 |\tilde{\alpha}| - 2 \tilde{\alpha}_{j''}.
 \end{equation*}
 By \eqref{infperisleq10} we have $2 |\tilde{\alpha}| - 2 \tilde{\alpha}_{j''}=|\tilde{\alpha}|\delta /2$, hence by the inequality above we have 
 \begin{equation*}\label{infperisleq11}
 \sum_{k<j'}\tilde{\alpha}_k + \sum_{\tau(k)\geq\tau(j')}\tilde{\alpha}_k  <|\tilde{\alpha}|\delta,
 \end{equation*}
 which is equivalent to \eqref{infperisleq6}.
 
 We now show that \eqref{infperisleq7} is true for $\tilde{\alpha}$. Since for $k \in \{j',j''\}$ we have $\tau(k)\leq \tau(j')$ and $k> j'$ we have
 \begin{equation*}\label{infperisleq13}
 \left|\sum_{\tau(k)>\tau(j')}\tilde{\alpha}_k -  \sum_{k<j'}\tilde{\alpha}_k \right| < \sum_{k\neq j',j''}\tilde{\alpha}_k.
 \end{equation*}
 By \eqref{infperisleq10} we have $\tilde{\alpha}_{j'}=|\tilde{\alpha}|\delta /6$ and $\sum_{k\neq j',j''}\tilde{\alpha}_k=\delta/12$ hence by the inequality above we have that  \eqref{infperisleq7} is true for $\tilde{\alpha}$.

 We now prove that for $\alpha \in \mathcal{A}_{-}'$, $F$ has infinitely many horizontal periodic orbits accumulating on the origin.
 By Theorem \ref{ren56} it suffices to show that for infinitely many $n \in \N$ we have $p_n(\mu_{j'}(\alpha,\tau)) \in I_{P(\mu_{j'}(\alpha,\tau))}$.
 Note that we have
 \begin{equation*}\label{infperisleq14}
 \dfrac{D(\mu_{j'}(\alpha,\tau),\nu(|\alpha|))}{C(\mu_{j'}(\alpha,\tau),\nu(|\alpha|))}=\dfrac{\mu_{j'}(\alpha,\tau)+\nu(|\alpha|)}{\mu_{j'}(\alpha,\tau)-\nu(|\alpha|)},
 \end{equation*}
 hence since $\alpha \in \mathcal{A}_{-}'\subseteq \mathcal{V}_{-}$ we have
 \begin{equation}\label{infperisleq15}
 \dfrac{D(\mu_{j'}(\alpha,\tau),\nu(|\alpha|))}{C(\mu_{j'}(\alpha,\tau),\nu(|\alpha|))}<\Phi<1.
 \end{equation}
 Assume first that $-\bar{\mu}<\mu_{j'}<-\nu$, with $\bar{\mu}=\frac{\nu}{\Phi^3}$. Using H{\"o}lder conjugacy of $C$ and $D$ it can be seen that \eqref{infperisleq15} is equivalent to
 \begin{equation*}\label{infperisleq16}
 \dfrac{1}{C(\mu_{j'},\nu)}<1-\dfrac{\Phi}{D(\mu_{j'},\nu)}<\dfrac{1}{D(\mu_{j'},\nu)}.
 \end{equation*}
 By Theorem \ref{rencor} \eqref{pep15}, for all $n\geq 1$ we have that $p_n(\mu_{j'})=1-\Phi/D(\mu_{j'},\nu)$, hence by the inequality above we get $p_n(\mu_{j'}) \in I_{P(\mu_{j'})}$ for infinitely many $n \in \N$.
 Now assume $\mu_{j'}\leq -\bar{\mu}$. It can be seen that \eqref{infperisleq15} is equivalent to:
 \begin{equation*}\label{infperisleq17}
 \dfrac{1}{C(\mu_{j'},\nu)}<\dfrac{1}{C(\mu_{j'},\nu)\Phi}<\dfrac{1}{D(\mu_{j'},\nu)}.
 \end{equation*}
 By Theorem \ref{rencor} \eqref{pep3}, for all even $n \in \N$ we have that $p_n(\mu_{j'})=(C(\mu_{j'},\nu)\Phi)^{-1}$, hence by the inequality above we get $p_n(\mu_{j'}) \in I_{P(\mu_{j'})}$ for infinitely many $n \in \N$.
 
 We now show that there is a non-empty open set $\mathcal{A}_{-} \subseteq \mathbb{A}\cap \mathcal{R}(\tau)$ such that for all $\alpha \in \mathcal{A}_{-}$, $F$ has infinitely many horizontal periodic islands accumulating on the origin.
 By Theorem \ref{ren56} it suffices to show that there is a non-empty open set $\mathcal{A}_{-} \subseteq \mathcal{A}_{-}'$ such that for all $\alpha \in \mathcal{A}_{-}$ we have $L_S(\mu_{j'}'(\alpha,\tau)) \cap \partial \mathcal{P}=\emptyset$.
 
 Consider the sets
 \begin{equation*}\label{infperisleq18}
 \mathcal{H}_k=\left\{  \alpha \in \mathbb{A}: |\alpha|-\theta_{j'}(\alpha,\tau)-2 \sum_{j \leq k}\alpha_j =0    \right\},
 \end{equation*}
 for $k=0,1,...,d$. Note that we have $L_S(\mu_{j'}'(\alpha,\tau)) \cap \partial \mathcal{P}\neq \emptyset$ if for some $k\in \{0,1,...,d\}$ we have
 \begin{equation*}\label{infperisleq19}
 -\mu_{j'}(\alpha,\tau)= \tan\left( \frac{\pi-|\alpha|}{2}+  \sum_{j \leq k}\alpha_j   \right).
 \end{equation*}
 By \eqref{muS} and the two identities above it follows that we have $L_S(\mu_{j'}'(\alpha,\tau)) \cap \partial \mathcal{P}\neq \emptyset$ if and only if $\alpha \in \mathcal{H}_k$ for some $k\in \{0,1,...,d\}$.
 
 Set $\mathcal{A}_{-} = \mathcal{A}_{-}' \backslash \bigcup_{k=0}^{d}\mathcal{H}_k$. Since $\mathcal{H}_k$ are codimension 1 closed subsets of $\mathbb{A}$, we have that $\mathcal{A}_{-}$ is a non-empty open set and since for $\alpha \in \mathcal{A}_{-}$ we have $L_S(\mu_{j'}'(\alpha,\tau)) \cap \partial \mathcal{P}=\emptyset$, $F$ has infinitely many horizontal periodic islands accumulating on the origin.
 
 We now consider the case $\tau \in \zeta_{+}(d)$. This case is mostly analogous to the previous one, so for brevity we will only streamline the proof.
 
 Given $j \in J_{\mathcal{R}}(\tau)$, consider the set
 \begin{equation*}\label{infperisleq20}
 J_{\zeta_{+}}(j,\tau)=\left\{  j'' \in \{1,...,d\}: j > j'' \ \textrm{and} \  \tau(j'')>\tau(j')  \right\}.
 \end{equation*}
 Take $j'\in J_{\mathcal{R}}(\tau)$ such that $J_{\zeta_{+}}(j',\tau) \neq \emptyset$ and take $j'' \in J_{\zeta_{+}}(j',\tau)$.
 
  Consider the set $\mathcal{V}_{+}$, of all $\alpha \in \mathbb{A}\cap \mathcal{R}(\tau)$, such that:
 \begin{equation*}\label{infperisleq21}
 |\alpha| \notin\left\{\dfrac{2\pi}{n}\right\}_{n\geq 1}, \quad \dfrac{\mu_{j'}(\alpha,\tau)}{\nu(|\alpha|)}>1 \quad \textrm{and} \quad \dfrac{\mu_{j'}(\alpha,\tau)-\nu(|\alpha|)}{\mu_{j'}(\alpha,\tau)+\nu(|\alpha|)}<\Phi.
 \end{equation*}
 By a similar argument to the previous case, if $|\alpha| \notin\left\{2\pi/n\right\}_{n\geq 1}$, there is a $\delta>0$ such that for $\theta_{j'}(\alpha/|\alpha|,\tau) \in (-1,-1+\delta)$, the expression above is satisfied.
 
 To find a nonempty open set $\mathcal{A}_{+}'\subseteq \mathcal{V}_{-}$ by \eqref{ren01a} and \eqref{reflective}, it suffices to show there is an $\tilde{\alpha} \in \mathcal{V}_{+}$ such that we have \eqref{infperisleq7} and:
 \begin{equation*}\label{infperisleq22}
 \sum_{\tau(k)<\tau(j')}\tilde{\alpha}_k -  \sum_{k<j'}\tilde{\alpha}_k  <|\tilde{\alpha}|(-1+\delta).
 \end{equation*}
 Indeed it can be seen that both this inequality and \eqref{infperisleq7} hold for the same choice of $\tilde{\alpha}$ of the previous case. 
 
 We prove that for $\alpha \in \mathcal{A}_{+}'$, $F$ has infinitely many horizontal periodic orbits accumulating on the origin.
 By Theorem \ref{ren56} it suffices to show that for infinitely many $n \in \N$ we have $p_n(\mu_{j'}(\alpha,\tau)) \in I_{P(\mu_{j'}(\alpha,\tau))}$.
 Note that we have
 \begin{equation*}\label{infperisleq23}
 \dfrac{C(\mu_{j'}(\alpha,\tau),\nu(|\alpha|))}{D(\mu_{j'}(\alpha,\tau),\nu(|\alpha|))}=\dfrac{\mu_{j'}(\alpha,\tau)-\nu(|\alpha|)}{\mu_{j'}(\alpha,\tau)+\nu(|\alpha|)},
 \end{equation*}
 hence since $\alpha \in \mathcal{A}_{+}'\subseteq \mathcal{V}_{+}$ we have
 \begin{equation}\label{infperisleq24}
 \dfrac{C(\mu_{j'}(\alpha,\tau),\nu(|\alpha|))}{D(\mu_{j'}(\alpha,\tau),\nu(|\alpha|))}<\Phi<1.
 \end{equation}
 Assume first that $\nu<\mu_{j'}<\bar{\mu}$. It can be seen that \eqref{infperisleq24} is equivalent to
 \begin{equation*}\label{infperisleq25}
 \dfrac{1}{D(\mu_{j'},\nu)}<1-\dfrac{1}{D(\mu_{j'},\nu)\Phi}<\dfrac{1}{C(\mu_{j'},\nu)}.
 \end{equation*}
 By Theorem \ref{rencor} \eqref{pep4}, for all odd $n$ we have that $p_n(\mu_{j'})=1-(D(\mu_{j'},\nu)\Phi)^{-1}$, hence by the inequality above we get $p_n(\mu_{j'}) \in I_{P(\mu_{j'})}$ for infinitely many $n \in \N$.
 Now assume $\mu_{j'}\geq \bar{\mu}$. It can be seen that \eqref{infperisleq24} is equivalent to:
 \begin{equation*}\label{infperisleq26}
 \dfrac{1}{D(\mu_{j'},\nu)}<\dfrac{\Phi}{C(\mu_{j'},\nu)}<\dfrac{1}{C(\mu_{j'},\nu)}.
 \end{equation*}
 By Theorem \ref{rencor} \eqref{pep21}, for all $n \in \N$ we have that $p_n(\mu_{j'})=(C(\mu_{j'},\nu)\Phi)^{-1}$, hence by the inequality above we get $p_n(\mu_{j'}) \in I_{P(\mu_{j'})}$ for all $n \in \N$.
 
 Setting $\mathcal{A}_{+} = \mathcal{A}_{+}' \backslash \bigcup_{k=0}^{d}\mathcal{H}_k$, we get that for $\alpha \in \mathcal{A}_{+}$ we have $L_S(\mu_{j'}'(\alpha,\tau)) \cap \partial \mathcal{P}=\emptyset$, hence by Theorem \ref{ren56} $F$ has infinitely many horizontal periodic islands at heights which converge to $0$, hence accumulating on the real line. \hfill\ensuremath{\square}

\subsection{Proof of Theorem \ref{notergodic}}
 Let $U$ be an invariant set for $R$  that contains a neighbourhood of the origin. By Theorem \ref{infperisl} it contains infinitely many periodic islands. Suppose there is a point $z \in U$ with a dense orbit in $U$. Then $\{R^n(z)\}_n$ can get arbitrarily close to a periodic point $z'$, this implies that for some $m \in \mathbb{N}$, $R^m(z)$ is contained in a periodic island. Hence its orbit is contained in a circle thus contradicting the hypothesis that the orbit of $z$ is dense in $U$.\hfill\ensuremath{\square}

\end{document}